\documentclass[reqno,dvipsnames,12pt]{amsart}
\usepackage[utf8]{inputenc}
\usepackage{amsfonts, amsthm, amsmath, amssymb,bm}
\usepackage{amssymb,amscd}
\usepackage{mathtools}
\usepackage[numbers]{natbib}
\usepackage{esint}
\usepackage{mathabx}

\usepackage{xcolor}
\usepackage{enumerate}
\usepackage{cases}
\usepackage{hyperref}
\usepackage[colorinlistoftodos, textwidth=2cm]{todonotes}
\usepackage[normalem]{ulem}

\newcommand{\B}{\mathcal{B}}
\newcommand{\D}{\mathcal{D}}
\newcommand{\M}{\mathcal{M}}
\newcommand{\R}{\mathcal{R}}
\renewcommand{\H}{\mathcal{H}}

\newcommand{\p}{\mathfrak{p}}
\newcommand{\q}{\mathfrak{q}}
\renewcommand{\o}{\mathfrak{o}}

\newcommand{\A}{\mathbb{A}}
\newcommand{\Q}{\mathbb{Q}}
\newcommand{\C}{\mathbb{C}}
\newcommand{\Rr}{\mathbb{R}}
\newcommand{\Z}{\mathbb{Z}}

\newcommand{\bs}{\backslash}
\newcommand{\1}{\mathbf{1}}
\newcommand{\I}{\mathrm{I}}
\newcommand{\J}{\mathrm{J}}
\newcommand{\f}{\mathrm{f}}
\renewcommand{\d}{\,\mathrm{d}}
\newcommand{\Ad}{\mathrm{Ad}}
\newcommand{\diag}{\operatorname{diag}}
\newcommand{\aut}{\operatorname{aut}}
\newcommand{\loc}{\operatorname{loc}}
\newcommand{\Eis}{\operatorname{Eis}}

\DeclareMathOperator{\W}{W} 
\def\*{\times}
\newcommand\IP[1]{\langle #1\rangle} 
\newcommand{\mat}[1]{\begin{pmatrix}#1\end{pmatrix}}
\renewcommand{\tilde}{\widetilde}
\renewcommand{\hat}{\widehat}

\newcommand{\MT}{\operatorname{MT}}
\newcommand{\vol}{\operatorname{Vol}}
\newcommand{\gen}{\mathrm{gen}}
\newcommand{\sbr}[1]{\left[#1\right]}
\newcommand{\pbr}[1]{\left(#1\right)}

\newcommand{\GL}{\mathrm{GL}}

\newcommand{\PGL}{\mathrm{PGL}}
\newcommand{\G}{\mathrm{G}}
\newcommand{\N}{\mathrm{N}}
                
\newtheorem{thm}{Theorem}

\newtheorem{lem}{Lemma}[section]

\newtheorem{prop}[lem]{Proposition}
\newtheorem{cor}[thm]{Corollary}
\newtheorem{rmk}{Remark}[section]

\newcommand{\hide}[1]{}

\title[Moments via an RTF]
{Moments of $L$-functions via a relative trace formula}

\begin{document}

\author{Subhajit Jana}
\address{Queen Mary University of London, Mile End Road, London E14 NS, United Kingdom.}
\email{s.jana@qmul.ac.uk}

\author{Ramon Nunes}
\address{Departamento de Matem\'atica, Universidade Federal do Cear\'a, Campus do Pici, Bloco 914, 60440-900 Fortaleza-CE, Brasil}
\email{ramon@mat.ufc.br}



\begin{abstract}
    We prove an asymptotic formula for the second moment of the $\mathrm{GL}(n)\times\mathrm{GL}(n-1)$ Rankin--Selberg central $L$-values $L(1/2,\Pi\otimes\pi)$, where $\pi$ is a fixed cuspidal representation of $\mathrm{GL}(n-1)$ that is tempered and unramified at every place, while $\Pi$ varies over a family of automorphic representations of $\mathrm{PGL}(n)$ ordered by their (archimedean or non-archimedean) conductors.
    
    As another application of our method, we prove existence of infinitely many cuspidal representations $\Pi$ of $\mathrm{PGL}(n)$ such that $L(1/2,\Pi\otimes\pi_1)$ and $L(1/2,\Pi\otimes\pi_2)$ do not vanish simultaneously where $\pi_1$ and $\pi_2$ are cuspidal representations of $\mathrm{GL}(n-1)$ that are unramified and tempered at every place and have trivial central characters.
\end{abstract}

\maketitle

\tableofcontents


\part{Introduction}

\section{Motivation}

The central values of $L$-functions occupy a pivotal role in the confluence of analytic number theory and automorphic forms. In the light of Langlands' far-reaching conjectures and the famous Generalized Riemann Hypothesis, the automorphic representations and their $L$-functions are believed to yield deep arithmetic facts. Sometimes it is more tractable to understand analytic and statistical properties of the automorphic representations and their $L$-functions in \emph{families}, rather than individually, due to the availability of various trace formul{\ae}.

\subsection{Higher moments over conductor-aspect family}

An interesting and natural question is to find the asymptotic behaviour of the $k$-th moment of the central $L$-values for $\GL(n)$ over a family defined in terms of the sizes of the \emph{conductors} (a measurement of the complexity) of the underlying automorphic representations. Philosophically, as $k$ increases the problem becomes more and more difficult as this leads to better and better estimates of the central $L$-values (towards the Generalized Lindel\"of hypothesis). On the other hand, philosophically, the problem should become easier as $n$ increases as the underlying family size also increases. Therefore, there is a natural \emph{balancing question} -- for which $k$ (in terms of $n$) can one solve the above-mentioned problem? We refer to \cite{blomer2012period, Jana2020RS, JaNu2021reciprocity, nelsonvenkatesh2021orbit, yang2023relative} and the references therein to name a few works on moments of higher degree central $L$-values.

Apart from finding an asymptotic expansion of the higher moment, it is also interesting to wonder what coefficients appear in an asymptotic expansion. The systematic study of such questions was first considered for the Riemann zeta function and Dirichlet $L$-functions. In \cite{conreyetal2005integral} (see also \cite{CFKRS}), very precise conjectures were made regarding a connection between the coefficients in the asymptotic expansion and objects arising from \emph{Random Matrix Theory}. It is natural to expect that moments of $L$-functions over families of $\GL(n)$ should behave similarly.

In this paper, we consider a ``cuspidal variant'' of the $2(n-1)$-th moment of the $\PGL(n)$ central $L$-values over a conductor aspect family. More precisely, we prove an asymptotic expansion of the second moment of $\GL(n-1)\times\PGL(n)$ Rankin--Selberg central $L$-values over a conductor-aspect family, where the automorphic representation of $\GL(n-1)$ is a \emph{fixed cuspidal} one. Moreover, we give an explicit description of the coefficients that appear in the asymptotic expansion.

\subsection{Application to non-vanishing of central $L$-values}

The zeros of $L$-functions, or the absence of them, usually encode important number-theoretic information. For instance, the Prime Number Theorem is equivalent to the non-vanishing of the Riemann zeta function at the line $\Re(s)=1$. At the central point $s=1/2$, the Birch--Swinnerton-Dyer conjecture links the non-vanishing of $\GL(2)$ central $L$-values to the finiteness of the group of rational points on an elliptic curve.

Given an automorphic representation $\pi$ of $\GL(k)$, a natural problem is to show the existence of infinitely many cuspidal representations $\sigma$ of $\GL(n)$ such that the central Rankin--Selberg $L$-value $L(1/2,\pi\otimes\sigma)$ does not vanish. This problem is often studied via a certain moment problem of the $\GL(n)\times\GL(k)$ Rankin--Selberg central $L$-values, as described above, by showing an asymptotic formula for a sum of $L(1/2,\pi\otimes\sigma)$ while $\sigma$ varying over a suitable family.

In this paper, we show that when $k=2(n-1)$ and $\pi$ is the Eisenstein series of $\pi_1\boxplus\pi_2$ where $\pi_j$ are cuspidal representations of $\GL(n-1)$ then there exist infinitely many cuspidal $\sigma$ such that $L(1/2,\pi\otimes\sigma)$ does not vanish. We refer to \cite{Tsuzuki2021Hecke, JaNu2021reciprocity, yang2023relative} for more details on higher rank non-vanishing questions.

\subsection{(Non)-existence of a reciprocity}

As briefly mentioned above, it is believed that the moment problem tends to be easier as the ratio of logarithms of the size of the family and the conductor gets larger. Thus, \emph{heuristically}, the current problem, where the above ratio is $\frac{1}{2}$, is believed to be easier than the one in \cite{Jana2020RS} (also \cite{JaNu2021reciprocity}) where this ratio is $\frac{1}{2}-\frac{1}{2n}$. However, it turns out that the difficulty of the problem in the present paper and that of in \cite{Jana2020RS} are not comparable, as described below.

The arguments in \cite{Jana2020RS} and \cite{JaNu2021reciprocity} crucially rely on a \emph{reciprocity formula} for periods; see \cite[eq. (1.1)]{Jana2020RS} and \cite[\S 1.4]{JaNu2021reciprocity}. The reciprocity feature in \cite{Jana2020RS, JaNu2021reciprocity} transforms the moment problems into certain simpler moment problems, bypassing the necessity of carefully analysing the geometric side of a trace formula. On the other hand, the moment problem in the current work does not immediately feature any reciprocity scenario that we can fruitfully use. So we have to employ the full throttle of a (relative) trace formula. This forces us to understand the orbital integrals that arise on the geometric side, in particular, their regularization and meromorphic properties.

In retrospect, we think that the $2k$-th moment problem of $\PGL(n)$ central $L$-values (equivalently, the second moment of $\GL(k)\times\PGL(n)$ Rankin--Selberg central $L$-values for a fixed $\GL(k)$-cuspidal representation) over the conductor aspect family feature significantly different flavour (and difficulty) depending on whether $k\le n-1$ or $k\ge n$. In the latter case, an entirely period theoretic approach seems viable. In contrast, in the former case, one needs a complete understanding of the geometric side of a trace formula. In that sense, the current work seems to be the hardest one (considering the above heuristics on the conductor size vs.\ family size) among the ones which fall in the former category.

\section*{Acknowledgements}

The first author acknowledges the support of the award EP/Z536611/1 from the UKRI during the revision of the article. The first author thanks Shreyasi Datta for her immense support during various ups and downs while writing this long paper. The second author was supported by Instituto Serrapilheira (grant no. 8277) and Funcap (grant no.\ AJC-0222-00009.01.00/24). This work was initiated while both authors were visiting scholars at the Max-Planck-Institut f\"ur Mathematik (MPIM) in Bonn. We are grateful towards the MPIM for its excellent working environment. Finally, we thank the anonymous referee for their careful reading and numerous helpful suggestions.

\section{Main results}

Let $F$ be a number field and $n\ge 2$ an integer. Let $\pi$ be a cuspidal automorphic representation for $\GL_n(F)$ with, say, the trivial central character.
Our main focus in this paper is to study the moment
\begin{equation*}
    \mathbb{M}(\pi;\H):=\int_{\gen}\frac{|L(1/2,\Pi\otimes\pi)|^2}{\ell(\Pi)}\H(\Pi)\d \Pi,
\end{equation*}
where $\int_{\gen}\bullet\d\Pi$ denotes the integral over the standard (that contributes to the spectral decomposition) \emph{generic} automorphic spectrum for $\PGL_{n+1}(F)$ with respect to a certain automorphic Plancherel measure $\d\Pi$; \emph{cf.} \S \ref{sec:basic-notation}. Here $\ell(\Pi)$ are certain harmonic weights that appear in relative trace formul{\ae} and $\H$ is a certain \emph{weight function}. We are mainly interested in the weights $\H$ that pick up a certain portion of the spectrum for which the \emph{global analytic conductor} $C(\Pi)$ remains \emph{essentially} bounded. This means that $\H$ restricts the ramifications at the non-archimedean places (level aspect) or controls the size of the Langlands parameters at archimedean places (analytic conductor aspect).

\vspace{5mm}

In this context, we now describe our main results. The notations $\mathfrak{C}(\sigma)$ and $C(\sigma)$ denote non-archimedean and archimedean \emph{conductors}, respectively; see \S\ref{sec:conductors} for details. We also refer to \S\ref{sec:basic-notation} for other unexplained notations.

\begin{thm}\label{second-moment-nonarch}
Let $F$ be a number field and $\pi$ be a cuspidal automorphic representation
for $\GL_n(F)$ that is tempered and unramified at every place with trivial
central character. Let $\q$ be an integral ideal of $F$. There is a class of weight
functions $\H_\q:=\H_{\f}\H_\infty$, depending on $\pi$, on the isomorphism class of generic automorphic representations for
$\PGL_{n+1}(F)$ with the properties:
\begin{itemize}
    \item if $\mathfrak{C}(\Pi_\f)\nmid \q$ then $\H_\f(\Pi_\f)=0$,
    \item if $\mathfrak{C}(\Pi_\f)\mid\q$ then $\H_\f(\Pi_\f)\gg_\epsilon (1-\epsilon)^{\omega_F(\q)}$ where $\omega_F(\q):=\sum_{v\mid\q}1$,
    \item and $\H_\q(\Pi)\ge 0$ for all generic $\Pi$;
\end{itemize}
such that the for every $L>0$
\begin{equation*}
\mathbb{M}(\pi;\H_\q)=L(1,\pi,\Ad)\frac{\zeta_\q(1)}{\zeta_\q(n+1)}N(\q)^n\left(A\log N(\q)+B\right)+C\log N(\q)+D +O_{\pi,L}\left(N(\q)^{-L}\right)
\end{equation*}
as $N(\q)\to\infty$.

Here $A$ only depends on $F$ and $\H_\infty$ and $B$ only depends on $\pi$, $F$, and $\H_\infty$. Moreover, $C$ is a $\q$-independent linear combination of
\begin{equation}\label{quotients-of-Lfunctions}
    \frac{L_\q(\frac{n+3}{2},\tilde{\pi})} {L_\q(\frac{n+1}{2},\tilde{\pi})},
    \quad \frac{L_\q(\frac{n+3}{2},{\pi})}{L_\q(\frac{n+1}{2},{\pi})}
\end{equation}
and, finally, $D$ is a $\q$-independent linear combination of the terms in \eqref{quotients-of-Lfunctions} and
\begin{equation}\label{derivatives-quotients-Lfunctions}
    \partial_{s=0}\left(\frac{L_\q(\frac{n+3}{2}+\frac{n+1}{2}s, \tilde{\pi})}
    {L_\q(\frac{n+1}{2}+\frac{n-1}{2}s,\tilde{\pi})}\right),\quad
    \partial_{s=0}\left(\frac{L_\q(\frac{n+3}{2}+\frac{n+1}{2}s,\pi)}
    {L_\q(\frac{n+1}{2}+\frac{n-1}{2}s,\pi)}\right).
\end{equation}
However, $C$ and $D$ may depend on $F$ and $\H_\infty$.
\end{thm}

We have an archimedean analogue of Theorem \ref{second-moment-nonarch} for a
\emph{totally real} number field $F$.

\begin{thm}\label{second-moment-arch}
Let $F$ be a totally real number field and $\pi$ be a cuspidal automorphic
representation for $\GL_n(F)$ that is tempered and unramified at every place,
with the trivial central character. Let $\mathfrak{X}:=(X_v)_{v\mid\infty}$
be a tuple of positive real numbers with $X_v\ge 1$ and $X:=\prod_{v\mid\infty}X_v$.
There exists a class of weight functions $\H_{\mathfrak{X}}:=\H_{\f}\H_\infty$, depending on
$\pi$, on the isomorphism class of generic automorphic representations for
$\PGL_{n+1}(F)$ with the properties:
\begin{itemize}
    \item if $\Pi_\f$ is ramified then $\H_\f(\Pi_\f)=0$,
    \item if $C(\Pi_v)\le X_v$ for all $v\mid\infty$ then $\mathcal{H_\infty}(\Pi_\infty)\gg 1$,
    \item and $\H_{\mathfrak{X}}(\Pi)\ge 0$ for all generic $\Pi$;
\end{itemize}
such that, 
\begin{equation*}
    \mathbb{M}(\pi;\H_{\mathfrak{X}})=L(1,\pi,\Ad)A'X^n\log X  +O_\pi\left(X^n\right),
\end{equation*}
where $A'>0$ depends only on $F$.
\end{thm}

\begin{rmk}
    The class of weight functions in Theorems \ref{second-moment-nonarch} and Theorem \ref{second-moment-arch} are described in \S \ref{sec:spectral-side}. This weight functions are constructed from test functions $f$ on $\GL_{n+1}(\A)$ which we require to be of the form as described in \S \ref{type-1-test-function}. Philosophically, these weight functions behave as smoothened characteristic functions of representations have arithmetic conductor dividing $\q$ or analytic conductor bounded by $X_v$, respectively.
\end{rmk}

To depict the strength of Theorem \ref{second-moment-nonarch} and
Theorem \ref{second-moment-arch} we record the following result which we will
obtain simultaneously.

\begin{thm}\label{hybrid-moment}
    Let $F$, $\pi$ and $\mathfrak{X}$ be as in Theorem \ref{second-moment-arch}. Let $\q$ be as in Theorem \ref{second-moment-nonarch}. Then
    $$\sum_{\substack{\Pi\text{ cuspidal}\\\mathfrak{C}(\Pi_\f)\mid\q\\C(\Pi_v)\le X_v,\,v\mid\infty}}\frac{|L(1/2,\Pi\otimes\pi)|^2}{L(1,\Pi,\Ad)}\ll\frac{\zeta_\q(1)}{\zeta_\q(n+1)}(N(\q)X)^n\log(N(\q)X)$$
    as $N(\q)X\to\infty$.
\end{thm}

\begin{rmk}
Note that the size of the family over which the sum is being taken in Theorem \ref{hybrid-moment} has the cardinality
$\asymp \frac{\zeta_\q(1)}{\zeta_\q(n+1)}(N(\q)X)^n$.
Thus Theorem \ref{hybrid-moment} shows a Lindel\"of-consistent upper bound for the second moment of $\GL(n+1)\times\GL(n)$ Rankin--Selberg $L$-values over the conductor-aspect family.
\end{rmk}

Most of the work necessary to study the moment $\mathbb{M}(\pi;\H)$ can be carried
to the study of the so-called mixed moments, namely,
\begin{equation*}
\mathbb{M}(\pi_1,\pi_2;\H):=\int_{\gen}\frac{L(1/2,\Pi\otimes\pi_1)L(1/2,\tilde{\Pi}\otimes\tilde{\pi}_2)}{\ell(\Pi)}\H(\Pi)\d \Pi.
\end{equation*}
where $\pi_1$ and $\pi_2$ are cuspidal automorphic representations for $\GL_n(F)$ with trivial central characters. In fact, we can extract from the proof of Theorem \ref{second-moment-nonarch} and Theorem \ref{second-moment-arch} that for $\pi_1\ncong\pi_2$ and a test function $\H=\H_\f\H_\infty$ where $\H_\f$ is as in Theorem \ref{second-moment-nonarch} and $\H_\infty$ is as in Theorem \ref{second-moment-arch} one has
\begin{equation}\label{mixed-moment-1}
    \mathbb{M}(\pi_1,\pi_2;\H)=\left(C_1 L(1,\pi_1\otimes\tilde{\pi}_2)+C_2L(1,\tilde{\pi}_1\otimes\pi_2)\right)\frac{\zeta_\q(1)}{\zeta_\q(n+1)}(N(\q)X)^n+O\left((N(\q)X)^\epsilon\right)
\end{equation}
for some constants $C_1,C_2$ depending only on $F$, $n$, and the choice of $\H_\infty$, however, independent of $\mathfrak{X}$ and $\q$.

\vspace{5mm}

One of our goals in this paper is to show simultaneous non-vanishing of $L(1/2,\Pi\otimes\pi_1)$ and $L(1/2,\Pi\otimes\pi_2)$. Note that \eqref{mixed-moment-1} would have been a good starting point to prove such non-vanishing \emph{if} we could show the constant in the main term namely, $C_1 L(1,\pi_1\otimes\tilde{\pi}_2)+C_2L(1,\tilde{\pi}_1\otimes\pi_2)\neq 0$. However, we are unable to show that with the above choice of $\H$. In Theorem \ref{mixed-moment-2}, we modify $\H$ in a suitable way and obtain a similar asymptotic formula of $\mathbb{M}(\pi_1,\pi_2,\H)$ as in \eqref{mixed-moment-1}.

\begin{thm}\label{mixed-moment-2}
Let $F$ and $\mathfrak{X}$ be as in Theorem \ref{second-moment-arch}. Let $\pi_1\ncong\pi_2$ be two non-isomorphic cuspidal automorphic representations for $\GL_n(F)$ that are tempered and unramified at every place and have trivial central characters.
Fix a prime $\p_1$ be such that $\pi_{1,\p_1}\ncong\pi_{2,\p_1}$ and let $\p_0\neq \p_1$ be another prime. Then for every $\q$ coprime with $\p_0\p_1$, there exists a class of weight functions $\H:=\H_{\q,\mathfrak{X}}$ such that for every $\eta>0$ there exists a $\delta>0$ depending only on $n$ so that 
\begin{equation*}
\frac{1}{(N(\q)X)^n}\sum_{\substack{\Pi\text{ cuspidal}\\C(\Pi)>(N(\q)X)^{1-\eta}}}\frac{L(1/2,\Pi\otimes\pi_1)L(1/2,\tilde{\Pi}\otimes\tilde{\pi}_2)}{L(1,\Ad,\Pi)}\H(\Pi)
=J_{\mathfrak{X}}\frac{\zeta_\q(1)}{\zeta_\q(n+1)}+O\left(N(\q)X)^{-\delta}\right)
\end{equation*}
for some $J_{\mathfrak{X}}\asymp 1$ as $N(\q)X\to\infty$.
\end{thm}

\begin{rmk}
    The primes $\p_0$ and $\p_1$ are auxiliary and crucial for our proof. We construct the weight function $\mathcal{H}$ with a certain local condition at $\p_0$ so that its support is contained in the class of cuspidal representations. On the other hand, we need another local condition at $\p_1$ to make sure that $J_{\mathfrak{X}}\neq 0$.
\end{rmk}

\begin{rmk}
    We can prove that for the weight function $\H$ that we construct, the \emph{essential} support of the sum in Theorem \ref{mixed-moment-2} is $(N(\Pi)X)^{1-\eta} <C(\Pi)<N(\q)X$. More precisely, we have that $\H_\f$ is supported on $\Pi_\f$ with $\mathfrak{C}(\Pi_\f)\mid\q$ and $\H_{\infty}$ decays polynomially in $(C(\Pi_\infty)/X)$. The class of weight functions in Theorem \ref{mixed-moment-2} are constructed from test functions $f$ on $\GL_{n+1}(\A)$ which we require to be of the form as described in \S \ref{type-2-test-function}.
\end{rmk}

\begin{cor}\label{non-vanishing}
	There are infinitely many irreducible cuspidal representations $\Pi$ such that
	$L(1/2,\Pi\otimes\pi_1)$ and $L(1/2,\Pi\otimes\pi_2)$ do not vanish simultaneously.
\end{cor}

\subsection{Comparison with other work}

During the preparation of this work, we became aware of a similar work by Yang \cite{yang2023relative}. Even though the rough strategies of both papers are similar, there are, however, some differences.
\begin{enumerate}
    \item While studying the geometric side of the relative trace formula, Yang uses a Bruhat decomposition for the double cosets of $\PGL_n(F)$ with respect to its Borel subgroup, while we quotient by a maximal parabolic subgroup (see Proposition \ref{doublecosets-0}). As a consequence, we have a smaller number of cosets to consider and have lesser difficulty in proving the analytic continuation of the orbital integrals.
    \item We have a different treatment of the term arising as residues when analytically continuing the contribution of the continuous spectrum and degenerate terms during regularization of the geometric side. In the non-archimedean conductor aspect, our proof uses direct computation to treat the residue and degenerate terms, and consequently, we have a full asymptotic formula in the non-archimedean case, as in Theorem \ref{second-moment-nonarch}; compare with \cite[Theorem B]{yang2023relative}.
    \item We also produce a novel archimedean analogue as in Theorem \ref{second-moment-arch} of \cite[Theorem B]{yang2023relative} via the theory of \emph{analytic newvectors} due to the first author and Nelson. The archimedean analysis of the residue terms and some of the orbital integrals are significantly more technical than their non-archimedean counterparts.
    \item Finally, \cite[Theorem C]{yang2023relative} (even when $\mathfrak{R}'_j=1$ \emph{loc.\ cit.}) on simultaneous non-vanishing requires that the local components of the fixed representations coincide for at some finite place, whereas Corollary \ref{non-vanishing} on simultaneous non-vanishing does not have any such constraint. 
\end{enumerate}

\vspace{2mm}

We end this section with a sketch of the main argument underlying the proof of the moment estimates.

\subsection{High-level sketch}

The proof uses the theory of integral representations of Rankin--Selberg $L$-functions due to Jacquet--Piatetski-Shapiro--Shalika and the pre-trace formula. This approach generalizes and draws inspiration from the fundamental work of Jacquet \cite{jacquet1986resultat} but under a more quantitative perspective, much like in the work of Ramakrishnan--Rogawski \cite{ramakrishnan2005average}. We offer more details below.
This is in contrast with most of the classical works in moments of $L$-functions which proceed via the approximate functional equation and trace formulae for averaging Whittaker coefficients (the Kuznetsov formula). The latter is a robust
and well-proven strategy that works very well in low rank but as the rank increases, this method becomes combinatorially more challenging to implement.

Our starting point is to reinterpret the moment
\begin{equation*}
  \mathbb{M}(\pi_1,\pi_2;\H)=\int_{\gen}\frac{L(1/2,\Pi\otimes\pi_1)
  L(1/2,\tilde{\Pi}\otimes\tilde{\pi}_2)}{\ell(\Pi)}\H(\Pi)\d \Pi
\end{equation*}
in terms of periods, where $\int_\gen$ means we are integrating over the generic automorphic spectrum. We do so by using the Jacquet--Piatetski-Shapiro--Shalika
integral representation for the above Rankin-Selberg $L-$functions. Let $[\GL(n)]:=
\GL_n(F)\backslash\GL_n(\A)$ and $\B(\Pi)$ denote an orthonormal basis of $\Pi$. Then for fairly general and \emph{nice} weight function $\H$ we can find vectors $\phi_j\in\pi_j$ and $f\in C_c^\infty(\PGL_{n+1}(\A))$ so that
\begin{multline}\label{sketch:Lfunctions-as-periods}
  \mathbb{M}(\pi_1,\pi_2;\H)=\int_{\gen}\sum_{\Phi\in\mathcal{B}(\Pi)}\int_{[\GL(n)]}
  \Pi(f)\Phi\sbr{\mat{x&\\&1}}\phi_1(x)\d x\\
  \int_{[\GL(n)]}\overline{\Phi}\sbr{\mat{y&\\&1}}
  \overline{\phi_2(y)}\d y\d \Pi.
\end{multline}
Our interest lies in weight functions $\H$ which are constructed by taking $f$
so that its $v$-component for $v\mid\q$ is a normalized characteristic function of the Hecke
congruence subgroup $K_{0}(\p^e)$ and for an archimedean place we will have an
approximate version of this, which we call a \emph{normalized majorant}, after
\cite{JaNe2019anv}.

Supposing we can freely exchange the order of integration, we rewrite \eqref{sketch:Lfunctions-as-periods} as
\begin{multline*}
   \mathbb{M}(\pi_1,\pi_2;\H)= \int_{[\GL(n)]}\int_{[\GL(n)]}\phi_1(x) \overline{\phi_2(y)}\\
\left(\int_\gen\sum_{\phi\in\mathcal{B}(\Pi)}\Pi(f)\Phi\sbr{ \mat{x&\\&1}}
\overline{\Phi\sbr{\mat{y&\\&1}}}\d \Pi\right) \d x\d y.
\end{multline*}
We apply the pre-trace formula to the inner integral, thus rewriting the above expression as
\begin{equation*}
  \int_{[\GL(n)]}\int_{[\GL(n)]}  K_f\sbr{\mat{x&\\&1},\mat{y&\\&1}}\phi_1(x)\overline{\phi_2(y)}\d x\d y,
\end{equation*}
where 
\begin{equation*}
  K_f(x,y):=\sum_{\gamma \in \PGL_{n+1}(F)}f(x^{-1}\gamma y).
\end{equation*}
At this point it is helpful to understand the double cosets of
\begin{equation}\label{double-coset-GGG}
\GL_n(F)\backslash \PGL_{n+1}(F)/\GL_n(F)
\end{equation} whose complete description we give in Proposition
\ref{doublecosets-0}.

We have actually oversimplified things. The zeta integrals
$$\int_{[\GL(n)]}\Phi\sbr{\mat{x&\\&1}}\phi(x)\d x$$
do not converge in general for non-cuspidal $\Phi$. To relate the above integral to a central $L$-value we are required to employ
some kind of \emph{regularization} to the above integral. The regularization is a simplified version of the one in \cite{IchinoYamana2016} which was already employed by the authors
in \cite{JaNu2021reciprocity}. The way our regularization works is as follows.
We notice that as $\phi$ is cuspidal it decays rapidly in every direction of a
Siegel domain for $[\PGL(n)]$. As $\Phi$ grows at most polynomially at the cusp 
for convergence purposes we only need to focus on convergence along the integral 
over the central variable of $\GL(n)$. We deal with this problem by introducing
the constant term $\Phi_{U_{n+1}}$ of $\Phi$ along a certain unipotent subgroup $U_{n+1}$
(see \eqref{eq:def-constant-term} for precise definitions) and observing that
for $g\in \GL_{n+1}(\A)$ and $z\in \A^\times$, the expression
$$\Phi\left[\mat{z&\\&1}g\right]-\Phi_{U_{n+1}}\left[\mat{z&\\&1}g\right]$$
decays rapidly as $|z| \rightarrow \infty$ and
has controlled growth as $|z| \rightarrow 0$. As a consequence, one obtains that
the regularized zeta integral
\begin{equation*}
  \int_{[G_n]}\left( \Phi-\Phi_{U_{n+1}} \right)\sbr{\mat{x&\\&1}}\phi(g)
  |\det g|^s\d s 
\end{equation*}
converges absolutely for large $\Re(s)$ and can be related to the $L$-function.

This manoeuvre has two consequences. The first one is that we necessarily start working with sufficiently large $\Re(s)$. Then we pass to the central point via a careful process of analytic continuation of all terms on the spectral side. This process inevitably gives rise to some \emph{residue term} $\R$.

The second consequence is that instead of the kernel
$K_f(x,y)$, we will need to work with the regularized kernel $K^\ast(x,y)$,
given by \eqref{def-K*}. This leads us to study more general double-coset
decompositions similar to \eqref{double-coset-GGG} but with one (or both)
$\GL_n(F)$ replaced the mirabolic subgroup $P_{n+1}(F)$ of $\GL_{n+1}(F)$.

A priori, on the geometric side, there will be infinitely many cosets, each giving rise to an \emph{orbital} integral. We show, using the specific shape of our test function $f$, that all but finitely many (in fact, five) orbital integrals survive. We call them
$I^+$, $I^-$, $I^1$, $I^2$, and $I^\perp$. The sum of \emph{genuine orbital integrals} $I^+$ and $I^-$ gives rise to the \emph{main term} which we asymptotically evaluate.
The term $\D:=I^1+I^2$ gives rise to the \emph{degenerate term}, which is a result of the regularization process. Although $\D$ 
and $\R$ can have singularities we show that $\D+\R$ is holomorphic at our point 
of interest which allow us to evaluate it asymptotically. This process gives rise 
to a \emph{secondary main term}. Finally, the term $I^\perp$ can be bounded by an 
arbitrary negative power of the parameter going to infinity (\emph{i.e.}, $N(\q)$, 
$X$ or $N(\q)X$).

\section{Basic notations}\label{sec:basic-notation}

The letter $F$ will denote either a number field or a local field. In each
section, we will specify the convention for $F$. The letter $\Pi$ (or $\pi$) will usually denote a representation (local or global).

If $F$ is a non-archimedean local field we will denote its ring of integers by
$\o$ and its unique maximal ideal by $\p$. We fix a norm on $F$ and denote it by $|.|$.
We also denote the order of residue field $\o/\p$ by $N(\p)=|\p|^{-1}$.
Finally, we fix an unramified additive character $\psi_F$. That is, one that is trivial on $\o$ but not on $\p^{-1}$.

If $F$ is a number field, we let $\o=\o_F$ be its ring of integers,
$\mathfrak{d}=\mathfrak{d}_F$ be its different ideal and
$\Delta=\Delta_F=[\o_F:\mathfrak{d}_F]$ its discriminant.
If $v\mid \mathfrak{d}_F$, we say that  $v$ is ramified. Otherwise we say that
it is unramified. For each place $v$ we will denote the completion of $F$ at $v$
by $F_v$. In this case, $\o_v$, $\p_v$, $|.|_v$ will denote the corresponding
objects in the local field. We also let $\mathfrak{d}_v:=\p_v^{v(\mathfrak{d})}$
and $\Delta_v=[\o_v:\mathfrak{d}_v]$. If it is clear from the context then we
will drop the subscript $v$ from the notations.
We will denote the adele ring of $F$ by $\A$.

We fix a factorizable additive character $\psi$ on $F\backslash\A$.
The choice of additive character is not of utmost importance but in order to make certain calculations explicit, we choose $\psi=\psi_{\Q}\circ \operatorname{tr}_{A_\Q\backslash\A_F}$ where
$\psi_\Q$ is the standard additive character of $\Q\backslash\A$. This means that $\psi(x)=\prod_v\psi_v(x_v)$, where for all $v$ not dividing $\mathfrak{d}_F$, the character $\psi_v$ is the unramified character $\psi_{F_v}$ of $F_v$. On the other hand, for $v$ dividing the discriminant, $\psi_v(x)=\psi_{F_v}(\lambda_v x)$,
where $\lambda_v$ is a generator of $\mathfrak{d}_v$. In particular, one has $|\lambda_v|=\Delta_v^{-1}$.

More generally, for any factorizable global object $\mathfrak{L}$ (automorphic representations and their $L$-functions etc.) we denote the $v$-th component of $\mathfrak{L}$ by $\mathfrak{L}_v$; \emph{i.e.}, $\mathfrak{L}=\times_{v\le\infty}\mathfrak{L}_v$. For any finite set of places $S$ we denote $\mathfrak{L}_S:=\times_{v\in S}\mathfrak{L}_v$ and $\mathfrak{L}^S:=\times_{v\notin S}\mathfrak{L}_v$, whenever defined. We also abbreviate $\mathfrak{L}_\f:=\prod_{v<\infty}\mathfrak{L}_v$ and $\mathfrak{L}_\infty:=\prod_{v\mid\infty}\mathfrak{L}_v$. For an ideal $\q$ of $\o$ we will also write $\mathfrak{L}_\q$ for $\mathfrak{L}_{\{v\mid\q\}}$. We similarly denote $\mathfrak{L}^\q$, $\mathfrak{L}^\infty$, and $\mathfrak{L}_\infty$.

In this paper, $L$ will always denote the \emph{finite} part of an $L$-function. That is $L=\prod_{v<\infty} L_v$ and $\Lambda$ denotes the complete $L$-function. That is $\Lambda=\prod_vL_v$.

For an algebraic group $H$ defined over $F$, we denote the centre of $H$ by $Z_H$. We abbreviate $\overline{H}:=Z_H\backslash H$ and $[H]:=H(F)\backslash H(\A)$ when $F$ is a number field.

We denote the algebraic group $\GL(m)$ by $G_m$ and its center $Z_{G_m}$ by $Z_m$. We also write $N_m$ and $A_m$ for the standard maximal unipotent subgroup and the maximal diagonal torus of $G_m$, respectively. For each finite $v$ we fix maximal compact subgroups $K_{m,v}:=G_m(\o_v)$ of $G_m(F_v)$ and for archimedean $v$ we let $K_{m,v}$ be either the orthogonal group $\mathrm{O}(m)$ or the unitary group $\mathrm{U}(m)$ according to whether $v$ is real or complex.

We fix Haar measures on the local groups $F_v$ and $F_v^\times$ with the property that at every non-archimedean place $v$, we are assigning measure $\Delta_v^{-1/2}$ to the compact subgroups $\o_v$  and $\o_v^\times$, respectively.
That allows us to define measures on the subgroups $N_m(F_v)$, $A_m(F_v)$ and $Z_m$ via the natural bijections
$$N_m(F_v)\cong F_v^{\frac{m^2-m}{2}},\quad A_m(F_v)\cong (F_v^\times)^m,\quad Z_m(F_v)\cong F_v^\times.$$
We also fix probability Haar measures for all maximal compact subgroups $K_{m,v}$.

For $s\in\C$ and a Haar measure $\d x$ in some group as above, we shorthand $\d_s x:=|\det(x)|^s \d x$.

Also, for a measurable set $X$ with a measure $\mu$ which is clear from the context,
we let $\1_X$ be its characteristic function and $\tilde{\1}_X$ its $L^1$-normalized
characteristic functions, \emph{i.e.},
\begin{equation*}
  \tilde{\1}_X(x)=\mu(X)^{-1}\1_X(x).
\end{equation*}

Finally, let us denote the maximal parabolic subgroup corresponding to the partition $(m-1,1)$ of $m$ by $Q_m$. Its unipotent radical will be denoted $U_m$. Finally, we denote the mirabolic group, consisting of those matrices whose bottom row equals $(0,\ldots,0,1)$, by $P_m$.

\section{Preliminaries and set-up}\label{sec:the-set-up}

In this section, we assume that $F$ is a number field. We fix a finite set of places $S$ of $F$ containing all ramified primes and all archimedean places of $F$.

\subsection{Automorphic forms}

We refer to \cite[\S 4.1]{JaNu2021reciprocity} for a brief discussion of the unitary automorphic spectrum and the \emph{generic} spectrum for $G_m(F)$. If $\sigma$ is a unitary automorphic representation any element in $\sigma$ can be written as $\Eis(f)$ where $f$ is a square-integrable vector induced from a discrete datum $\chi$ attached to a parabolic subgroup $P$ of $G_m$. As described in \cite[\S 4.2]{JaNu2021reciprocity} we define $\|\Eis(f)\|_\sigma:=\|f\|$.

We extend  the additive character $\psi$ from \S \ref{sec:basic-notation}, with the same notation, to an additive character of the group $N_{m}(\A)$ by
$$N_{m}\ni u\mapsto \psi\left(\sum_{i=1}^{m-1} u_{i,i+1}\right).$$
We consider genericity of an automorphic representation $\sigma$ for $G_m(F)$ with respect to the character $\psi$; see \cite[\S 4.3]{JaNu2021reciprocity}. For an automorphic form $\varphi\in\sigma$ we define its Whittaker function by
\begin{equation}\label{def-whittaker}
	W_\varphi(g):=\int_{[N_m]}\varphi(ug)\overline{\psi(u)}\d u.
\end{equation}
The above vanishes identically unless $\sigma$ is generic.

If $\varphi$ is cuspidal then it has a Fourier--Whittaker expansion
\begin{equation}\label{fourier-whittaker-expansion}
	\varphi(g)=\sum_{\gamma\in N_m(F)\bs P_m(F)}W_\varphi(\gamma g)\\
	       =\sum_{\gamma'\in N_{m-1}(F)\bs G_{m-1}(F)}W_\varphi
              \sbr{\mat{\gamma'&\\&1}g},
\end{equation}
where both sums converge absolutely.

Any automorphic representation $\sigma$ decomposes into local representations as $\otimes_v\sigma_v$, where $\sigma_v$ is an irreducible representation of $G_m(F_v)$. Moreover, if $\sigma$ is generic then so is $\sigma_v$. In this case, if $\varphi\in\sigma$ is factorizable then $W_\varphi=\prod_vW_{\varphi,v}$. Similarly, if $\sigma$ is unitary then so is $\sigma_v$. 

For generic unitary $\sigma_v$ we fix a norm on $\sigma_v$ as described in \cite[eq.(6)]{JaNu2021reciprocity}. Namely, for $W_1,W_2\in\sigma_v$ we let
\begin{equation}\label{inner-product-normalization}
\langle W_1,W_2\rangle:=\iota(\sigma_v)\int_{N_{m-1}(F_v)\backslash G_{m-1}(F_v)}W_1\left[\begin{pmatrix}g&\\&1\end{pmatrix}\right]\overline{W_2\left[\begin{pmatrix}g&\\&1\end{pmatrix}\right]}dg,
\end{equation}
where
\begin{equation*}
\iota(\sigma_v):=
\begin{cases}
    \frac{\zeta_v(m)}{L(1,\sigma_v\otimes\widetilde{\sigma}_v)},& \text{ if $v$ is non-archimedean};\\
    1,& \text{ if $v$ is archimedean}.
\end{cases}
\end{equation*}

Finally, we record a relation between the norms of $\varphi$ and its Whittaker functions. From \cite[Lemma 4.1]{JaNu2021reciprocity} we have
\begin{equation}\label{harmonic-weight}
	\|\varphi\|^2=\ell(\sigma)\prod_{v\in S}\|W_{\varphi,v}\|^2,
\end{equation}
where $\ell(\sigma)$ is given by a $\sigma$-independent (but possibly $P$-dependent) constant multiple of the value at $1$ of the adjoint $L$-function of $\sigma$.

\subsection{Hecke algebra}

Let $v$ be either an archimedean or a non-archimedean place. For $m\ge 2$,
we define the Hecke algebra of $G_m(F_v)$, denoted by $\H_v(G_m)$
as the algebra of functions defined
in $G_m(F_v)$ that are smooth, compactly supported and left-and-right-invariant under $K_v$.

Given a smooth representation $\pi$ of $G_m(F_v)$, we have an action of the Hecke
algebra as follows: For $\omega\in \pi$ and $\xi\in\H_v(G_m)$, we let
\begin{equation*}
  \pi(\xi)\omega:=\int_{G_m(F_v)}\xi(g)\pi(g)\omega\d g.
\end{equation*}
It follows from the uniqueness (up to scalars) of the spherical vector that if
$\pi$ is an irreducible unramified representation
and $v_\pi$ is a spherical vector in $\pi$ then $v_\pi$
is an eigenfunction of all operators $\pi(\xi)$ with $\xi\in \H_v(G_m)$.
For any $\xi\in\H_v(G_m)$ we define $\lambda_\pi(\xi)$ to satisfy the equation
\begin{equation*}
\pi(\xi)v_{\pi}=\lambda_\pi(\xi)v_\pi.
\end{equation*}

If $v$ is non-archimedean and $\mu=(\mu_1,\ldots,\mu_m)$ is an $m$-tuple of non-negative integers, we let $\p^\mu:=(\p^{\mu_1},\ldots,\p^{\mu_m})$ and we shorthand $\lambda_\pi(\p^\mu)$ for $\lambda_\pi(\1_{K_v\diag(\p^\mu)K_v})$. It is known that $\lambda_\pi(\p^\mu)$ is a symmetric polynomial depending only on $\mu$, in the Satake parameters of $\pi$; see \cite[eq.(8.14)]{satake1963spherical}.

\subsection{Local factors and conductors}
\label{sec:conductors}

We refer to \cite[Lectures 6 and 8]{cogdell2004lectures} for details.

Let $F$ be a local field. Given any irreducible representation $\pi$ of $G_m(F)$ we can attach local $L$, $\gamma$, and $\epsilon$ factors to $\pi$. One has
$$\gamma(s,\pi)=\epsilon(s,\pi)\frac{L(1-s,\tilde{\pi})}{L(s,\pi)}.$$
We can attach $m$ complex numbers $\{\mu_i\}_{i=1}^m$, called \emph{Langlands parameters}, to $\pi$.

When $F$ is archimedean $L(s,\pi)$ can be written as $\prod_{i=1}^m\Gamma_F(s+\mu_i)$ and $\epsilon(s,\pi)$ is a constant of magnitude one. In this case, we define the \emph{analytic conductor} $C(\pi)$ of $\pi$ by the quantity $\prod_{i=1}^m(1+|\mu_i|)$.

When $F$ is non-archimedean $L(s,\pi)$ can be written as
$\prod_{i=1}^m(1-\mu_iN(\p)^{-s})^{-1}$. If $\pi$ is unramified then $\mu_i\neq 0$.
In this case, $\epsilon(s,\pi)$ can be written as $\epsilon(1/2,\pi) N(\p)^{c(\pi)(1/2-s)}$.
We call $c(\pi)$ to be the \emph{conductor exponent} of $\pi$ and define the
\emph{analytic conductor} as the quantity $C(\pi):=N(\p)^{c(\pi)}$. We also denote
$\p^{c(\pi)}$ by $\mathfrak{C}(\pi)$ and call it the \emph{conductor ideal}.

If $\pi$ is an automorphic representation we define its \emph{global} conductor by
$$C(\pi):=\prod_{v\le\infty}C(\pi_v).$$
Note that $C(\pi_v)=1$ for almost all $v$, so the above product converges.

\subsection{Newvectors}\label{sec:newvectors} 

Let $F$ be a non-archimedean local field. Let $\pi$ be an irreducible generic representation of $\overline{G_m}(F)$. Let $K_0(\p^N)$ be the standard Hecke congruence subgroup of ${G_m}(\o)$. That is, $K_0(\p^N)$ consists of matrices in $\G_m(\o)$ whose last rows are congruent to $(0,\dots,0,\ast)\mod \p^N$. Then there is a unique (up to scalars) vector $W$ in $\pi^{\overline{K_0}(\p^{c(\pi)})}$. This is a \emph{newvector} of $\pi$. Moreover, $c(\pi)$ is the lowest non-negative number $N$ such that $\pi^{\overline{K_0}(\p^{N})}\neq \{0\}$.
We refer to \cite{JPSS1981conducteur} for details.

If the underlying additive character $\psi$ is unramified then we denote the newvector in the $\psi$-Whittaker model of $\pi$, normalized as $W(1)=1$, by $W_\pi$. Miyauchi \cite{Miyauchi2014Whittaker} produced, generalizing previous work by Shintani \cite{Shintani1976explicit}, an explicit formula for $W_\pi$ given by
\begin{equation}\label{shintani}
    W_\pi(\diag(\p^\nu))=
    \begin{cases}
        \delta^{1/2}(\diag(\p^\nu))\chi_\nu(\mu_\pi),\quad &\text{if }\nu\in\Z^m\text{ is dominant},\\
        0,\quad &\text{otherwise}.
    \end{cases}
\end{equation}
Here $\mu_\pi$ are the \emph{Langlands parameters} attached to $\pi$ and $\chi_\nu$ is the Schur polynomial attached to $\nu$; see \cite{Shintani1976explicit} and \cite{Miyauchi2014Whittaker} for more details.

Finally, when $\psi$ is ramified, we need to consider, as in
\cite[proof of Lemma 2.1]{cogdellpiatetski-shapiro1994converse}, a \emph{shifted} newvector associated
to the character $\psi(\cdot)=\psi_F(\lambda\cdot)$, given by
\begin{equation}\label{shift-of-newvector}
W_{\pi}^{\lambda}(g):=
W_\pi\left(\diag((\lambda^{m-j})_{j=1}^{m})g\right),
\end{equation}
where $\lambda\in F^\times$ is a generator of $\mathfrak{d}$.

\subsection{Zeta integrals}

The general theory of integral representations of Rankin--Selberg $L$-function was introduced and developed by Jacquet, Piatetski-Shapiro, and Shalika; see \cite{jacquet1983rankin}. We record the relevant results in this section.

\subsubsection{$\GL(n+1)\times\GL(n)$ zeta integral}\label{zeta-integrals-n-1}
Let $\pi$ and $\pi'$ be generic automorphic representations for $G_{n+1}(F)$ and $G_n(F)$. Let $W$ and $W'$ be global Whittaker vectors in $\pi$ and $\pi'$ associated with certain automorphic forms in the respective spaces and realized in the Whittaker models for $\psi$ and $\overline{\psi}$, respectively. Then for $s\in\C$ with $\Re(s)$ sufficiently large we define the $\GL(n+1)\times\GL(n)$ \emph{global zeta integral} by
$$\Psi(1/2+s,W,W'):=\int_{N_n(\A)\bs G_n(\A)}W\sbr{\mat{g&\\&1}}W'(g)\d_sg.$$
If $W$ and $W'$ are factorizable vectors then the global integral factors into local integrals, namely,
$$\Psi(1/2+s,W,W')=\prod_{v}\Psi_v(1/2+s,W_v,W'_v)$$
where for $\Psi_v(1/2+s,W_v,W_v')$ is defined for sufficiently large $\Re(s)$ by
$$\int_{N_n(F_v)\bs G_n(F_v)}W_v\sbr{\mat{g&\\&1}}W'_v(g)\d_sg.$$
Moreover, if $v$ finite and unramified, $\pi_v$ and $\pi'_v$ are spherical, then one has
\begin{equation}\label{unramified-zeta-integral-m-1}
	\Psi_v(1/2+s,W_{\pi_v},W_{\pi'_v})=L_v(1/2+s,\pi_v\otimes\pi'_v).
\end{equation}
    
Both global and local zeta integrals have meromorphic continuation for all $s\in\C$. In particular, the ratio $\frac{\Psi_v(1/2+s,W_v,W'_v)}{L_v(1/2+s,\pi_v\otimes\pi'_v)}$ is entire in $s$. Moreover, the local zeta integral satisfies the functional equation
\begin{equation}\label{LFE-m-m-1}
\Psi_v(1/2-s,\tilde{W}_v,\tilde{W}'_v)=\omega_{\pi'_v}(-1)^n\gamma_v(1/2+s,\pi_v\otimes\pi'_v)\Psi_v(1/2+s,W_v,W'_v),
\end{equation}
where $\tilde{W}$ is contragredient of $W$ and $\omega_\pi$ is the central character of $\pi$.
We refer to \cite[Lecture 2]{cogdell2007functions} for details.

\subsubsection{$\GL(n)\times\GL(n)$ zeta integral}\label{zeta-integrals-n}
Let $\pi$ and $\pi'$ be generic automorphic representations for $G_n(F)$. Let $W$
and $W'$ be global Whittaker vectors in $\pi$ and $\pi'$ associated with certain
automorphic forms in the respective spaces and realized in the Whittaker models for
$\psi$ and $\overline{\psi}$, respectively. Let $\Phi$ be a Bruhat--Schwartz function
on $\A^n$. Then for $s\in\C$ with $\Re(s)$ sufficiently large we define the
$\GL(n)\times\GL(n)$ \emph{global zeta integral} by
\[
  \Psi(s,W,W',\Phi):=\int_{N_n(\A)\bs G_n(\A)}W(g)W'(g)\Phi(e_ng)\d_sg,
\]
where here and throughout the text, $e_n$ denotes the row vector $(0,\dots,0,1)$.
If $W$ and $W'$ are factorizable vectors and $\Phi$ is a factorizable function
then the global integral factors into local integrals, namely,
$$\Psi(s,W,W',\Phi)=\prod_{v}\Psi_v(s,W_v,W'_v,\Phi_v)$$
where for $\Psi_v(s,W_v,W_v',\Phi_v)$ is defined for sufficiently large $\Re(s)$ by
$$\int_{N_n(F_v)\bs G_n(F_v)}W_v(g)W'_v(g)\Phi_v(e_ng)\d_sg.$$
Moreover, if $v$ finite and unramified, $\pi_v$ and $\pi_v'$ are spherical then one has
\begin{equation}\label{unramified-zeta-integral-m}
    \Psi_v(s,W_{\pi_v},W_{\pi'_v},\mathbf{1}_{\o_v^n})=L_v(s,\pi_v\otimes\pi_v')
\end{equation}
Both global and local zeta integrals have meromorphic continuation for all $s\in\C$. In particular, the global zeta integral satisfies the functional equation
$$\Psi(s,W,W',\Phi)=\Psi(1-s,\tilde{W},\tilde{W}',\hat{\Phi}),$$
where $\hat{\Phi}$ is the Fourier transform of $\Phi$.
The local zeta integral also satisfies the functional equation
\begin{equation}\label{LFE-m-m}
\Psi_v(1/2-s,\tilde{W}_v,\tilde{W}'_v,\hat{\Phi_v})=\omega_{\pi'_v}(-1)^n\gamma_v(1/2+s,\pi_v\otimes\pi'_v)\Psi_v(1/2+s,W_v,W'_v,\Phi_v).
\end{equation}
We refer to \cite[Lecture 2]{cogdell2007functions} for details.

\begin{rmk}\label{rmk-ramified-zeta-function}
If $v$ is ramified then replacing newvectors by shifted newvectors as in \eqref{shift-of-newvector}, we obtain variants of \eqref{unramified-zeta-integral-m-1} and \eqref{unramified-zeta-integral-m} where the right-hand side gets multiplied by a factor of the shape $\Delta_v^{\mu(s)}$ where $\mu$ is an affine function (different in each case)
of $s$ whose coefficients only depend on $n$.
\end{rmk}

\subsection{The Pre-trace Formula}

Let ${f}\in C^{\infty}_c(\overline{G_{n+1}}(\A))$ and consider the operator
$R(f)$ on $C^\infty(\overline{G_{n+1}}(\A))$ defined by setting
\begin{equation}\label{defn-projection-operator}
	(R(f)\Phi)(x):=\int_{\overline{G_{n+1}}(\A)}f(g)\Phi(xg)\d g
  =\int_{\overline{G_{n+1}}(\A)}f(x^{-1}g)\Phi(g)\d g.
\end{equation}
Suppose that $\Phi$ is left invariant by $\overline{G_{n+1}}(F)$. Then by folding-unfolding, we have
\begin{equation*}
	(R(f)\Phi)(x) =\int_{[\overline{G_{n+1}}]}\sum_{\gamma\in\overline{G_{n+1}}(F)}
	{f}(x^{-1}\gamma g)\Phi(g)\d g                                                        =\int_{[\overline{G_{n+1}}]}K_f(x,g)\Phi(g)\d g,
\end{equation*}
where
\begin{equation}\label{def-K-f}
 K_f(x,y):=\sum_{\gamma\in\overline{G_{n+1}}(F)}
	{f}(x^{-1}\gamma y).
\end{equation}
for $x,y\in \overline{G_{n+1}}(\A)$.

Realizing the last expression of $R(f)\Phi(x)$ as an inner product over
$[\overline{G_{n+1}}]$ and $K_f(x,\cdot)$ as an element in $C_c^\infty([\overline{G_{n+1}}])$ we can spectrally decompose $K_f(x,y)$ as follows; see \cite[\S 4.2]{JaNu2021reciprocity} for details. We have
\begin{equation}\label{spectral-decomp}
	K_f(x,y)=\int_{\aut} \sum_{\Phi\in\tilde{\B}(\Pi)}\frac{(R(f)\Phi)(x)\overline{\Phi(y)}}{\|\Phi\|^2} \d\Pi.
\end{equation}
where $\tilde{\B}(\Pi)$ denotes an orthogonal basis of $\Pi$.
It is known that the above integral-sum is absolutely convergent; see \cite[\S 4, Lemma 4.4]{arthur1978trace}. We denote $R(f)\vert_\Pi$ by $\Pi(f)$.

\subsection{Simple regularization}\label{sec:simple-reg}

Let $U_{n+1}$ denotes the unipotent radical of the parabolic in $G_{n+1}$ associated with the partition $n+1 = (n)+(1)$. For an automorphic form $\Phi$ for $G_{n+1}(F)$ we define its constant term along the unipotent subgroup $U_{n+1}$ by
\begin{equation}\label{eq:def-constant-term}
\Phi_{U_{n+1}}(g):=\int_{[U_{n+1}]}\Phi(ug)\d u.
\end{equation}
Let $\phi$ be a cusp form for $G_n(F)$ and $s\in\C$.
We define the \emph{regularized global zeta integral} of $\Phi$ and $\phi$ by
$$\int_{[G_n]}\left(\Phi-\Phi_{U_{n+1}}\right)\sbr{\mat{g&\\&1}}\phi(g)\d_sg.$$
The above integral converges absolutely for $\Re(s)$ sufficiently large; see\footnote{In \cite[Proposition 4.1]{JaNu2021reciprocity} it is assumed that $\Phi$
	is generic. However, it can be checked from the proof that the genericity of $\Phi$
	is only used to show the equality of the regularized zeta integral with the
	period involving Whittaker functions.}
\cite[Proposition 4.1]{JaNu2021reciprocity}. In fact, from the proof of
\cite[Proposition 4.1]{JaNu2021reciprocity} it can be seen that for large $\Re(s)$
we have
\begin{multline*}
	\int_{[G_n]}\left(\Phi-\Phi_{U_{n+1}}\right)\sbr{\mat{g&\\&1}}\phi(g)\d_sg
	\\=\int_{N_n(\A)\backslash G_n(\A)}\left(\int_{[N_{n+1}]}\Phi\sbr{u\mat{g&\\&1}}\overline{\psi(u)}\d u\right)W'_\phi(g)\d_s g.
\end{multline*}
The inner $[N_{n+1}]$-integral vanishes unless $\Phi$ is generic, in which case, we obtain
\begin{equation}\label{reg-zeta-integral}
	\int_{[G_n]}\left(\Phi-\Phi_{U_{n+1}}\right)\sbr{\mat{g&\\&1}}
	\phi(g)\d_sg=\int_{N_n(\A)\backslash G_n(\A)}W_\Phi\sbr{\mat{g&\\&1}}
	W'_\phi(g)\d_s g.
\end{equation}
The right hand side is $\Psi(1/2+s,W_\Phi,W'_\phi)$ as defined in \S\ref{zeta-integrals-n-1}. Here we assume that $W_\Phi$ and $W'_\phi$ are realized in $\psi$ and $\psi^{-1}$ Whittaker models, respectively, of their underlying representations.

\section{Choices of test functions and test vectors}
\label{sec:choice-of-vectors}

In this section, we record the choices for the test functions and test vectors that will be needed for different theorems. Loosely speaking, we will need \emph{type I test functions} for Theorems \ref{second-moment-nonarch} and \ref{second-moment-arch} which will be tailored to have good properties on the spectral side. On the other hand, for Theorem \ref{mixed-moment-2} we will need \emph{type II test functions} aimed to have good properties on the geometric side. For each type of test function, we will also record the choice of the relevant test vectors.

For our purpose, the test function $f$ will always be an element of $C^\infty_c(\overline{G_{n+1}}(\A))$. We also assume $f$ to be factorizable, that is, $f=\otimes_{v\le \infty} f_v$ where $f_v\in C_c^\infty(\overline{G_{n+1}}(F_v))$. Here and elsewhere, we denote
\begin{equation*}
  f_1\ast f_2(g)=\int f_1(h)f_2(g^{-1}h)\d h=\int f_1(gh)f_2(h)\d h
\end{equation*}
where the integrals are over either adelic or local points of a group depending on the context.

We choose the test vectors $\phi_1\in\pi_1$ and $\phi_2\in\pi_2$ via their Whittaker functions $W_i:=W_{\phi_i}$ which we assume to be factorizable: $W_i:=\otimes_{v\le\infty} W_{i,v}$.

\begin{rmk}
  Note that for any local field $F_v$ any $g\in\overline{G_{n+1}}(F_v)$ with
  $g_{n+1,n+1}\neq 0$ has a unique representative of the form
  $$\mat{h&u\\&1}\mat{\I_n&\\v&1},\quad h\in G_n(F_v), u,v^\top\in F_v^n.$$
  Thus to define a function on $\overline{G_{n+1}}(F_v)$ that is supported
  away from the matrices with the condition $g_{n+1,n+1}=0$ it suffices to define the values
  on the matrices of the above shape.
  We will use this fact below a few times without mentioning this justification.
\end{rmk}

Finally, in a few places in this section (and elsewhere in the paper), we will need to choose a parameter $\tau>0$ which we always assume to be sufficiently small but fixed. Moreover, as usual in analytic number theory, the actual values of $\tau$ will be irrelevant and they may change from line to line.

\subsection{Type I test function}\label{type-1-test-function}

Recall $\q$ from Theorem \ref{second-moment-nonarch} and $\mathfrak{X}$
from Theorem \ref{second-moment-arch}. Below we describe type I test functions $f\in C_c^\infty(\overline{G_{n+1}}(\A))$, which are of the form
$$f=\bigotimes_{v\mid\q} f_v 
\bigotimes_{\substack{v<\infty\\v\nmid\q}} f_v\bigotimes_{v\mid\infty} f_v$$
with $f_v\in C_c^\infty (\overline{G_{n+1}}(F_v))$.

In some parts of this work, in particular in this section and in Part \ref{part-local},
when working at a given finite place $v$ we may tacitly assume that the given place
is unramified. The necessary modifications to deal with the ramified places will
only be touched upon via some remarks at the end of the relevant sections

\subsubsection{Uninteresting places: $v<\infty$, $v\nmid \q$}\label{subsub-typeI-uninteresting}

At each such $v$, we let $f_v=f_v^0=f_v^0\ast f_v^0$, where
$${f}_v^0:=\mathbf{1}_{\overline{G_{n+1}}(\o_v)}.$$
Also, we choose $W_{i,v}$ to be the spherical vector in $\pi_{i,v}$ with $W_{i,v}(1)=1$.

\subsubsection{Level places: $v\mid\q$}\label{subsub-typeI-level}

Let us denote $K_0(\p^M)$ to be the standard Hecke congruence
subgroup of $G_{n+1}(\o_\p)$ of level $\p^M$. Let $$\q:=\prod_{v\mid\q}\p_v^{e_v}.$$ At each
such $v$, we let $f_v=f^0_v=f^0_v\ast f^0_v$, where
$$
{f}_v^0:=\tilde{\1}_{\overline{K_0}(\p_v^{e_v})}=\vol\left(\overline{K_0}(\p_v^{e_v})\right)^{-1}\mathbf{1}_{\overline{K_0}(\p_v^{e_v})}.$$
The choice of $W_{i,v}$ is the same as in \S \ref{subsub-typeI-uninteresting}.

\begin{rmk}\label{rmk-typeI-modifications-ramified}
If $v$ is a finite ramified place, independently of whether $v\mid \q$ or not,
we keep the definition of $f_v$ (and $f^0_v$) as in \S \ref{subsub-typeI-uninteresting} or
\S\ref{subsub-typeI-level}. On the other hand we shall modify the choices of
$W_{i,v}$ by taking shifted spherical vectors as in \eqref{shift-of-newvector}.
\end{rmk}

\subsubsection{Archimedean places: $v\mid\infty$}

At each (real) archimedean place $v$, we fix sufficiently small $0<\tau_v'<\tau_v$. For each $X_v\ge 1$ we define 
$f_v:=f_v^0\ast f_v^0$ where
\begin{equation}\label{def-f0-arch}
    f_v^0(g)= X_v^{n}\Omega_{11}(a)\Omega_{12}(b)\Omega_{21}(cX_v),
\end{equation}
whenever
$$\overline{G_{n+1}}(F_v)\ni g=\mat{a&b\\&1}\mat{\I_n&\\c&1}\,\text{ with }\, a\in G_n(F_v),\, b,c^\top\in F_v^n;$$
and $f_v(g)=0$, otherwise.

Here $\Omega_{11}$ is a smooth non-negative function on $G_n(F_v)$ supported on a $\tau_v$-radius ball
around the identity such that  $\Omega_{11}$ is identically $1$ on a $\tau_v'$-radius ball around the identity and $\Omega_{11}|\det|^{-1}$ is $L^1$-normalized. On the other hand, $\Omega_{12}$ and $\Omega_{21}$ are smooth $L^1$-normalized non-negative functions on $F_v^n$ (thought as column/row space) supported on ${\tau}_v$-radii ball around the origin such that they are identically $1$ on a ${\tau}'_v$-radius ball around the origin.

Recalling \cite[Definition 1.2]{JaNe2019anv} we say that $f_v^0$ is a \emph{normalized majorant of $\overline{K_0}(X_v,\tau_v)$} where
we define $K_0(X_v,\tau_v)$, an \emph{archimedean congruence subset of level $X_v$} (see \cite[eq.(1.4)]{JaNe2019anv}), as
\begin{equation*}
\left\{\begin{pmatrix} a & b \\c&d \end{pmatrix}
 \in G_{n+1}(F_v)\ \middle|\ a \in G_n(F_v), d \in G_1(F_v),
    \begin{aligned}
		 & \|a - \I_n\| < \tau_v, \quad |b|<\tau_v, \\
		 & |c|<\frac{\tau_v}{X_v}, \quad |d-1|<\tau_v
    \end{aligned}
	\right\},
\end{equation*}
and $\overline{K_0}(X_v,\tau_v)$ as the image of ${K_0}(X_v,\tau_v)$ inside $\overline{G_{n+1}}(F_v)$.
Elaborating, this implies the following assertions.
\begin{enumerate}
	\item $f^0_v$ is non-negative, supported on $\overline{K_0}(X_v,\tau_v)$, and is
	     $\asymp 1$ on $\overline{K_0}(X_v,\tau_v')$ for some $X_v$-independent
        $0<\tau_v'<\tau_v$.
	\item $\partial_a^\alpha\partial_b^\beta\partial_c^\gamma\partial_d^\delta\,
		      f^0_v\sbr{\mat{a&b\\c&d}}\ll X_v^{n+|\gamma|}$ for any fixed multi-indices
	      $\alpha,\beta,\gamma,\delta$.
          \item $f^0_v$ is $L^1$-normalized.
\end{enumerate}
From \cite[Lemma 8.1]{JaNe2019anv} we see that $f_v$ is also a normalized majorant of $\overline{K_0}(X_v,\tau_v)$.

Finally, we choose $W_{i,v}$ to be a smooth vector in $\pi_{i,v}$ with the following normalization. We fix test functions $0\neq \theta_i\in C_c^\infty(N_n(F_v)\bs G_n(F_v),\psi_v)^{K_{n,v}}$ so that $\|\theta_i\|_{L^2(N_n(F_v)\bs G_n(F_v))}=1$. We then normalize $W_{i,v}$ so that 
for $i=1,2$, one has
\begin{equation}\label{normalize-W-1}
    \int_{N_n(F_v)\bs G_n(F_v)}\theta_i(g)W_{i,v}(g)\d g=1.
\end{equation}

\subsection{Type II test function}\label{type-2-test-function}

Let $\q$ and $\mathfrak{X}$ be as in \S\ref{type-1-test-function}. Further, recall $\p_0$ and $\p_1$ from Theorem \ref{mixed-moment-2}. Below we describe type II test functions $f\in C_c^\infty(\overline{G_{n+1}}(\A))$, which are of the form
$$f=\bigotimes_{v\mid\q} f_v \bigotimes f_{\p_0} \bigotimes f_{\p_1}
\bigotimes_{v\nmid\q\p_0\p_1} f_v\bigotimes_{v\mid\infty} f_v$$
with $f_v\in C_c^\infty (\overline{G_{n+1}}(F_v))$.

\subsubsection{Uninteresting and level places: $\infty>v\neq \p_0,\p_1$}

We choose the same test functions $f_v$ and $W_{i,v}$ as in \S\ref{type-1-test-function}.

\subsubsection{Supercuspidal place: $v=\p_0$}

At $v=\p_0$, let $\sigma$ be a supercuspidal representation of $G_{n+1}(F_v)$. We fix
$$f_v(g):=\langle \sigma(g) W_\sigma, W_\sigma\rangle,$$
the matrix coefficient attached to $W_\sigma$. It is known that matrix coefficients of supercuspidal representations are compactly supported modulo the center.

On the other hand, we choose $W_{i,v}$ to be the spherical vector in $\pi_{i,v}$ with $W_{i,v}(1)=1$.

\subsubsection{Auxiliary place: $v=\p_1$}\label{subsubsec:auxiliary}

Let $\mu\in\Z^n_{\ge 0}$ with $|\mu|:=\sum \mu_i$ and $\nu\in\Z_{\ge 0}$. We assume that $\nu$ is sufficiently larger than $|\mu|$, but fixed. The actual choice of $\p_1$ and $\mu$ will be fixed in Lemma \ref{lem:main-term-type-2}.

At $v=\p_1$, we fix
\begin{equation*}
{f}_v(g):={\mathbf{1}_{K_{v}\diag(\p_1^\mu)K_{v}}(a)}\mathbf{1}_{\o_{v}^n}(b){\mathbf{1}_{\o_{v}^n}(\p_1^{-\nu}c^\top)},
\end{equation*}
whenever $$\overline{G_{n+1}}(F_v)\ni g=\mat{a&b\\&1}\mat{\I_n&\\c&1}\,\text{ with }\, a\in G_n(F_v),\, b,c^\top\in F_v^n;$$
and $f_v(g)=0$, otherwise.

Indeed, $f_v$ is a compactly supported function on $\overline{G_{n+1}}(F_v)$. Note that $f_v$ is supported on
$$\left\{\mat{a+bc&b\\c&1}\mid a\in K_{v}\diag(\p_1^\mu)K_{v}, b\in \o_v^n, c^\top\in \p_1^\nu\o_v^n\right\}\mod Z_{n+1}(F_v).$$
The set on the right-hand side above is a compact set in $G_{n+1}(F_v)$. This is because the matrices vary in a compact set in $\mathrm{Mat}_{n+1}(F_v)$ and determinants of the matrices vary over a compact set in $F_v^\times$.

On the other hand, we choose $W_{i,v}$ to be the spherical vector in $\pi_{i,v}$ with $W_{i,v}(1)=1$.

\begin{rmk}\label{rmk-typeII-modifications-ramified}
As in Remark \ref{rmk-typeI-modifications-ramified}, if $v$ is a finite ramified place,
independently of whether $v=\p_0,\p_1$, $v\mid \q$ or none of the above, 
we keep the choice of $f_v$ and take $W_{i,v}$ to be shifted spherical vectors as in \eqref{shift-of-newvector}.
\end{rmk}
\subsubsection{Archimedean places: $v\mid\infty$}

Our choice here is similar to the archimedean component for the type I test function.

At each (real) archimedean place $v$, we fix sufficiently small $\tau_v>0$. For each $X_v\ge 1$ we define  
$f_v:=f_v^0\ast^{\rm{u}}\alpha_v$ where
$f_v^0$ is as in \eqref{def-f0-arch}. The convolution $\ast^{\rm{u}}$ is defined below.

However, our choices for the functions $\Omega_{11},\,\Omega_{12},\,\Omega_{21}$ here are slightly different.
First, $\Omega_{11}$ is a smooth non-negative function on $G_n(F_v)$ such that $\Omega_{11}|\det|^{-1}$ is $L^1$-normalized, and $\Omega_{11}$ is supported on $K_v\mathcal{T}(\tau_v)K_v$,
where $\mathcal{T}(\tau_v)\subset A_n(F_v)$ denotes the $\tau_v$-radius ball around the identity.
Moreover, we assume that $\Omega_{11}$ can be written as $\Omega^1\ast\Omega^2$ where
$\Omega^i\in C_c^\infty(K_v\bs G_n(F_v)/K_v)$. On the other hand, $\Omega_{12}$ and $\Omega_{21}$
are smooth $L^1$-normalized non-negative functions on $F_v^n$ (thought as column/row space)
supported on ${\tau}_v$-radii ball around the origin. Moreover, we assume that $\Omega_{12}(0)$,
$\Omega_{21}(0)$, and $\lambda_{\pi_j}(\Omega_{11})$ for $j=1,2$ are all non-zero.

On the other hand,
$\alpha_v$ is a non-negative smooth spherical $L^1$-normalized function on
$F_v^n$ supported on a small enough (relative to $\tau_v$) neighborhood of the
origin. We define
\begin{equation*}
  f_v^0\ast^{\rm{u}}\alpha_v(g):=\int_{F_v^n}f_v^0\sbr{g\mat{\I_n&b'\\&1}}\alpha_v(b')\d b'.
\end{equation*}

We can easily see the following properties of $f_v^0$
\begin{enumerate}
	\item $f^0_v$ is non-negative and compactly supported.
	\item $\partial_a^\alpha\partial_b^\beta\partial_c^\gamma\partial_d^\delta\,
		      f^0_v\sbr{\mat{a&b\\c&d}}\ll X_v^{n+|\gamma|}$ for any fixed multi-indices
	      $\alpha,\beta,\gamma,\delta$.
          \item $f^0_v$ is $L^1$-normalized.
\end{enumerate}
We also readily check that $f_v$ also satisfies the above three properties.

Finally, we choose $W_{i,v}$ to be spherical and suitably normalized. In particular, we may choose the normalization in \cite{stade2002archimedean}.

\part{Relative trace formula}

In this part we let $F$ be a number field. Let $f\in C_c^\infty(\overline{G_{n+1}}(\A))$
be factorizable. Although the results in this part will be valid for any such $f$,
keeping the applications in mind we fix $f$ to be either of type $I$ or of type $II$;
see \S\ref{type-1-test-function} and \S\ref{type-2-test-function}. In either case,
we denote $S$ to be a finite set of places such that each $v\notin S$ is finite and
unramified. 

\section{Regularized spectral decomposition}\label{sec:spectral-decomposition}

We define
\begin{equation}\label{def-K*}
	K_f^*(x,y):=K^0_f(x,y)-K^1_f(x,y)-K^2_f(x,y)+K_f^{12}(x,y),
\end{equation}
where $K^0_f:=K_f$ as defined in \eqref{def-K-f}, and
\begin{equation*}
	K_f^1(x,y):=\int_{[U_{n+1}]}K_f(ux,y)\d u,\quad
	K_f^2(x,y):=\int_{[U_{n+1}]}K_f(x,uy)\d u,
\end{equation*}
and
\begin{equation*}
	K_f^{12}(x,y):=\int_{[U_{n+1}]\times[U_{n+1}]}K_f(u_1x,u_2y)\d u_1 \d u_2.
\end{equation*}
Noting that as \eqref{spectral-decomp} converges absolutely we can exchange the
integral-sum there with the compact $[U_{n+1}]$-integral. We obtain
$$K_f^1(x,y)=\int_{\aut} \sum_{\Phi\in\tilde{\B}(\Pi)}
	\frac{(\Pi(f)\Phi)_{U_{n+1}}(x)\overline{\Phi(y)}}{\|\Phi\|^2} \d\Pi.$$
Similar expressions can be obtained for $K_f^2$ and $K_f^{12}$, from what we deduce
\begin{equation}\label{regularized-spectral-decomposition}
	K^*(x,y)=\int_{\aut} \sum_{\Phi\in\tilde{\B}(\Pi)}
	\frac{\pbr{\Pi(f)\Phi-(\Pi(f)\Phi)_{U_{n+1}}}(x)
		\overline{\pbr{\Phi-\Phi_{U_{n+1}}}(y)}}{\|\Phi\|^2} \d\Pi,
\end{equation}
where both the sum and the integral converge absolutely.

Let $\phi_1,\phi_2$ be two cusp forms for $G_n(F)$ and $s_1,s_2\in\C$, with
large enough $\Re(s_i)$. We define
\begin{equation}\label{def-P*}
\mathcal{P}^*(s_1,s_2,\phi_1,\phi_2,f):=\int_{[G_n]\times[G_n]}
	K^*\sbr{\mat{x&\\&1},\mat{y&\\&1}}\phi_1(x)\overline{\phi_2(y)} \d_{s_1}x\d_{s_2}y.
 \end{equation}
We use the expression of $K^*$ in \eqref{regularized-spectral-decomposition}. We parameterize $\int_{\aut}$ in terms of
the discrete data $\chi$,
as described in \cite[\S4.2]{JaNu2021reciprocity}.
The test function property of $f$ ensures rapid decay in $\chi$. A proof of this can be found in \cite[\S 5.1]{finisetal2011spectral}. On the other hand, each summand decays rapidly in the central parts of $x,y\in[G_n]$, as obtained
in \cite[Lemma 4.2]{JaNu2021reciprocity} (with at most polynomial type dependency on $\Phi$, hence $\chi)$. Moreover, the cuspidality of $\phi_i$ ensures that it has rapid decay on $[\overline{G_n}]$. Thus we are allowed to interchange $[G_n]\times[G_n]$-integral with the $\Pi$-integral and $\Phi$-sum. Applying \eqref{reg-zeta-integral} and discussions
surounding it, we get that the above equals
\begin{equation}\label{spec-decomp-period}
	\mathcal{P}^*(s_1,s_2,\phi_1,\phi_2,f)=\int_{\gen} \sum_{\Phi\in\tilde{\B}(\Pi)}
	\frac{\Psi(1/2+s_1,W_{\Pi(f)\Phi},W_{1}){\Psi(1/2+{s_2},
			\overline{W_{\Phi}},\overline{W_{2}})}}{\|\Phi\|^2} \d\Pi,
\end{equation}
with absolute convergence.

\section{The spectral side}
\label{sec:spectral-side}

Applying \eqref{harmonic-weight} we rewrite \eqref{spec-decomp-period} for large $\Re(s_i)$ as
\begin{equation*}
\mathcal{P}^*(s_1,s_2,\phi_1,\phi_2,f)
=\int_{\gen}\frac{L(1/2+s_1,\Pi\otimes\pi_1)L(1/2+s_2,\tilde{\Pi}\otimes\tilde{\pi}_2)}{\ell(\Pi)}\H(\Pi)\d\Pi.
\end{equation*}
Here we define
$$\H(\Pi):=\prod_{v<\infty}H_v(\Pi_v)\prod_{v\mid\infty}h_v(\Pi_v);$$
where the local weight, at any place $v$, is given by
\begin{align*}
H_v(\Pi_v)
&:=H_v(\Pi_v;s_1,s_2,W_{{1,v}},W_{{2,v}},f_v) \\
&:=\sum_{W\in\B(\Pi_v)}\frac{\Psi_v(1/2+s_1,\Pi_v(f_v)W,W_{{1,v}}){\Psi_v(1/2+s_2,\overline{W},\overline{W_{{2,v}}})}}{L_v(1/2+s_1,\Pi_v\otimes{\pi}_{1,v})L_v(1/2+s_2,\tilde{\Pi}_v\otimes\tilde{\pi}_{2,v})},
\end{align*}
and for $v\mid\infty$
\begin{align*}
h_v(\Pi_v)
& :=h_v(\Pi_v;s_1,s_2,W_{{1,v}},W_{{2,v}},f_v) \\
& :=\sum_{W\in \B(\Pi_v)}{\Psi_v(1/2+s_1,\Pi_v(f_v)W,W_{{1,v}}){\Psi_v(1/2+s_2,\overline{W},\overline{W_{{2,v}}})}}.
\end{align*}
Here $\B(\Pi_v)$ denotes an \emph{orthonormal} basis of $\Pi_v$.
Applying (the unramified part of) Lemma \ref{bounds-spectral-nonarchimedean},
we have that when $v$ is unramified $H_v(\Pi_v)=1$ when $\Pi_v$ is unramified
and it vanishes otherwise. Thus, letting
$$\mathcal{H}_S(\Pi_S)=\prod_{\substack{v\in S\\v<\infty}}
H_v(\Pi_v)\prod_{v\mid\infty}h_v(\Pi_v),$$
we can rewrite the above expression for $\mathcal{P}^\ast$ as
\begin{equation}\label{spectral-side}
\mathcal{P}^*(s_1,s_2,\phi_1,\phi_2,f)=
\int_{\gen^S}\frac{L(1/2+s_1,\Pi\otimes\pi_1)L(1/2+s_2,\tilde{\Pi}\otimes\tilde{\pi}_2)}{\ell(\Pi)}\mathcal{H}_S(\Pi_S)\d \Pi,
\end{equation}
where $\gen^S$ denotes the isomorphism class of generic standard automorphic representations that are unramified outside $S$.

Identity \eqref{spectral-side} only works if $\Re(s_1)$ and $\Re(s_2)$ 
are sufficiently large. In the next section, we show that it can be analytically
continued to the central point $s_1=s_2=0$, up to some additional terms.

\section{Meromorphic continuation of the spectral side}
\label{sec:meromorphic-continuation-spectralside}

We rewrite momentarily the spectral side in \eqref{spectral-side} as
$$\mathcal{P}^*=\int_{\gen^S}\frac{\Lambda(1/2+s_1,\Pi\otimes{\pi}_1)\Lambda(1/2+s_2,\tilde{\Pi}\otimes\tilde{\pi}_2)}{\ell(\Pi)}H_S(\Pi_S)\d\Pi$$
where
$H_S(\Pi_S):=\prod_{v\in S}H_v(\Pi_v)$. Here $\Lambda$ denotes the completed $L$-functions and $\Re(s_i)$ are sufficiently large.

From the meromorphic properties of the completed $L$-functions (see
\cite[Theorem 4.2]{cogdell2007functions}) we know that if $\pi$ is cuspidal
then $\Lambda(s,\Pi\otimes\pi)$ is entire unless
$$\Pi\cong\mathcal{I}(\tilde{\pi},s):=\mathcal{I}(M,\tilde{\pi}\cdot|\det|^s\otimes
|\cdot|^{-ns}),\quad s\in i\Rr,$$
where $M\simeq G_n\times G_1$ is a Levi subgroup of $G_{n+1}$, in which it has a
simple pole at $s=1$ with residue $c\Lambda(1,\pi,\Ad)$, for some constant $c$
independent of $\pi$ .

We follow the arguments in \cite[\S 9]{JaNu2021reciprocity} and we require the
following simple lemma:

\begin{lem}\label{analytic-continuation-local-factor}
    The local factor $H_v(\mathcal{I}(\pi,s)_v)$ which is originally only
    defined for $\Re(s)=0$ as before, can be meromorphically continued for all
    $s\in\C$ and is entire in $s_1,s_2$.
\end{lem}

The proof is essentially contained in the proof of \cite[Lemma 9.1]{JaNu2021reciprocity}.
We give a sketch here.

\begin{proof}
    The sum defining $H_v(\Pi_v)$ is absolutely convergent as $f_v$ is compactly supported. Thus it is enough to talk about the meromorphic properties of each summand. Working as in the proof of \cite[Lemma 9.1]{JaNu2021reciprocity} we choose an orthogonal basis of $\mathcal{I}(\pi,s)_v$ consisting of vectors that are entire in $s$ and have norms equal to a holomorphic multiple of $L_v(1, \mathcal{I}(\pi,s)_v\otimes\mathcal{I}(\tilde{\pi},-s)_v)^{-1}$. Entireness in $s_1,s_2$ follows from the properties of zeta integrals; see \S\ref{zeta-integrals-n-1}.
\end{proof}

We now meromorphically continue $\mathcal{P}^*$ to $-\epsilon<\Re(s_i)<1/2$.

First, for $\pi_1\ncong\pi_2$ we can write $\mathcal{P}^*=\mathcal{P}^*_1+
	\mathcal{P}^*_2+\mathcal{P}^*_{\neq}$ where $\mathcal{P}^*_i$, for $i=1,2$, is given by
\begin{equation}\label{defn-P1}
	\int_{\Re(s)=0}
	\Lambda(1/2+s_1,\mathcal{I}(\tilde{\pi}_i,s)\otimes\pi_1) \Lambda(1/2+s_2,
	\mathcal{I}(\pi_i,-s)\otimes\tilde{\pi}_2) \frac{H_S(\mathcal{I}
		(\tilde{\pi}_i,s)_S)}{\ell(\mathcal{I}(\tilde{\pi}_i,s))}\d s
\end{equation}
and $\mathcal{P}^*_{\neq}:= \mathcal{P}^*-\mathcal{P}^*_1-\mathcal{P}^*_2$ which is
entire in $s_1$ and $s_2$.

Our goal now is to meromorphically continue $\mathcal{P}^*_i$ for $i=1,2$.
We do that exactly as in the proof of \cite[Lemma 9.2]{JaNu2021reciprocity}. For the convenience of the reader we briefly sketch the main steps in that proof: We start with $s_1$ to the right of to the line $\Re(s_1)=1/2$ but sufficiently close to it. Using Lemma \ref{analytic-continuation-local-factor} we deform the path of integration remaining in a zero-free region of $\ell(\dots)$, picking up a pole in the process. Noting that the integral over the deformed path is analytic in a small region to the left of the line $\Re(s_1)=1/2$, we deform the path of integration back to $\Re(s)=0$, picking up no pole this time. This, along with meromorphicity of the integral in \eqref{defn-P1} and the residual term, proves the meromorphic continuation of $\mathcal{P}_1^\ast$ in a neighborhood of the line $\Re(s_1)=1/2$.

More explicitly, we thus obtain that $\mathcal{P}^*_1$ can be meromorphically continued for $-\epsilon<\Re(s_1)<1/2$
and any $s_2\in\C$ and the continuation is given by the sum of the integral in \eqref{defn-P1} and the residue
\begin{equation*}
	r_{n,F}\frac{L(s_1+s_2,\pi_1\otimes\tilde{\pi}_2)L\left(\frac{n+1}{2}+s_2-ns_1,\tilde{\pi}_2\right)}
	{L\left(\frac{n+3}{2}-(n+1)s_1,\tilde{\pi}_1\right)}
  \mathcal{H}_S(\mathcal{I}(\tilde{\pi}_1,1/2-s_1)_S),
\end{equation*}
where $r_{n,F}$ is a constant depending only on $F$.

Analogously, $\mathcal{P}^*_2$ can be meromorphically continued to $-\epsilon<
	\Re(s_2)<1/2$ and any $s_1\in\C$, and the continuation is given by the sum of
the integral \eqref{defn-P1} and the residue
\begin{equation*}
	r_{n,F}\frac{L(s_1+s_2,{\pi}_1\otimes\tilde{\pi}_2)L\left(\frac{n+1}{2}+s_1-ns_2,
  {\pi}_1\right)}{L\left(\frac{n+3}{2}-(n+1)s_2,\pi_2\right)}
  \mathcal{H}_S(\mathcal{I}(\tilde{\pi}_2,s_2-1/2)_S).
\end{equation*}
In, conclusion, as we analytically continue $\mathcal{P}^\ast$ to $-\epsilon<\Re(s_i)<1/2$,
we get the extra term $\mathcal{R}(s_1,s_2)=\mathcal{R}(s_1,s_2,\phi_1,\phi_2,f)$,
where
\begin{multline}\label{residue-term}
   \mathcal{R}(s_1,s_2):=L(s_1+s_2,\pi_1\otimes\tilde{\pi}_2) \left(\frac{L(\frac{n+1}{2}+s_2-ns_1,\tilde{\pi}_2)}{L(\frac{n+3}{2}-(n+1)s_1,
    \tilde{\pi}_1)}\mathcal{H}_S(\mathcal{I}(\tilde{\pi}_1,1/2-s_1)_S)\right.\\
    \left.+\frac{L(\frac{n+1}{2}+s_1-ns_2,{\pi}_1)}{L(\frac{n+3}{2}-(n+1)s_2,
    {\pi}_2)}\mathcal{H}_S(\mathcal{I}(\tilde{\pi}_2,s_2-1/2)_S)\right)
\end{multline}

On the other hand, if $\pi_1\cong{\pi}_2=\pi$ we can write
$\mathcal{P}^*=\mathcal{P}^*_\pi+\mathcal{P}^*_{\neq}$ where $\mathcal{P}^*_\pi$ is defined as
\begin{equation}\label{defn-P}
	\int_{\Re(s)=0}
	\Lambda(1/2+s_1,\mathcal{I}(\tilde{\pi},s)\otimes\pi) \Lambda(1/2+s_2,
	\mathcal{I}(\pi,-s)\otimes\tilde{\pi}) \frac{H_S(\mathcal{I}
		(\tilde{\pi},s)_S)}{\ell(\mathcal{I}(\tilde{\pi},s))}\d s
\end{equation}
and $\mathcal{P}^*_{\neq}:=\mathcal{P}^*-\mathcal{P}^*_\pi$ which is entire in $s_1$ and $s_2$.

We assume that $s_i$ are in generic positions, in particular, avoiding the hyperplane $s_1+s_2=0$. Proceeding with the same contour-shifting argument and using Lemma \ref{analytic-continuation-local-factor}, as above, we see that
$\mathcal{P}^*_\pi$ admits meromorphic continuation to $-\epsilon<\Re(s_1),\,
	\Re(s_2)<1/2$, where it is given by the sum of the above integral and the 
  term $\mathcal{R}(s_1,s_2)$ in \eqref{residue-term}, specified to $\pi_1\cong\pi_2\cong \pi$.

We summarize the above conclusions in the following proposition.

\begin{prop}\label{main-expression-spectral-side}
The expression of $\mathcal{P}^*$ in \eqref{spectral-side} is valid when $\Re(s_i)$ are sufficiently positive. $\mathcal{P}^*$ has a meromorphic continuation to the region $-\epsilon<\Re(s_i)<1/2$ given by
$$\mathcal{P}^\ast(s_1,s_2)=\M(s_1,s_2)+\R(s_1,s_2)$$
where $\M(s_1,s_2)$ is given by the right-hand side of \eqref{spectral-side}
and $\R(s_1,s_2)$ is as defined in \eqref{residue-term}.
\end{prop}

\section{Preliminary study of the geometric side}
\label{sec:geometric-side}

We now give another interpretation of the period $\mathcal{P}^*(s_1,s_2,\phi_1,\phi_2,f)$,
the so-called \emph{geometric side}.

Notice that according to the decomposition \eqref{def-K*}, for fixed
$s_1,s_2,\phi_1,\phi_2,f$, we have that
\begin{equation}\label{regularizing-P*}
    \mathcal{P}^*=\lim_{T\rightarrow \infty}\left\{ \mathcal{P}^0(T)-\mathcal{P}^1(T)-\mathcal{P}^2(T)+\mathcal{P}^{12}(T)\right\}
\end{equation}
where, for $\eta\in\{0,1,2,12\}$, we put
\begin{equation}\label{def-regularized-P-eta}
    \mathcal{P}^\eta(T):=\int_{[G_n]\times[G_n]}^TK^\eta_f\sbr{\begin{pmatrix}x&\\&1 \end{pmatrix},\begin{pmatrix}y&\\&1 \end{pmatrix}}\phi_1(x)\overline{\phi_2(y)}\d_{s_1}x\d_{s_2}y,
\end{equation}
where the superscript $T$ stands for the fact that we are restricting the integral
to those $x$ and $y$ satisfying $$T^{-1}<|\det x|,|\det y|<T.$$
Note that since $f$ is compactly supported, $K^\eta_f\sbr{\begin{pmatrix}x&\\&1
  \end{pmatrix},\begin{pmatrix}y&\\&1 \end{pmatrix}}$ is at most polynomially
  bounded in $\|x\|$ and $\|y\|$. Thus the cuspidality of $\phi_i$ ensures that
  the integrand in \eqref{def-regularized-P-eta} decays rapidly in
  $[\overline{G_n}]\times[\overline{G_n}]$. Finally, the determinant truncation
  makes the integral \eqref{def-regularized-P-eta} absolutely convergent.
We also notice that the determinant condition above is invariant under the
multiplication both by $G_n(F)$ and $N_n(\A)$.

In order to study each $\mathcal{P}^\eta$, we will the decompose the sum in the definition
of $K_f$ into orbits of the left-and-right-action of $G_n(F)$ and/or the mirabolic
subgroup $P_{n+1}(F)$. Our first step is identifying the orbits of such actions.
Fortunately we are able to have a rather explicit description of such orbits.
In particular in the case of the bilateral action of $G_n$, our result is
quite similar to the description in \cite[\S (1.3)]{jacquet1986resultat} of the
case $n=1$ with the exception that we find one extra class, that of the element
$\xi^\perp$ (see \eqref{coset-reps} below).

\subsection{The orbits}

In order to describe the orbits of the bilateral action of $G_n(F)$ on
$\overline{G_{n+1}}(F)$, we define the following matrices.
\begin{align}\label{coset-reps}
	&e:=\I_{n+1},\quad n^+:=\begin{pmatrix} \I_n&e_n^\top\\&1 \end{pmatrix},\quad
	n^-:=\begin{pmatrix} \I_n&\\e_n&1 \end{pmatrix},\nonumber \\
	&\xi(t):=\begin{pmatrix} \I_n&te_n^\top\\e_n&1 \end{pmatrix},\quad
	\xi^\perp:=\begin{pmatrix} \I_n&e_1^\top\\e_n&1 \end{pmatrix},\quad
	w':=\begin{pmatrix} \I_{n-1}&&\\&&1\\ &1&\end{pmatrix}.
\end{align}
We prove the following coset decomposition.

\begin{prop}\label{doublecosets-0}
	Let $\gamma \in \overline{G_{n+1}}(F)$. Then $\gamma \in G_n(F) \gamma_0 G_n(F)$
	for some $\gamma_0\in\mathfrak{X}_0$, where
	\begin{equation*}
		\mathfrak{X}_0:=\left\{e,\;n^+,\;n^-,\;\xi^\perp,\,w',\,n^+w',\,w'n^+\right\}\cup
		\left\{\xi(t)\,\mid\,t\in F,\, t\neq 0,\,1\right\}.
	\end{equation*}
	Moreover, the above cosets are disjoint.
\end{prop}

\begin{proof}
For $\gamma_1,\gamma_2\in \overline{G_{n+1}}(F)$, we write that $\gamma_1\equiv
\gamma_2$ whenever they belong to the same coset in $G_n(F)\backslash
\overline{G_{n+1}}(F)/G_n(F)$. Let $\gamma\in \overline{G_{n+1}}(F)$ and let
$\begin{pmatrix} a&b\\c&d \end{pmatrix}$ be any lift of $\gamma$ to $G_{n+1}(F)$,
where $a\in \operatorname{Mat}_{n\times n}(F)$, $b\in \operatorname{Mat}_{n\times 1}(F)$,
$c\in \operatorname{Mat}_{1\times n}(F)$, $d\in F$. Let $\operatorname{rk}(a)$
denote the $F$-rank of $a$. Note that since $\gamma\in\overline{G_{n+1}}(F)$ we must
have $\operatorname{rk}(a)\ge n-1$.
 
 We consider case by case.
	\begin{enumerate}
		\item  Suppose $\operatorname{rk}(a)=n$, $d\neq 0$. This is the most
		      interesting case and we divide its study into several sub-cases.

		      \begin{enumerate}
			      \item If $b=0,\,c=0$, it is straightforward to see that $\gamma\equiv e$.
			      \item If $b\neq 0,\,c= 0$, one can immediately see that
			            \begin{equation*}
				            \gamma\equiv \begin{pmatrix} \I_n& b\\ &1  \end{pmatrix}.
			            \end{equation*}
			            By using that $G_n(F)$ acts transitively on $F^n\setminus\{0\}$,
			            we may conjugate the right-hand side by an element of $G_n(F)$
			            to see that $\gamma\equiv n^+$.
			      \item If $b= 0,\,c\neq 0$, arguing completely analogously
              to the previous case, we can show that $\gamma\equiv n^-$.
			      \item \label{14}If $b,c\neq 0$ but $ca^{-1}b=0$, initial
			            bookkeeping shows that one is reduced to
			            \begin{equation*}
				            \gamma= \begin{pmatrix} \I_n& b\\ c&1  \end{pmatrix},
			            \end{equation*}
			            where now $b,c\neq 0$ but $c b=0$. We claim that we can find
			            $y\in G_n(F)$ such that $e_ny=c$ and $y^{-1}e_1^\top=b$.
                  To see that we first choose $y$ such that $e_ny =c$. Thus,
                  $e_nyb=0$, so we may consider $yb\in F^{n-1}\setminus \{0\}$.
                  Thus we can find $y'\in G_n(F)$ so that $\mat{y'&\\&1}yb=e_1^\top$.
                  Moreover, $e_n\mat{y'&\\&1}y=c$ still holds. Now, conjugating
                  by $y^{-1}$ shows that $\gamma\equiv \xi^\perp$.
			      \item Finally, if $b,c\neq 0$ and $ca^{-1}b\neq 0$,
			            we argue in a similar fashion to the previous case. However,
                  since now $cb\neq 0$, we may find $y\in G_n(F)$ such that $e_ny=c$ and
			            $yb=te_n^\top$. This shows that $\gamma\equiv \xi(t)$. Furthermore,
			            it is easy to see that $\xi(t)\nequiv\xi(t')$ for $t\neq t'$.
			            And finally, we have that $t\neq 1$, since $\det \begin{pmatrix}
			\I_n&te_n^\top\\e_n&1 \end{pmatrix}=1-t$.
		      \end{enumerate}
		\item Suppose $\operatorname{rk}(a)=n$, $d=0$. In this case, we automatically
		      have that $b\neq 0$ and $c\neq 0$. Thus, arguing as in (1.4) and (1.5), we
		      are reduced to the case where
		      \begin{equation*}
			      \gamma=\begin{pmatrix} \I_n&b\\e_n&0 \end{pmatrix}.
		      \end{equation*}
		      Since the determinant of the right-hand side equals $e_nb$
		      we have that $e_n b\neq 0$. Thus, conjugation allows us to reduce further
		      to the situation where $b=te_n^\top$. Now note that
          $$t^{-1}\mat{\I_n & te_n^\top\\ e_n& }\mat{t\I_n & \\& 1}=
          \mat{\I_n & e_n^\top\\ e_n& } = n^+w'.$$
        
		\item Suppose $\operatorname{rk}(a)=n-1$, $d=0$
		      Performing row and column operations, we know that
		      there exist $x,\,y\in G_n(F)$ such that
		      \begin{equation*}
			      x^{-1}ay= \tilde{\I}:=\begin{pmatrix} \I_{n-1}&\\&0\end{pmatrix}.
		      \end{equation*}
		      We are then reduced to the case where
		      \begin{equation*}
			      \gamma=\begin{pmatrix} \tilde{\I}&b\\c&0 \end{pmatrix}.
		      \end{equation*}
		      Since no column and no row of $\gamma$ can be zero,
		      we have that $e_nb,ce_n^\top\neq 0$. Now as $e_nb\neq 0$, we can find an element $B$ of the form $\mat{\I_{n-1}&\ast&\\&\ast&\\&&1}$
        such that $B\gamma=\mat{\tilde{\I}& e_n^\top\\c&0}$.
        Arguing similarly, using $C$ of the shape $\mat{\I_{n-1}&&\\\ast&\ast&\\&&1}$
        and multiplying on the right we can show that $\gamma\equiv w'$
		\item Suppose $\operatorname{rk}(a)=n-1$, $d\neq 0$. Proceeding just as in the
		      previous case we can suppose that
		      \begin{equation*}
			      \gamma = \begin{pmatrix} \tilde{\I}&e_n^\top\\e_n&d\end{pmatrix}.
		      \end{equation*}
Now note that
        $$
        d^{-1}\mat{d\I_n&\\&1} \begin{pmatrix} \tilde{\I}&e_n^\top\\e_n&d
        \end{pmatrix}\mat{\I_{n-1}&&\\&d&\\&&1}=w'n^+.
        $$
	\end{enumerate}
	The disjointness of the cosets is immediate from our argument.
\end{proof}

The first three terms are responsible for the main terms. We will show that the
contribution corresponding to the terms $w'$, $w'n^+$, $n^+w'$ and any $\xi(t)$
all vanish. Finally, the contribution of
the term $\xi^\perp$ doesn't vanish but we will show that it is negligible.

As a consequence of our regularization process, we will also need to consider the
double cosets where we replace one or both instances of $G_n(F)$ by the mirabolic
subgroup $P_{n+1}(F)$. The next result is an easy consequence of Proposition
\ref{doublecosets-0}.

\begin{prop}\label{doublecosets-with-unipotent}
Let $\mathfrak{X}_1:=\{e,n^-,w'\}$, $\mathfrak{X}_2:=\{e,n^-,w'\}$
and $\mathfrak{X}_{12}:=\{e,w'\}$. Then we have
\begin{align*}
\overline{G_{n+1}}(F) & =\bigsqcup_{\xi\in\mathfrak{X}_1}P_{n+1}(F)\xi G_n(F)
=\bigsqcup_{\xi\in \mathfrak{X}_2}G_n(F)\xi P_{n+1}(F)\\
& =\bigsqcup_{\xi\in\mathfrak{X}_{12}}P_{n+1}(F)\xi P_{n+1}(F),
\end{align*}
Here $e,n^-,w'$ are as \eqref{coset-reps}.
\end{prop}

\begin{proof}
    The description of $\mathfrak{X}_1$ and $\mathfrak{X}_2$ follow from Proposition
    \ref{doublecosets-0}, the fact that $n^+\in P_{n+1}(F)$, and the identities
    $$\xi(t)=\mat{\I_n-te_n^\top e_n& te_n^\top\\&1}n^-=n^-\mat{(1-t)^{-1}\I_n&(1-t)^{-1} te_n^\top\\&1},$$
    $$\xi^\perp=\mat{\I_n-e_1^\top e_n& e_1^\top\\&1}n^-=n^-\mat{\I_n&e_1^\top\\&1},$$
    and finally,
    $$w'n^+=\mat{\J_{-1}&e_n^\top\\&1}n^-,\quad n^+w'=-n^-\mat{-\J_{-1}&-e_n^\top\\&1},$$
    where $\J_t:=\diag(\I_{n-1},t)$. 
    The last two of these equalities also explain the description of $\mathfrak{X}_{12}$.
\end{proof}

Once equipped with the coset decompositions in Proposition \ref{doublecosets-0}
and Proposition \ref{doublecosets-with-unipotent}, we decompose the terms
$\mathcal{P}^\eta(T)$ as follows.

\begin{equation}\label{decomposition-P-eta}
    \mathcal{P}^\eta(T)=\sum_{\xi\in\mathfrak{X}_{\eta}}I^{\eta}(\xi,T,s_1,s_2;\phi_1,\phi_2,f),
\end{equation}
where, applying \eqref{def-regularized-P-eta}, we have
\begin{equation}\label{I0-definition}
	I^0(\xi,T;s_1,s_2,\phi_1,\phi_2,f):=\int_{[G_n]\times[G_n]}^T
  \sum_{\gamma\in [\xi]_0}{f}\sbr{
	\begin{pmatrix}x^{-1}&\\&1 \end{pmatrix}\gamma
	\begin{pmatrix}y&\\&1 \end{pmatrix}}\phi_1(x)\overline{\phi_2(y)}\d_{s_1} x\d_{s_2} y,
\end{equation}

\begin{multline}\label{I1-definition}
	I^1(\xi,T;s_1,s_2,\phi_1,\phi_2,f):=\int_{[G_n]\times[G_n]}^T
  \int_{[U_{n+1}]}\sum_{\gamma\in [\xi]_1}
	{f}\sbr{\begin{pmatrix}x^{-1}&\\&1 \end{pmatrix}u^{-1}\gamma
	\begin{pmatrix}y&\\&1 \end{pmatrix}}\\ \times\phi_1(x)\overline{\phi_2(y)}\d u\d_{s_1} x\d_{s_2} y,
\end{multline}

\begin{multline}\label{I2-definition}
	I^2(\xi,T;s_1,s_2,\phi_1,\phi_2,f):=\int_{[G_n]\times[G_n]}^T
  \int_{[U_{n+1}]}
  \sum_{\gamma\in [\xi]_2}{f}\sbr{
	\begin{pmatrix}x^{-1}&\\&1 \end{pmatrix}\gamma
	u\begin{pmatrix}y&\\&1 \end{pmatrix}}\\ \times\phi_1(x)\overline{\phi_2(y)}
 \d u\d_{s_1} x\d_{s_2} y,
\end{multline}

\begin{multline}\label{I12-definition}
	I^{12}(\xi,T;s_1,s_2,\phi_1,\phi_2,f):=\int_{[G_n]\times[G_n]}^T
  \int_{[U_{n+1}]}  \int_{[U_{n+1}]}
  \sum_{\gamma\in [\xi]_{12}}{f}\sbr{
	\begin{pmatrix}x^{-1}&\\&1 \end{pmatrix}u_1^{-1}\gamma u_2
	\begin{pmatrix}y&\\&1 \end{pmatrix}}\\ \times\phi_1(x)\overline{\phi_2(y)}\d u_1\d u_2 \d_{s_1} x\d_{s_2} y.
\end{multline}

Finally, for $\eta\in\{0,1,2,12\}$, $[\xi]_\eta$ denotes one of the double cosets
described in Propositions \ref{doublecosets-0} and \ref{doublecosets-with-unipotent}.
More precisely
\begin{align*}
  [\xi]_0=G_n(F)\xi G_n(F) ,\quad [\xi]_1=P_{n+1}(F)\xi G_n(F),\\
  [\xi]_2=G_n(F)\xi P_{n+1}(F), \quad [\xi]_{12}=P_{n+1}(F)\xi P_{n+1}(F).
\end{align*}

In \S\ref{sec:orbital-integrals-without-unipotent} we study the term $I^0$ and
in \S\ref{sec:orbital-integrals-with-unipotent} we treat the remaining $I^\eta$.
In both these sections we hold $s_1,s_2,\phi_1,\phi_2$ fixed for which reason we drop them from the notations.

\section{Genuine orbital integrals}
\label{sec:orbital-integrals-without-unipotent}

In this section, we will analyze $I^0(\xi,T)$ for $\xi\in\mathfrak{X}_0$ and
understand its limit as $T\to\infty$.

\subsection{The orbits $\xi=e$ and $\xi=n^+$}\label{I_0-e-and-n+}

It is easy to see that $[e]_0\cup [n^+]_0$ corresponds
to the classes of matrices of the form 
$$\left\{u\mat{g&\\&1}\,\mid\,g\in G_n(F),u\in U_{n+1}(F)\right\}.$$
Therefore, from \eqref{I0-definition},
\begin{equation*}
	I^0(e,T)+I^0(n^+,T)=\int^T_{[G_n]\times G_n(\A)}\sum_{u\in U_{n+1}(F)}f\left[
		\begin{pmatrix} x^{-1}&\\&1 \end{pmatrix}u \begin{pmatrix} y&\\&1 \end{pmatrix}
		\right] \phi_1(x)\overline{\phi_2(y)}\d_{s_1} x\d_{s_2} y.
\end{equation*}
We identify $U_{n+1}(F)$ with $F^n$ in the obvious way and apply Poisson
summation leading to
\begin{equation}\label{apply-poisson}
	I^0(e,T)+I^0(n^{+},T)=\tilde{I}(T)+\tilde{I}^+(T),
\end{equation}
where
\begin{align}\label{def-I-tilde}
	\tilde{I}(T):&=\int^T_{[G_n]\times G_n(\A)}\int_{U_{n+1}(\A)}f\left[
		\begin{pmatrix} x^{-1}&\\&1 \end{pmatrix}u \begin{pmatrix} y&\\&1 \end{pmatrix}
		\right] \phi_1(x)\overline{\phi_2(y)}\d u\d_{s_1} x\d_{s_2} y\nonumber\\
  &=\int^T_{G_n(\A)\times [G_n]}\int_{U_{n+1}(\A)}f\left[
		\begin{pmatrix} x^{-1}&\\&1 \end{pmatrix}u \begin{pmatrix} y&\\&1 \end{pmatrix}
		\right] \phi_1(x)\overline{\phi_2(y)}\d u\d_{s_1} x\d_{s_2}y.
\end{align}
and
\begin{multline*}
	\tilde{I}^+(T):=\sum_{0\neq c^\top\in F^n}\int^T_{[G_n]\times G_n(\A)}
	\int_{\A^n}f\left[ \begin{pmatrix} x^{-1}&\\&1 \end{pmatrix}
		\begin{pmatrix} \I_n&b\\&1 \end{pmatrix}
		\begin{pmatrix} y&\\&1 \end{pmatrix} \right] \overline{\psi(cb)}\phi_1(x)\overline{\phi_2(y)}\\
        \d b\d_{s_1} x\d_{s_2} y.
\end{multline*}
The term $\tilde{I}(T)$ does not converge as $T\rightarrow\infty$ but, as we
will see, it gets canceled by terms coming from the regularization process
(see Lemma \ref{some-unip-orbital-integral-equal} below).

For the term $\tilde{I}^+(T)$, we parameterize $c\neq 0$ as $e_n\gamma$ for $\gamma\in
P_n(F)\backslash G_n(F)$ and change
variables $b\mapsto \gamma^{-1}b$ to obtain that $\tilde{I}^+(T)$ equals
\begin{equation*}
	\sum_{\gamma\in P_n(F)\bs G_n(F)}\int^T_{[G_n]\times G_n(\A)}
	\int_{\A^n}f\left[ \begin{pmatrix} x^{-1}&\\&1 \end{pmatrix}
		\begin{pmatrix} \I_n&\gamma^{-1}b\\&1 \end{pmatrix}
		\begin{pmatrix} y&\\&1 \end{pmatrix} \right]
	\overline{\psi(e_nb)}
	\phi_1(x)\overline{\phi_2(y)}\d b\d_{s_1} x\d_{s_2} y
\end{equation*}

Changing variables $y\mapsto \gamma^{-1} y$, using the automorphy of $\phi_2$,
and unfolding the sum over $P_n(F)\bs G_n(F)$ we obtain that the above equals
\begin{equation*}
	\int^T_{P_n(F)\bs G_n(\A)\times G_n(\A)}
	\int_{\A^n}f\left[ \begin{pmatrix} x^{-1}&\\&1 \end{pmatrix}
		\begin{pmatrix} \I_n&b\\&1 \end{pmatrix}
		\begin{pmatrix} y&\\&1 \end{pmatrix} \right] \overline{\psi(e_nb)}
	\phi_1(x)\overline{\phi_2(y)}\d b\d_{s_1} x\d_{s_2} y.
\end{equation*}
Expanding $\phi_1$ in the Fourier--Whittaker expansion \eqref{fourier-whittaker-expansion},
using the $G_n(F)$ invariance of the truncation support and unfolding once
again, leads to
\begin{equation*}
	\int^T_{N_n(F)\backslash G_n(\A)\times G_n(\A)}
	\int_{\A^n}f \left[ \begin{pmatrix} x^{-1}&\\&1 \end{pmatrix}
		\begin{pmatrix} \I_n&b\\&1 \end{pmatrix}
		\begin{pmatrix} y&\\&1 \end{pmatrix}\right]
	\overline{\psi(e_n b)}W_{1}(x)\overline{{\phi_2}(y)}\d b\d_{s_1} x\d_{s_2} y.
\end{equation*}
Finally, writing $x$ as $ux$ with $u\in[N_n]$ and
$x\in N_n(\A)\bs G_n(\A)$, changing variable $b\mapsto uxb$ and $y\mapsto uxy$,
and using the definition of the Whittaker function as in \eqref{def-whittaker}
we obtain that
\begin{equation*}
	\tilde{I}^+(T)=\int^T_{N_n(\A)\backslash G_n(\A)\times G_n(\A)}
	W_{1}(x)\overline{W_{2}(xy)}\hat{f}^u(e_nx,y)\d_{1+s_1+s_2} x\d_{s_2} y,
\end{equation*}
where, for $c\in\A^n$ and $y\in G_n(\A)$, we define
\begin{equation}\label{def-fu-hat}
  f^u(b,a)=f\sbr{\mat{a&b\\&1}},\quad
\hat{f}^u(c,a):=\int_{\A^n}f^u(b,a)\overline{\psi(cb)}\d b.
\end{equation}
The compact support of $f$ (modulo $Z_{n+1}(\A)$) ensures that $f^u$ is compactly
supported in both variables. Consequently
$\hat{f}^u(c,a)$ is a Bruhat--Schwartz function in $c$ and compactly
supported in $a$. Thus the theory of $\GL(n)\times\GL(n)$ global zeta integral
(see \cite[Lecture 2]{cogdell2007functions}) ensures that if $\Re(s_1+s_2)$ is sufficiently large the
last expression for $\tilde{I}^+(T)$ converges absolutely and uniformly in $T$.
Letting $T\to\infty$, recalling the definition of the global zeta integral from
\S\ref{zeta-integrals-n}, we deduce that
\begin{equation*}\label{I-tilde-plus-def}
  \tilde{I}^+(s_1,s_2):=\lim_{T\rightarrow \infty}\tilde{I}^+(s_1,s_2,T) =\int_{G_n(\A)}
	\Psi(1+s_1+s_2,W_{1},\overline{\pi_2(y)W_{2}},\hat{f}^u(\cdot,y))\d_{s_2}y.
\end{equation*}
As the $y$-integral has compact support, the above integral converges absolutely and uniformly in $s_1,s_2$ in fixed vertical strips. Employing meromorphic continuation of the global zeta integrals thus we deduce meromorphic continuation of
$\tilde{I}^+(s_1,s_2)$ to all $s_1,s_2\in\C$.
Moreover, the above integral is Eulerian, \emph{i.e.}, $\tilde{I}^+(s_1,s_2)=\prod_v\tilde{I}^+_v(s_1,s_2)$ where
\begin{equation}\label{Iv+tilde}
    \tilde{I}^+_v(s_1,s_2):={\int_{G_n(F_v)}\Psi_v(1+s_1+s_2,W_{1,v},\overline{\pi_{2,v}(y)W_{2,v}},\hat{f}^u_v(\cdot,y))\d_{s_2} y}
\end{equation}
The above integral converges for sufficiently large $\Re(s_1+s_2)$ and has meromorphic continuation to all of $\C^2$, which can be seen similarly as above.

On the other hand, the choices of $f$ in \S\ref{type-1-test-function} and \S\ref{type-2-test-function}, and \eqref{def-fu-hat} also ensure that
$$\hat{f}^u_v(c,y)=\mathbf{1}_{\o_v^n}(c)\mathbf{1}_{G_n(\o_v)}(y),\quad v\notin S.$$
Thus recalling choices of vectors in \S\ref{type-1-test-function} and \S\ref{type-2-test-function} we see that for $v\notin S$
$$\tilde{I}^+_v(s_1,s_2)=\Psi_v(1+s_1+s_2,W_{1,v},\overline{W_{2,v}},\mathbf{1}_{\o_v^n}).$$
Applying \eqref{unramified-zeta-integral-m}, first for large positive $\Re(s_i)$ and then
by uniqueness of meromorphic continuation, we arrive at the following lemma.
\begin{lem}\label{final-for-Itilde+}
We have
\begin{equation*}
    \tilde{I}^+(s_1,s_2)=\lim_{T\rightarrow \infty}\tilde{I}^{+}(s_1,s_2,T)=L^{S}(1+s_1+s_2,\pi_1\otimes\tilde{\pi}_2) \prod_{v\in S}\tilde{I}^+_v(s_1,s_2),
\end{equation*}
for $s_1,s_2\in\C$.
\end{lem}

\subsection{The orbit $\xi=n^-$} \label{I_0-n-}

We first note that we can parameterize $[n^-]_0$ by elements of the shape 
$$\left\{\mat{g&\\&1}u^\top\,\mid\,g\in G_n(F),\I_{n+1}\neq u\in U^\top_{n+1}(F)\right\}.$$
Therefore, from \eqref{I0-definition} we have
\begin{equation*}
	I^0(n^-,T)=\sum_{\I_{n+1}\neq u\in U_{n+1}(F)}
	\int^T_{G_n(\A)\times [G_n]}f\sbr{\begin{pmatrix} x^{-1}&\\&1\end{pmatrix}u^\top
  \begin{pmatrix}y&\\&1 \end{pmatrix}}
	\phi_1(x)\overline{\phi_2(y)}\d_{s_1} x\d_{s_2} y.
\end{equation*}
We parametrize $U_{n+1}(F)$ by $F^n$ and write the above as
\begin{equation*}
	I^0(n^-,T)=\sum_{0\neq c^\top\in F^n}\int^T_{G_n(\A)\times [G_n]}f
	\begin{pmatrix} x^{-1}y&\\cy&1 \end{pmatrix}
	\phi_1(x)\overline{\phi_2(y)}\d_{s_1} x\d_{s_2} y,
\end{equation*}
We further parameterize $c\neq 0$ as $e_n\gamma$ for $\gamma\in P_n(F)\backslash G_n(F)$,
change variables $y\mapsto \gamma^{-1} y$, use automorphy of $\phi_2$, and
unfold the sum over $P_n(F)\bs G_n(F)$. We thus get
\begin{equation*}
	I^0(n^-,T)=\int^T_{G_n(\A)\times P_n(F)\backslash G_n(\A)}f
	\begin{pmatrix} x^{-1}y&\\e_ny&1 \end{pmatrix}
	\phi_1(x)\overline{\phi_2(y)}\d_{s_1} x\d_{s_2} y,
\end{equation*}
Working as in \S\ref{I_0-e-and-n+} we obtain that the above equals 
\begin{equation*}
	\int^T_{G_n(\A)\times N_n(\A)\backslash G_n(\A)}
	W_{1}(x)\overline{W_{2}(y)}f\sbr{\mat{x^{-1}y&\\e_n y&1}} \d_{s_1} x\d_{s_2} y.
\end{equation*}
Letting $T\rightarrow \infty$, thanks to the compact support of $f$, we obtain that
\begin{equation}\label{I-minus-def}
	{I}^-(s_1,s_2):=\lim_{T\rightarrow \infty}I^0(n^-,T;s_1,s_2)= \int_{G_n(\A)}\Psi\left(s_1+s_2,
	\pi_1(x^{-1})W_{1},\overline{W_{2}},f^{\ell}(\cdot,x)\right)
	\d_{-s_1} x,
\end{equation}
where
\begin{equation*}
  f^\ell(c,a):=f\left[\mat{a&\\c&1}\right].
\end{equation*}
We apply $\GL(n)\times\GL(n)$ global functional equation to the above zeta integral
(see \S\ref{zeta-integrals-n}) to write it as
$$\int_{G_n(\A)}\Psi\left(1-s_1-s_2,
	\tilde{\pi}_1(x^\top)\tilde{W}_{1},\overline{\tilde{W}_{2}},\hat{f}^{l}(\cdot,x)\right)
	\d_{-s_1} x$$
where for $c'\in \A^n$,
\begin{equation}\label{def-fl-hat}
\hat{f}^l(c,x):=\int_{\A^n}f\sbr{\mat{x&\\c'&1}}\overline{\psi(c'c^\top)}\d c'.
\end{equation}
Working as in \S\ref{I_0-e-and-n+} we write
$I^-(s_1,s_2)=\prod_v{I}_v^-(s_1,s_2)$ where
\begin{equation}\label{Iv-tilde}
    I^-_v(s_1,s_2):={\int_{G_n(F_v)}\Psi_v(1-s_1-s_2,\tilde{\pi}_1(x^\top)\tilde{W}_{1,v},\overline{\tilde{W}_{2,v}},\hat{f}^l_v(\cdot,x))\d_{-s_1} x}.
\end{equation}
Recalling once again the choices of $f$ in \S\ref{type-1-test-function} and
\S\ref{type-2-test-function}, and applying the unramified computation
\eqref{unramified-zeta-integral-m}, first for large negative $\Re(s_i)$ and
then meromorphically continuing, as in \S\ref{I_0-e-and-n+} we arrive at the following lemma.
\begin{lem}\label{final-for-I-}
We have
\begin{equation*}
    I^-(s_1,s_2)=\lim_{T\rightarrow \infty}{I}^{0}(n^-,s_1,s_2,T)
    =L^{S}(1-s_1-s_2,\tilde{\pi}_1\otimes{\pi}_2)\prod_{v\in S}I^-_v(s_1,s_2),
\end{equation*}
for $s_1,s_2\in\C$.
\end{lem}

\subsection{The family $\xi=\xi(t)$}

Recall that, by definition, the orbits in this case are parameterized by
$t\in F$, $t\neq 0,1$ and are given by
\begin{equation*}
	\left\{ \begin{pmatrix} x^{-1}&\\&1 \end{pmatrix} \xi(t)
	\begin{pmatrix} y&\\&1 \end{pmatrix}\mid x,y\in G_n(F) \right\},\quad
	\xi(t)=\begin{pmatrix} I_n&te_n^{\top}\\e_n&1 \end{pmatrix}.
\end{equation*}
We compute the stabilizer of $\xi(t)$ by the action of
$G_n(F)\times G_n(F)$.
Let $x,\,y\in G_n(F)$ be such that
\begin{equation*}
	\mat{x^{-1}&\\&1}\xi(t)\mat{y&\\&1}\equiv \xi(t) \mod Z_{n+1}(F).
\end{equation*}
We readily show that
$$x=y,\quad e_nx=e_n,\quad xe_n^\top=e_n^\top.$$
The last two equalities translate into $x\in G_{n-1}(F)$. Therefore we may
parameterize $[\xi(t)]_0$ as
\begin{equation*}
	\left\{ \begin{pmatrix} x^{-1}&\\&1 \end{pmatrix} \xi(t)
	\begin{pmatrix} y&\\&1 \end{pmatrix} ;\; x\in G_{n-1}(F)\bs G_n(F),
	y\in G_n(F) \right\}.
\end{equation*}
As a consequence, from \eqref{I0-definition} we have
\begin{equation*}
	I^0(\xi(t),T)=\int^T_{G_{n-1}(F)\backslash G_n(\A)\times G_n(\A)}{f}\left[
		\begin{pmatrix} x^{-1}&\\&1 \end{pmatrix} \xi(t)
		\begin{pmatrix} y&\\&1 \end{pmatrix}\right] \phi_1(x)\overline{\phi_2(y)}
	\d_{s_1}x\d_{s_2} y.
\end{equation*}
From Proposition \ref{lem:xi-t-global} we have that
\begin{equation}\label{vanish-xi-t-contribution}
    I^0(\xi(t),T)=0,\quad \text{for }t\neq 0,1.
\end{equation}
for $N(\q)X$ being sufficiently large.

\begin{lem}\label{lem:xi-t-compact-support}
    Let $f_v$ be any smooth function of $\overline{G_{n+1}}(F_v)$ supported on a fixed compact set. If, for some $x,y\in G_n(F_v)$ and $1\neq t\in F_v$,
	we have
	\begin{equation*}
		{f}_v\left[ \begin{pmatrix} x^{-1}&\\ &1 \end{pmatrix}
			\begin{pmatrix} I_n&te_n^{\top}\\ e_n&1 \end{pmatrix}
			\begin{pmatrix} y&\\ &1 \end{pmatrix} \right]\neq 0,
	\end{equation*}
	then $\left|\frac{t}{1-t}\right|_{v}\ll 1$. 
\end{lem}

\begin{proof}
    From the description of $f_{v}$ we find a compact set $C\subset G_{n+1}(F_v)$ such that there exists $z\in F_v^\times$ with
	$$
		z\begin{pmatrix}
			x^{-1}y & tx^{-1}e_n^{\top} \\
			e_ny    & 1
		\end{pmatrix}\in C.
	$$
Any matrix on the left-hand side above can be written as
    \begin{equation*}
		\gamma=\begin{pmatrix}
			a & b \\
			c & d
		\end{pmatrix},\quad a\in G_n(F_v),\quad d\neq 0
	\end{equation*}
	with
	\begin{equation}\label{invariants-of-the-orbit}
		ca^{-1}bd^{-1}=t,\quad
		\det\gamma=(\det a)d(1-t).
	\end{equation}
 As $\gamma\in C$ we obtain that entries of $a,b,c$ are $O(1)$ and $\det(\gamma) \asymp 1$. Let $\tilde{a}$ be the adjugate matrix of $a$, given by $\tilde{a}:=(\det a)a^{-1}$. Now, we may rewrite the first equality in \eqref{invariants-of-the-orbit} as
	\begin{equation*}
		(\det a)dt=c\tilde{a}b.
	\end{equation*}
Thus we deduce that
$$\frac{t}{1-t}=\frac{c\tilde{a}b}{\det(\gamma)}.$$
Since $a,b,c$ all have bounded entries we have $c\tilde{a}b=O(1)$, and consequently, 
$\frac{t}{1-t}=O(1)$.
\end{proof}

\begin{lem}\label{lem:xi-t-non-arch}
  Let $f_v=\tilde{\1}_{\overline{K_0}(\p_v^{e_v})}$ with $e_v\ge0$.
If, for some $x,y\in G_n(F_v)$ and $t\in F_v$, $t\neq 0,1$, we have
\begin{equation*}
    {f_v}\left[ \begin{pmatrix} x^{-1}&\\ &1 \end{pmatrix} \begin{pmatrix} I_n&te_n^{\top}\\ e_n&1 \end{pmatrix} \begin{pmatrix} y&\\ &1 \end{pmatrix} \right]\neq 0,
\end{equation*}
then $\left|\frac{t}{1-t}\right|_{v}\le N(\p_v)^{-e_v}$.
\end{lem}

\begin{proof}
  The proof is essentially identical to that of the previous lemma, except that
  we now have that $a$, its adjugate $\tilde{a}$ and $b$ all have coefficients
  in $\o_v$ and $c$ has coefficients in $\p_v^{e_v}$ and $\det \gamma\in \o_v^\times$.
  Following the same reasoning, we may deduce that
  $\frac{t}{1-t}\in \p_v^{e_v}$, thus concluding.
\end{proof}

\begin{lem}\label{lem:xi-t-arch}
	Let $f_v$ be either of type I or type II and $v\mid\infty$.
	If, for some $x,y\in\GL_n(F_v)$ and $1\neq t\in F_v$,
	we have
	\begin{equation*}
		{f}_v\left[ \begin{pmatrix} x^{-1}&\\ &1 \end{pmatrix}
			\begin{pmatrix} I_n&te_n^{\top}\\ e_n&1 \end{pmatrix}
			\begin{pmatrix} y&\\ &1 \end{pmatrix} \right]\neq 0,
	\end{equation*}
	then $\left|\frac{t}{1-t}\right|_{v}\ll X_v^{-1}$.
\end{lem}

\begin{proof}
We proceed as in the proof of Lemma \ref{lem:xi-t-compact-support}.
    In either case, $f_v$ is supported on the set
    $$C_{X_v}:=\left\{\left.\gamma := \mat{a&b\\c&d}\quad\right\vert \quad\begin{aligned}
        & a\in \mathrm{Mat}_n(F_v), b,c^\top\in F_v^n, d\in F_v^\times; \\& \|a\|,\|b\|,\|X_vc\|,\|d\| =O(1);\det(\gamma)\asymp 1
    \end{aligned}\right\}\mod Z_{n+1}(F_v).$$
    Hence there exists $z\in F_v^\times$ such that
	$$z\begin{pmatrix}
			x^{-1}y & tx^{-1}e_n^{\top} \\
			e_ny    & 1
		\end{pmatrix}\in C_{X_v}.$$
  The determinant condition above gives us $z\det(zx^{-1}y)(1-t)\asymp 1$. On the other hand, as the entries of $zx^{-1}y$ are bounded the entries of $(zx^{-1}y)^{-1}$ are $\asymp |z(1-t)|$.
  Consequently, we have
  $$zt=(ze_ny)(zx^{-1}y)^{-1}(ztx^{-1}e_n^\top)\ll X_v^{-1}|z(1-t)|.$$
Hence $\left|\frac{t}{1-t}\right|\ll X_v^{-1}$.
\end{proof}

\begin{prop}\label{lem:xi-t-global}
    Let $N(\q)X$ be sufficiently large. Let $f$ be of either type I or type II. If, for some $x,y\in G_n(\A)$ and $F\ni t\neq 1$ we have
    \begin{equation*}
		{f}\left[ \begin{pmatrix} x^{-1}&\\ &1 \end{pmatrix}
			\begin{pmatrix} I_n&te_n^{\top}\\ e_n&1 \end{pmatrix}
			\begin{pmatrix} y&\\ &1 \end{pmatrix} \right]\neq 0,
	\end{equation*}
 then $t=0$.
\end{prop}

\begin{proof}
    Applying Lemma \ref{lem:xi-t-compact-support}, Lemma \ref{lem:xi-t-non-arch}, and Lemma \ref{lem:xi-t-arch} we obtain
    $$\left|\frac{t}{1-t}\right|_\A\ll (N(\q)X)^{-1},$$
    in either case. If $t\neq 0$ then $\frac{t}{1-t}\in F^\times$, which implies
    $$\left|\frac{t}{1-t}\right|_\A=1.$$
    The last two displays contradict each other
    for large $N(\q)X$. Hence, $t=0$. 
\end{proof}

\subsection{The orbit $\xi=\xi^\perp$}

We are now left with the orbit of the element
\begin{equation*}
	\xi^\perp= \begin{pmatrix} \I_n&e_1^{\top}\\ e_n&1 \end{pmatrix}.
\end{equation*}
We first calculate the stabilizer of $\xi^\perp$ under the action of $G_n(F)\times G_n(F)$. If
$$\begin{pmatrix} x^{-1}&\\&1 \end{pmatrix} \xi^\perp
	\begin{pmatrix} y&\\&1 \end{pmatrix}\equiv\xi^\perp\mod Z_{n+1}(F)$$
 then we have
 $$x=y,\quad x^{-1}e_1^\top= e_1^\top,\quad e_nx=e_n.$$
The last two equalities translate into $x\in R(F)$ where
\begin{equation*}
R:=\left\{ \begin{pmatrix} 1&\ast&\ast\\ &h&\ast\\ &&1 \end{pmatrix} \right\}
  \subset G_{n+1},\quad h\in G_{n-2},\quad.
\end{equation*}
We then see that $[\xi^\perp]_0$ can be parameterized as
\begin{equation*}
	\left\{ \begin{pmatrix} x^{-1}&\\&1 \end{pmatrix} \xi^\perp
	\begin{pmatrix} y&\\&1 \end{pmatrix} ;\; x\in R(F)\bs G_n(F),
	y\in G_n(F) \right\}
\end{equation*}
Therefore, from \eqref{I0-definition} we have
\begin{equation*}
	I^0(\xi^\perp,s_1,s_2,T)=\int_{R(F)\backslash G_n(\A)\times G_n(\A)}f
	\sbr{\begin{pmatrix} x^{-1}y&x^{-1}e_1^{\top}\\ e_ny&1 \end{pmatrix}}
	\phi_1(x)\overline{\phi_2(y)}\d_{s_1} x\d_{s_2} y.
\end{equation*}
Our goal is to show the above integral is negligible as $N(\q)X\to\infty$.

\begin{prop}\label{rapid-decay-xi-perp-type-1}
Let $f$ be of type $I$. Then the integral representation of $I^0(\xi^\perp,s_1,s_2,T)$ is absolutely convergent for all $s_1,s_2$ and	uniformly in $T$. Thus,
\begin{equation*}
	I^0(\xi^\perp,s_1,s_2):=\lim_{T\to\infty}I^0(\xi^\perp,s_1,s_2,T)
\end{equation*}
defines an entire function in $s_1,s_2$. Moreover, for $s_1,s_2$ at a neighbourhood of $(0,0)$ 
$$I^0(\xi^\perp,s_1,s_2)\ll_N (N(\q)X)^{-N}$$
for all large $N$.
\end{prop}

\begin{proof}
	First, we change variable $y\mapsto xy$ and obtain that
  \[
    I^0(\xi^\perp,s_1,s_2,T)=\int^{T}_{R(F)\bs G_n(\A)\times  G_n(\A)}{f}
    \sbr{\begin{pmatrix} y&x^{-1}e_1^{\top}\\ e_nxy&1 \end{pmatrix}}
    \phi_1(x)\overline{\phi_2(xy)}\d_{s_1+s_2} x\d_{s_2} y.
  \]
  First, note that
  $$\det\sbr{\mat{y&x^{-1}e_1^\top\\e_nxy&1}}=\det(y).$$
  From the support condition of $
  {f}$ as \S\ref{type-1-test-function} we get the following. Let
  \[
    {f}\begin{pmatrix}
      y         & x^{-1}e_1^{\top} \\
      e_nxy & 1
    \end{pmatrix}
    \neq 0.
  \]
  Then for each finite place $v$ there exists a $z_v\in F_v^\times$ such that
  \[
  z_vx_v^{-1}e_1^\top\in\o_v^n,\quad \p_v^{-e_v}z_ve_{n}x_vy_v\in\o_v^n,
  \quad z_v\in \o_v, \quad z_vy_v\in M_n(\o_v), \quad z_v^{n+1}\det(y_v)\in \o_v^\times.
  \]
	The last three conditions above imply that
	\[
    z_v\in\o_v^{\times},\quad\det(y_v)\in\o_v^\times\implies y_v\in G_n(\o_v),
    \quad x_v^{-1}e_1^\top\in\o_v^n,\quad \p_v^{-e_v}e_nx_v\in\o_v^n.
  \]
	Similarly, for $v\mid\infty$ there is $z\in F_v^\times$
	\[|z_v-1|_v,\quad\|y_v-\I_n\|_v,\quad|x_v^{-1}e_1^\top|_v,\quad
    X_v|e_nx_v|_v<\tau_v.
  \]
	Thus we can find a fixed compact neighborhood $C\subset G_n(\A)$ around the
  identity so that we can estimate
	\begin{equation}\label{eq:support-applied-xi-perp}
		I^0(\xi^\perp,s_1,s_2,T)\ll\int_{y\in C}\int^T_{\substack{x\in R(F)\bs G_n(\A)\\
    x_v^{-1}e_1^\top,\,\p_v^{-e_v}e_nx_v\in\o_v^n,\,v<\infty\\
    |x_v^{-1}e_1^\top|_v,\,X_v|e_nx_v|_v\le \tau_v,\,v\mid\infty}}
    |\phi_1(x)\phi_2(xy)| \d_{\Re(s_1+s_2)} x \d_{\Re(s_2)} y.
	\end{equation}
	Let $\tilde{R}$ be the parabolic of $G_{n}$ associated with the partition
  $n=1+(n-2)+1$ with Levi $\tilde{M}\cong G_1\times G_{n-2}\times G_1$ and
  unipotent radical $\tilde{U}$. Note that $\tilde{R}\simeq G_1\times R\times G_1$
  and $R\simeq \tilde{U}\times G_{n-2}$. We also have a proper embedding of
  $\tilde{R}(\A)\bs G_n(\A)\hookrightarrow K:=\prod_vK_v$. 
  We write the $R(F)\bs G_n(\A)$ as $R(F)\bs\tilde{R}(\A)\times \tilde{R}(\A)
  \backslash G_n(\A)$. On the other hand, we have $R(F)\bs \tilde{R}(\A)\cong
  [R]\times G_1(\A)^2$. As $G_{n-2}$ normalizes $\tilde{U}$ we also have
	$[R]=[\tilde{U}]\times[G_{n-2}]$. All in all, we have an embedding of
	\[
    R(F)\backslash G_n(\A)\hookrightarrow [\tilde{U}]\times[G_{n-2}]\times
    G_1(\A)^2\times K.
  \]
  We further embed $[G_{n-2}]$ inside $[{G_1}]\times\mathcal{S}$ where
  $\mathcal{S}$ is a fixed Siegel domain of $[\overline{G_{n-2}}]$.
  Thus we can (non-uniquely) write any element of $x\in R(F)\bs G_n(\A)$ as
	$\tilde{u}\diag(a_1,z\tilde{g},a_2)k$ for $\tilde{u}\in[\tilde{U}]$,
  $\tilde{g}\in\mathcal{S}$, and $k\in K$ and correspondingly
  \[
    \d x=\d \tilde{u}\frac{\d^\times a_1\d^\times a_2}{|a_1/a_2|^{n-1}}
    \d^\times z\d\tilde{g}\d k.
  \]
  On the other hand, as a $\phi_1$ is cusp form, it follows from \cite[Lemma I.2.10]{MW}
  that for any $M_1,M_2,N\ge 0$
  \begin{equation*}
		\phi_1(\tilde{u}\diag(a_1,z\tilde{g},a_2)k)\ll \min(1,|a_1/z|^{-M_1})
    \min(1,|z/a_2|^{-M_2})\min(1,|a_1/a_2|^{-N})
	\end{equation*}
	uniformly in $\tilde{u}\in[\tilde{U}]$, $\tilde{g}\in\mathcal{S}$, and $k\in K$.
	Moreover, $|\phi_2(xy)| \ll 1$ uniformly for $x\in G_{n+1}(\A)$ and $y\in C$.
	Thus we can bound \eqref{eq:support-applied-xi-perp} (after letting $T\to\infty$) by
	\begin{multline*}
		\ll_{M_1,M_2,N}\int_C\int_{K}\int_{\mathcal{S}}\int_{\tilde{U}}
		\int_{\substack{a_1,a_2\in\A^\times\\ |a_1|_v^{-1},N(\p_v)^{e_v}|a_2|_v\le 1,
		v<\infty\\|a_1|_v^{-1},X_v|a_2|_v\le \tau_v,v\mid\infty}}\int_{F^\times\bs \A^\times}
		|a_1a_2z^{n-2}|^{\Re(s_1+s_2)} \\
		\times \min(1,|a_1/z|^{-M_1})\min(1,|z/a_2|^{-M_2})\min(1,|a_1/a_2|^{-2N})\d^\times z
		\frac{\d^\times a_1\d^\times a_2}{|a_1/a_2|^{n-1}}\d \tilde{u}\d\tilde{g}\d k
		\d_{\Re(s_2)} y.
	\end{multline*}
  Noting that $C,K,[\tilde{U}]$ are compact and $\mathcal{S}$ has finite volume,
  we argue that the above integral is
	\begin{multline*}
		\ll_{M_1,M_2,N} (N(\q)X)^{-N}\int_{\substack{a_1,a_2\in\A^\times\\|a_1|_v^{-1},|a_2|_v\le 1}}|a_1|^{\Re(s_1+s_2)-n+1-N}|a_2|^{\Re(s_1+s_2)+n-1+N}\\
		\int_{F^\times\bs \A^\times}|z|^{(n-2)\Re(s_1+s_2)}\min(|z|^{M_1},|z|^{-M_2})
		\d^\times z{\d^\times a_1\d^\times a_2}.
	\end{multline*}
	Let $s_1,s_2$ be arbitrary.
	The $z$-integral above is absolutely convergent and is $O(1)$ uniformly in $a_1,a_2$. Thus we can estimate the above integral by
	$$\ll \prod_{v}\int_{|a_1|_v\ge 1}|a_1|_v^{\Re(s_1+s_2)+1-n-N}\d^\times a_1\int_{|a_2|_v\le 1}|a_2|_v^{\Re(s_1+s_2)+n-1+N}\d^\times a_2.$$
	Making $N$ sufficiently large we see that all of the integrals and the infinite products converge absolutely. Hence, we conclude.
\end{proof}

\begin{prop}\label{rapid-decay-xi-perp-type-2}
Let $f$ be of type II. Then the integral representation of $I^0(\xi^\perp,s_1,s_2,T)$ is absolutely convergent for all $s_1,s_2$ and	uniformly in $T$. Thus,
\begin{equation*}
	I^0(\xi^\perp,s_1,s_2):=\lim_{T\to\infty}I^0(\xi^\perp,s_1,s_2,T)
\end{equation*}
defines an entire function in $s_1,s_2$. Moreover, for $s_1,s_2$ at a neighbourhood of $(0,0)$ 
$$I^0(\xi^\perp,s_1,s_2)\ll_N (N(\q)X)^{-N}$$
for all large $N$.
\end{prop}

The proof is essentially the same as the proof of Proposition \ref{rapid-decay-xi-perp-type-1}. We are only required to modify the support conditions of the relevant variables.

\begin{proof}
Recall that
  $$\det\left[\begin{pmatrix}
      y         & x^{-1}e_1^{\top} \\
      e_nxy & 1
    \end{pmatrix}\right]=\det(y),$$
    and note that
    $$\begin{pmatrix}
      y         & x^{-1}e_1^{\top} \\
      e_nxy & 1
    \end{pmatrix}=\begin{pmatrix}
      \I_n         & x^{-1}e_1^{\top} \\
       & 1
    \end{pmatrix}\begin{pmatrix}
      y - x^{-1}e_1^\top e_nxy       &  \\
      e_nxy & 1
    \end{pmatrix}.$$
    If
  \[
    {f}\sbr{\begin{pmatrix}
      y         & x^{-1}e_1^{\top} \\
      e_nxy & 1
    \end{pmatrix}}
    \neq 0
  \]
  then for each finite place $v\neq \p_0,\p_1$, working as before, we have
  \[
    y_v\in G_n(\o_v),
    \quad x_v^{-1}e_1^\top\in\o_v^n,\quad e_nx_v\in\p_v^{e_v}\o_v^n.
  \]
For $v=\p_0$ or $v=\p_1$, we find compact sets $\mathcal{C}_1\subset G_{n+1}(F_v)$, $\mathcal{C}_2\subset F_v^\times$, and $\mathcal{C}_3\subset F_v^n$ so that
$$\begin{pmatrix}
      y  & x^{-1}e_1^{\top} \\
      e_nxy & 1
    \end{pmatrix}\in\mathcal{C}_1 \mod Z_{n+1}(F_v),\quad \det(y_v)\in\mathcal{C}_2,\quad x_v^{-1}e_1^\top, e_nx_vy_v\in\mathcal{C}_3.$$
Thus we conclude that there exists a compact set $C_v\subset G_n(F_v)$ such that
    \[
    y_v\in C_v,
    \quad |x_v^{-1}e_1^\top|_v=O(1),\quad |e_nx_v|_v=O(1).
  \]
 On the other hand, for each $v\mid\infty$ we see that entries of $y_v-x_v^{-1}e_1^\top e_n x_vy_v$, $x_v^{-1}e_1^\top$, and $e_nx_vy_v$ are bounded. Thus the entries of $y_v$ are bounded. But we also have $\det(y_v)\asymp 1$. Thus there exists a compact set $C_v\subset G_n(F_v)$ such that
 $$y_v\in C_v,\quad |x_v^{-1}e_1^\top|_v =O(1),\quad X_v|e_nx_vy_v|_v=O(1)\implies X_v|e_nx_v|_v=O(1).$$
Therefore, working as in the proof of Proposition \ref{rapid-decay-xi-perp-type-1} we find a compact set $C\subset G_n(\A)$ so that
\begin{equation*}
    I^0(\xi^\perp,s_1,s_2,T)\ll\int\int
    |\phi_1(x)\pi_2(y)\phi_2(x)| \d_{\Re(s_1+s_2)} x \d_{\Re(s_2)} y,
\end{equation*}
where we integrate over $y\in C, x\in R(F)\backslash G_n(\A)$ such that
$$
\begin{cases}
    T^{-1}<|\det x|< T,\\
    x_v^{-1}e_1^\top,\,\p_v^{-e_v}e_nx_v\in\o_v^n,\,\p_0,\p_1\neq v<\infty\\
    |x_v^{-1}e_1^\top|_v,\,X_v|e_nx_v|_v=O(1),\,v\mid\infty\\
     |x_v^{-1}e_1^\top|_v,\,|e_nx_v|_v=O(1),\,v=\p_0,\p_1.
\end{cases}
$$
 From here on, the proof is similar to the proof of Proposition \ref{rapid-decay-xi-perp-type-1}.
\end{proof}

\subsection{The orbits $\xi=w'$, $n^+w'$ and $w'n^+$}

Our goal here is to show that the contribution of these three terms vanishes due
to the support of the function $f$. 

\begin{lem}\label{vanish-genuine-orbital}
    Let $f$ be of either type I or type II. Then
    $$I^0(w',T)=I^0(n^+w',T)=I^0(w'n^+,T)=0$$
    for any $s_1,s_2$.
\end{lem}

\begin{proof}
    Recalling \eqref{I0-definition} we see that it is enough to show that $$f\sbr{\begin{pmatrix}x&\\&1\end{pmatrix}\xi\begin{pmatrix}y&\\&1\end{pmatrix}}=0$$
    for any $x,y\in G_n(\A)$. Thus, in turn, it suffices to show that for some $v\mid\infty$ $$f_v\sbr{\begin{pmatrix}x&\\&1\end{pmatrix}\xi\begin{pmatrix}y&\\&1\end{pmatrix}}=0$$
    for any $x,y\in G_n(F_v)$.

    Let $\xi=w'$ or $\xi=n^+w'$. In either case, $g_{n+1,n+1}=0$ where $$g:=\begin{pmatrix}x&\\&1\end{pmatrix}\xi\begin{pmatrix}y&\\&1\end{pmatrix}.$$
    Hence, if $v\mid\infty$ then $f_v(g)=0$ according to the choices in \S\ref{type-1-test-function} and \S\ref{type-2-test-function}.

    Now let $\xi=w'n^+$. In this case, we have
    $$g=\mat{x\J_0y&xe_n^\top\\e_ny&1}=\mat{x\J_{-1}y& xe_n^\top\\&1}\mat{\I_n&\\e_ny &1}$$
    where, again, $\J_t=\diag(\I_{n-1},t)$.
    Now for $v\mid\infty$ if $f_v(g)\neq 0$ then
    $$\det(x\J_{-1}y)\asymp 1,\quad |xe_n^\top|_v,|e_ny|_v\le \tau_v.$$
    This implies $\det(x\J_{-1}y+xe_n^\top e_ny)\asymp 1$. But this is a contradiction as $x\J_{-1}y+xe_n^\top e_ny=x\I_{0}y$ and $\det(\J_0)=0$. Hence, $f_v(g)=0$.
\end{proof}

\section{Orbital integrals with unipotent integration}
\label{sec:orbital-integrals-with-unipotent}

Recall the coset representatives  under the double action of $(P_{n+1}(F), G_n(F))$ from Proposition \ref{doublecosets-with-unipotent}.
In this section, we will analyze $I^{\eta}(\xi, T)$ for $\eta\in\{1,2,12\}$ and $\xi\in\{e,n^-,w'\}$.

\subsection{Computing $I^1(n^-,T)$}\label{sec:I-1-n-T}

We first compute the stabilizer of $n^-$ under the action of $P_{n+1}(F)\times G_n(F)$. Note that if
$$\mat{x^{-1}&u\\&1}n^-\mat{y&\\&1}=n^-\mod Z_{n+1}(F)$$
then $$x=y,\quad e_n y=y,\quad u=0.$$
Thus the stabilizer is $P_n(F)$ and thus $[n^-]_1$ is given by
$$\left\{hn^-\mat{g&\\&1}\,\mid\, h\in P_{n+1}(F), g\in P_n(F)\bs G_n(F)\right\}.$$
Recalling \eqref{I1-definition} we obtain
\begin{equation*}
	I^1(n^-,T)=\int^T_{G_n(\A)\times P_n(F)\backslash G_n(\A)}
	\int_{U_{n+1}(\A)}
	{f}\left[\begin{pmatrix} x^{-1}&\\&1 \end{pmatrix} un^-
	\begin{pmatrix} y&\\&1 \end{pmatrix}
	\right] \phi_1(x)\overline{\phi_2(y)}\d u\d_{s_1} x\d_{s_2}y.
\end{equation*}
We plug in the Fourier expansion of $\phi_2$ from \eqref{fourier-whittaker-expansion}. Noting that any $\gamma\in P_n(F)$ commutes with $n^-$ and changing variables
\[
  u\mapsto\mat{\gamma&\\&1}u\mat{\gamma^{-1}&\\&1},\quad x\mapsto\gamma x
\]
and using automorphy of $\phi_1$ we obtain that $I^1(n^-,T)$ can be expressed as 
\[
  \int^T_{G_n(\A)\times N_n(F)\backslash G_n(\A)}
	\int_{U_{n+1}(\A)}
	{f}\left[\begin{pmatrix} x^{-1}&\\&1 \end{pmatrix} un^-
	\begin{pmatrix} y&\\&1 \end{pmatrix}
	\right] \phi_1(x)\overline{W_{2}(y)}\d u\d_{s_1} x\d_{s_2}y.
\]
We unfold the $N_n(F)\backslash G_n(\A)$-integral as $[N_n]\times
N_n(\A)\backslash G_n(\A)\ni(\beta,y)$-integral. Noting that $\beta\in N_n(\A)$ also commutes
with $n^-$, changing variable $u\mapsto\mat{\beta&\\&1}u\mat{\beta^{-1}&\\&1}$ and
$x\mapsto\beta x$, and applying $\psi$-equivariance of $W_{2}$ we recast
the above expression as
\begin{equation*}
	\int^T_{G_n(\A)\times N_n(\A)\backslash G_n(\A)}
	\int_{U_{n+1}(\A)}
	{f}\left[\begin{pmatrix} x^{-1}&\\&1 \end{pmatrix} un^-
	\begin{pmatrix} y&\\&1 \end{pmatrix}
	\right] W_{1}(x)\overline{W_{2}(y)}\d u\d_{s_1} x \d_{s_2} y.
\end{equation*}
Parameterizing $u\in U_{n+1}(\A)$ by $b\in \A^n$ and changing variable $b\mapsto xb$ we write the above as
\begin{equation*}
    \int^T_{G_n(\A)\times N_n(\A)\backslash G_n(\A)}\int_{\A^n}{f}\left[\begin{pmatrix} x^{-1}y+be_ny&b\\e_ny&1 \end{pmatrix}\right] W_{1}(x)\overline{W_{2}(y)}\d b\d_{s_2} y \d_{s_1+1}x.
\end{equation*}
We define a region in $\Rr^2$ by
\begin{equation}\label{def-region-C1}
\mathcal{C}_1:=\left\{(u_1,u_2)\in\Rr^2_{\ge 0}\mid u_1+u_2,\frac{n-1}{2}+nu_1-u_2>0\right\}.
\end{equation}
We claim that the above integral converges absolutely for
$(\Re(s_1),\Re(s_2))\in\mathcal{C}_1$, uniformly in $T$. To show that, we consider
the above integral without truncation, namely,
\begin{equation}\label{final-I1-global}
\int_{G_n(\A)}\int_{N_n(\A)\backslash G_n(\A)}\int_{\A^n} {f}\left[\begin{pmatrix} x^{-1}y+be_ny&b\\ e_ny&1 \end{pmatrix} \right] W_{1}(x)\overline{W_{2}(y)}\d b\d_{s_2} y \d_{s_1+1} x
\end{equation}
and prove that it converges absolutely for $\Re(s_1),\Re(s_2)\in\mathcal{C}_1$.

We define the local factor at a place $v$ by
\begin{multline}\label{final-I1-local}
I_v^1(n^-,s_1,s_2):=\int_{G_n(F_v)}\int_{N_n(F_v)\backslash G_n(F_v)}\int_{F_v^n}
{f}_v\left[\begin{pmatrix}x^{-1}+be_ny&b\\ e_ny&1 \end{pmatrix}\right]\\
W_{1,v}(yx)\overline{W_{2,v}(y)}\d b\d_{s_1+s_2+1} y \d_{s_1+1} x.
\end{multline}
Noting the factorizability of \eqref{final-I1-global}, to prove the above claim, after a change of variable $x\mapsto yx$, it suffices to show that 
\begin{itemize}
    \item the integral in \eqref{final-I1-local},
    \item and the infinite product $\prod_{v\text{ unramified}}I_v^1(n^-,s_1,s_2)$
\end{itemize}
converge absolutely for $\Re(s_1),\Re(s_2)\in\mathcal{C}_1$.

\begin{lem}\label{meromorphic-cont-degenerate}
    Let $v$ be any place and let $f_v$ be any compactly supported smooth
    function on $\overline{G_{n+1}}(F_v)$. Then the integral defined in
    \eqref{final-I1-local} converges absolutely for $\Re(s_1),\Re(s_2)\in\mathcal{C}_1$,
    thus defining a holomorphic function in $s_1,s_2$ in that region.
\end{lem}

\begin{proof}
In this proof, we work locally at the place $v$ without mentioning the subscript $v$. Also, for the purpose of showing absolute convergence, we assume that $f, W_1, W_2$ assume
positive real values and $s_1,s_2$ are real.

We divide the above integral into two parts, namely, $|e_ny|\le 1$ and $|e_ny| > 1$. Note that if $|e_ny|\le 1$ then $\mat{\I_n&\\e_ny&1}$ varies over a compact set. Thus from the compact support condition of $f$ we obtain that there is a compact set $C\subset G_{n+1}(F)$ so that
$$\mat{x^{-1}&b\\&1}\in C\mod Z_{n+1}(F).$$
The above implies that $x$ and $b$ vary over compact sets in $C_1\subset G_n(F)$ and $C_2\subset F^n$, respectively. Consequently, the integral in consideration is bounded by
$$\ll \int_{C_1}\int_{N_n(F)\bs G_n(F)}W_1(yx)W_2(y)\1_{|e_ny|\le 1}\d_{s_1+s_2+1}y\d_{s_1+1} x.$$
The above integral converges absolutely for $\Re(s_1+s_2)>-1$, which follows from \cite[Lemma 3.1]{JaNu2021reciprocity}.

Now we consider the part with $|e_ny|>1$. First, using Iwasawa coordinates
$$N_n(F)\bs G_n(F)\ni y=z\diag(a,1)k;\quad z\in F^\times, a\in A_{n-1}(F), k\in K_{n,v};$$
$$\d y=\d^\times z\frac{\d a}{\delta(a)|\det(a)|}\d k$$
and changing variable $x\mapsto k^{-1}x$ we rewrite \eqref{final-I1-local} as
\begin{multline*}
    \int_{G_n(F)}\int_{K_{n,v}}\int_{F^n}\int_{F^\times}{f}\left[\begin{pmatrix} x^{-1}&b\\&1 \end{pmatrix}\mat{\I_n&\\ze_n&1}\mat{k&\\&1}\right]|z|^{n(1+s_1+s_2)}\d^\times z\d b\\
    \int_{A_{n-1}(F)}W_1\sbr{\mat{a&\\&1}x}{W_2}\sbr{\mat{a&\\&1}k}\frac{\d_{s_1+s_2} a}{|\delta(a)|}\d k\d_{s_1+1} x.
\end{multline*}
Note that $|e_ny|>1\iff |z|^{-1}<1$ and
$$\mat{\I_n&\\e_nz&1}=\mat{\I_{n-1}&&\\&z^{-1}&1\\&&z}\mat{\I_{n-1}&&\\&&-1\\&1&z^{-1}}.$$
Thus the compact support of $f$ implies that
$$\mat{x^{-1}&b\\&1}\mat{\I_{n-1}&&\\&z^{-1}&1\\&&z}\equiv \mat{z^{-1}x^{-1}\mat{\I_{n-1}&\\&z^{-1}}&b+z^{-1}x^{-1}e_n^\top\\&1}\in C\mod Z_{n+1}(F).$$
Thus we find compact sets $C_3\subset G_n(F)$ and $C_4\subset F^n$ so that
$$z^{-1}x^{-1}\mat{\I_{n-1}&\\&z^{-1}}\in C_3,\quad b+z^{-1}x^{-1}e_n^\top\in C_4.$$
Changing variable 
$$b\mapsto b-z^{-1}x^{-1}e_n^\top,\quad x\mapsto z^{-1}\mat{\I_{n-1}&\\&z^{-1}}x$$ and applying $Z_{n+1}(F)$-invariance of $W_1$ we bound the above integral by
\begin{multline*}
  \ll \int_{|z|>1}|z|^{n(1+s_1+s_2)}|z|^{-(n+1)(1+s_1)}\\
  \int_{C_3^{-1}}\int_{K_{n,v}}\int_{A_{n-1}(F)}W_1\sbr{\mat{az&\\&1}x}{W_2}\sbr{\mat{a&\\&1}k}\frac{\d_{s_1+s_2} a}{|\delta(a)|}\d k\d_{s_1+1} x\d^\times z.
\end{multline*}
Again we change variable $z\mapsto z^{-1}$ and $a\mapsto az$ and apply \cite[Lemma 3.1]{JaNu2021reciprocity} to have
$$W_2\sbr{\mat{az&\\&1}k}\ll \delta^{1/2}\sbr{\mat{az&\\&1}},$$
and consequently, bound the three inner integrals above by
$\ll |z|^{\frac{n-1}{2}}$
as long as $s_1+s_2>-1$. Thus the total integral is bounded by
$$\int_{|z|<1}|z|^{ns_1-s_2+\frac{n+1}{2}}\d^\times z$$
which is convergent if $\frac{n+1}{2}+ns_1-s_2>0$.
\end{proof}

\begin{lem}\label{global-I1}
Let $f$ be of type I or type II. Then we have
$$I^1(n^-,s_1,s_2):=\lim_{T\to\infty} I^1(n^-,T,s_1,s_2)=\prod_{v\le\infty}I^1_v(n^-,s_1,s_2)$$
which is holomorphic in $s_1,s_2$ for $\Re(s_1),\Re(s_2)\in\mathcal{C}_1$ and has meromorphic continuation for all $s_1,s_2\in\C$.
\end{lem}

\begin{proof}
From Lemma \ref{degenerate-unramified}, we see that
$$\prod_{v\not\in S}I_v^1(n^-,s_1,s_2)$$
converges absolutely and uniformly in $s_1,s_2$ in compact sets  with $(\Re(s_1),\Re(s_2))\in\mathcal{C}_1$. Thus employing Lemma \ref{meromorphic-cont-degenerate} yields holomorphicity of $I^1(n^-,s_1,s_2)$ in this region. Moreover, also applying Lemma \ref{degenerate-level}, in this region we have
\begin{multline*}
I^1(n^-,s_1,s_2)=\Delta^{\mu_D(s_1,s_2)}L(1+s_1+s_2,\pi_1\otimes\widetilde{\pi}_2)N(\q)^{-n(s_1+s_2)}\\
\frac{\zeta_\q(1)L^{\q}(\frac{n+1}{2}+n s_1-s_2,\pi_2)}{\zeta_\q(n+1)L^{\q}(\frac{n+3}{2}+(n+1)s_1,\pi_1)}
I^1_\infty(n^-,s_1,s_2),
\end{multline*}
We conclude the proof by employing meromorphic continuation of $L$-functions and applying Lemma \ref{degenerate-arch}.
\end{proof}

\subsection{Computing $I^2(n^-,T)$}

The computation of $I^2(n^-,T)$ is quite similar to that of $I^1(n^-,T)$. We give a sketch here.
First, working as in \S\ref{sec:I-1-n-T} we first check that
$$[n^-]_2=\left\{\mat{g^{-1}&\\&1}n^-h\,\mid\, g\in P_n(F)\bs G_n(F),h\in P_{n+1}(F)\right\}$$
Thus recalling \eqref{I2-definition} we obtain that
\begin{equation*}
	I^2(n^-,T)=\int^T_{P_n(F)\backslash G_n(\A)\times G_n(\A)}
	\int_{U_{n+1}(\A)}
	{f}\left[\begin{pmatrix} x^{-1}&\\&1 \end{pmatrix} n^-u
	\begin{pmatrix} y&\\&1 \end{pmatrix}
	\right] \phi_1(x)\overline{\phi_2(y)}\d u\d_{s_1} x\d_{s_2}y.
\end{equation*}
Once again working as in \S\ref{sec:I-1-n-T} we arrive at
\begin{equation*}
	I^2(n^-,T)=\int^T_{N_n(\A)\backslash G_n(\A)\times G_n(\A)}
	\int_{U_{n+1}(\A)}
	{f}\left[\begin{pmatrix} x^{-1}&\\&1 \end{pmatrix} n^-u
	\begin{pmatrix} y&\\&1 \end{pmatrix}
	\right] W_{1}(x)\overline{W_{2}(y)}\d u\d_{s_1} x \d_{s_2} y.
\end{equation*}
We define $$f^-:g\mapsto f(g^{-1}),$$
parameterize $u\in U_{n+1}(\A)$ by $b\in\A^n$, and change variables $b\mapsto-yb$, $x\mapsto -x$, and $y\mapsto -y$ to write the above integral as
\begin{equation*}
    \int^T_{G_n(\A)\times N_n(\A)\backslash G_n(\A)}
	\int_{U_{n+1}(\A)}
	{f}^-\left[\begin{pmatrix} x^{-1}y+be_ny&b\\e_ny&1 \end{pmatrix}
	\right] W_{1}(y)\overline{W_{2}(x)}\d u\d_{s_1} y \d_{s_2} x.
\end{equation*}
Note that as $f$ is compactly supported in $\overline{G_{n+1}}(\A)$ so is $f^-$.

We define the local factor
\begin{multline}\label{final-I2-local}
    I^2_v(n^-,s_1,s_2):=\int_{G_n(F_v)}\int_{N_n(F_v)\backslash G_n(F_v)}
	\int_{F_v^n}
	{f}_v^{-}\left[\begin{pmatrix} x^{-1}+be_ny&b\\ e_ny&1 \end{pmatrix}
	\right]\\ W_{1,v}(y)\overline{W_{2,v}(yx)}\d b\d_{s_1+s_2+1} y \d_{s_1+1} x
\end{multline}
and also define a region in $\Rr^2$ by
\begin{equation}\label{def-region-C2}
    \mathcal{C}_2:=\left\{(u_1,u_2)\in\Rr^2_{\ge 0}\mid u_1+u_2,\frac{n-1}{2}+nu_2-u_1>0\right\}.
\end{equation}
Then applying Lemma \ref{meromorphic-cont-degenerate} and working exactly as in the proof of Lemma \ref{global-I1} we prove the following result.

\begin{lem}\label{global-I2}
Let $f$ be of type I or type II. Then we have
$$I^2(n^-,s_1,s_2):=\lim_{T\to\infty} I^2(n^-,T,s_1,s_2)=\prod_{v\le\infty}I^2_v(n^-,s_1,s_2)$$
which is holomorphic in $s_1,s_2$ for $\Re(s_1),\Re(s_2)\in\mathcal{C}_2$ and has meromorphic continuation for all $s_1,s_2\in\C$.
\end{lem}

\subsection{The degenerate term}

We define the \emph{degenerate term} as
\begin{equation}\label{def-degenerate}
    \mathcal{D}(s_1,s_2):=I^1(n^-,s_1,s_2)+I^2(n^-,s_1,s_2).
\end{equation}

Our main result in this subsection is as follows.

\begin{prop}\label{main-prop-degenerate}
Let $f$ be of type I. Then we have
\begin{multline*}\mathcal{D}(s_1,s_2)=
\Delta^{\mu_D(s_1,s_2)}L(1+s_1+s_2,\pi_1\otimes\widetilde{\pi}_2)N(\q)^{-n(s_1+s_2)}\frac{\zeta_\q(1)}{\zeta_\q(n+1)}\\
\left(\frac{L^{\q}(\frac{n+1}{2}+n s_1-s_2,\pi_2)}{L^{\q}(\frac{n+3}{2}+(n+1)s_1,\pi_1)}
I^1_\infty(n^-,s_1,s_2)
+\frac{L^{\q}(\frac{n+1}{2}+n s_2-s_1,\tilde{\pi}_1)}{L^{\q}(\frac{n+3}{2}+(n+1)s_2,\tilde{\pi}_2)}
I^2_\infty(n^-,s_1,s_2)\right)
\end{multline*}
for all $s_1,s_2\in\C$.

On the other hand, if $f$ is of type II then $\mathcal{D}(s_1,s_2)$ vanishes identically.
\end{prop}

\begin{proof}
    The first assertion follows immediately from the proofs of Lemma \ref{global-I1} and Lemma \ref{global-I2}.

    To prove the second assertion it suffices to show that $I^j(n^-,s_1,s_2)=0$ for $s_1,s_2$ lying in an open set with $(\Re(s_1),\Re(s_2))\in \mathcal{C}_1\cap\mathcal{C}_2$, which is the content of Lemma \ref{degenerate-sc}.
We conclude by applying Lemma \ref{global-I1} and Lemma \ref{global-I2}.    
\end{proof}

\begin{rmk}\label{upgrade-Ijn-holomorphic}
    We remark that from the proof of Lemma \ref{meromorphic-cont-degenerate} we see that for each $v$ the local integral $I_v^1(n^-,s_1,s_2)$ is holomorphic in $s_1,s_2$ whenever $\Re(s_1)+\Re(s_2)>-1$ and $\frac{n+1}{2}+n\Re(s_1)-\Re(s_2)>0$. Following a similar proof, we also have $I_v^2(n^-,s_1,s_2)$ is holomorphic in $s_1,s_2$ whenever $\Re(s_1+s_2)>-1$ and $\frac{n+1}{2}+n\Re(s_2)-\Re(s_1)>0$.
\end{rmk}

\subsection{Cancellation and vanishing of certain orbital integrals}

Recall $\tilde{I}(T)$ from \eqref{def-I-tilde}. Here we show the following.

\begin{lem}\label{some-unip-orbital-integral-equal}
We have
\begin{equation*}
    \tilde{I}(T)=I^1(e,T)=I^2(e,T)=I^{12}(e,T)
\end{equation*}
for all $s_1,s_2$.
\end{lem}

\begin{proof}
It is immediate to see that $[e]_1=P_{n+1}(F)$. Thus from \eqref{I1-definition} we see that
$$I^1(e,T)=\int^T_{G_n(\A)\times [G_n]}\int_{U_{n+1}(\A)}f\left[\begin{pmatrix} x^{-1}&\\&1 \end{pmatrix}u \begin{pmatrix} y&\\&1 \end{pmatrix}\right] \phi_1(x)\overline{\phi_2(y)}\d u\d_{s_1} x\d_{s_2}y.$$
We conclude $\tilde{I}(T)=I^1(e,T)$ from \eqref{def-I-tilde}. Similarly, we prove $\tilde{I}(T)=I^2(e,T)$.

It is also evident that $[e]_{12}=P_{n+1}(F)$. Thus from \eqref{I12-definition} we obtain that $I^{12}(e,T)$ equals
$$\int^T_{G_n(\A)\times [G_n]}\int_{U_{n+1}(\A)\times[U_{n+1}]}f\left[\begin{pmatrix} x^{-1}&\\&1 \end{pmatrix}u_1^{-1}u_2 \begin{pmatrix} y&\\&1 \end{pmatrix}\right] \phi_1(x)\overline{\phi_2(y)}\d u_1\d u_2\d_{s_1} x\d_{s_2}y.$$
Changing variable $u_1\mapsto u_2 u_1$ and using that $\vol\left([U_{n+1}]\right)=1$ we see the above equals
$$\int^T_{G_n(\A)\times [G_n]}\int_{U_{n+1}(\A)}f\left[\begin{pmatrix} x^{-1}&\\&1 \end{pmatrix}u \begin{pmatrix} y&\\&1 \end{pmatrix}\right] \phi_1(x)\overline{\phi_2(y)}\d u\d_{s_1} x\d_{s_2}y.$$
Recalling \eqref{def-I-tilde} we conclude $I^{12}(e,T)=\tilde{I}(T)$.
\end{proof}

We now show vanishing of other orbital integrals that were not treated so far.

\begin{lem}\label{vanish-unip-orbital-integral}
Let $f$ be either of type I or type II. Then we have
\begin{equation*}
	I^1(w',T)=I^2(w',T)=I^{12}(w',T)=0
\end{equation*}
for all $s_1,s_2$.
\end{lem}

\begin{proof}
Proofs of vanishing of $I^\eta(w',T)$ for $\eta=1,2$ are quite similar to the proof of Lemma \ref{vanish-genuine-orbital}, namely, we show that the domain of integrations in the integral representations of $I^\eta$ as in \eqref{I1-definition} and \eqref{I2-definition} do not intersect the support of $f$ as in \S\ref{type-1-test-function} and \S\ref{type-2-test-function}. Indeed, for any $x\in G_n(\A),\,y\in G_n(\A),\,u\in
	U_{n+1}(\A)$, writing
\begin{equation*}
    g=\begin{pmatrix}x&\\&1\end{pmatrix}uw'\begin{pmatrix}y&\\&1\end{pmatrix},
\end{equation*}
we always have $g_{n+1,n+1}=0$ and thus $f_v(g)=0$ for $v\mid\infty$. Hence, $I^1(w',T)=0$.

On the other hand, note that
\begin{equation*}
g=\begin{pmatrix}x&\\&1\end{pmatrix}w'u\begin{pmatrix}y&\\&1\end{pmatrix}=\mat{x\J_0y&\ast\\\ast&\ast},
\end{equation*}
where, we recall, $\J_t=\diag(\I_{n-1},t)$. In particular, $\det(x\J_0y)=0$. However, $f_v$, for $v\mid\infty$, is supported on the elements of the form $\mat{h&\\&1}$ with $h\in G_n(F_v)$ such that $\det(h)\asymp 1$. Hence, again, $f(g)=0$ and so $I^2(w',T)=0$.

We now show that $I^{12}(w', T)$ also vanishes identically. This time, however,
the reason is a different one, namely, the cuspidality of $\phi_1$ (or $\phi_2$).

First we compute the stabilizer of $w'$ under the action of $P_{n+1}(F)\times P_{n+1}(F)$.
It suffices to find elements $h\in P_{n+1}(F)$ so that $w'hw'\in P_{n+1}(F) \mod Z_{n+1}(F)$.
Suppose that for $a\in G_n(F)$, $b\in F^n$, so that
$$w'\mat{a&b\\0&1}w'\in P_{n+1}(F)\mod Z_{n+1}(F).$$
Then we have $a\in Q_n$ and $e_nb=0$. As a consequence, we deduce that
$$[w']_{12}=\left\{hw'\mat{g&te_n^\top\\&1}\,\mid\, h\in P_{n+1}(F), g\in Q_n(F)\bs G_n(F), t\in F\right\}.$$
Using the above parameterization we see from \eqref{I12-definition} that $I^{12}(w',T)$ equals
\begin{multline*}
	\int^T_{G_n(\A)\times Q_n(F)\backslash G_n(\A)} \int_{U_{n+1}(\A)}
		\int_{(F\backslash \A)^{n-1}}\int_{\A}{f}\left[\mat{x^{-1}&\\&1}u_1^{-1}
	w'u_{21}(b')u_{22}(b_n)\mat{y&\\&1}\right]\\
	\times \phi_1(x)\overline{\phi_2(y)}\d b_n\d b'\d u_1\d_{s_1}x\d_{s_2}y
\end{multline*}
where
\begin{equation*}
	u_{21}(b')=\begin{pmatrix}\I_{n-1}&&b'\\&1&\\ &&1\end{pmatrix}, b'\in F^{n-1};\quad
	u_{22}(b_n)=\begin{pmatrix}\I_{n-1}&&\\&1&b_n\\ &&1\end{pmatrix}, b_n\in F.
\end{equation*}
Note that 
$$u_1^{-1}w'u_{21}(b')=\mat{\tilde{u}(b')&\\&1}u'w'$$
where $\tilde{u}(b'):=\mat{\I_{n-1}&b'\\&1}\in U_n$ and $u':=\mat{\tilde{u}(-b')&\\&1}u_1^{-1}\mat{\tilde{u}(b')&\\&1}\in U_{n+1}$.
We change variables $x\mapsto \tilde{u}(b')x$, $u_1^{-1}\mapsto u'$, and identify $[U_n]$ with $(F\bs\A)^{n-1}$ to write the above integral as
\begin{multline*}
	\int^T_{G_n(\A)\times Q_n(F)\backslash G_n(\A)} \int_{U_{n+1}(\A)}
		\int_{\A}{f}\left[\mat{x^{-1}&\\&1}u'
	w'u_{22}(b_n)\mat{y&\\&1}\right]\\
	\times \left(\int_{[U_n]}\phi_1(\tilde{u}x)\d\tilde{u}\right)\overline{\phi_2(y)}\d b_n\d u_1\d_{s_1}x\d_{s_2}y.
\end{multline*}
The innermost integral vanishes identically as $\phi_1$ is cuspidal. Hence we conclude.
\end{proof}

\part{Local computations}\label{part-local}

In this part, we mostly work locally at a place $v\le\infty$. We will always
drop the subscript $v$ to ease the notation if the place is clear from the context.
We recall $\q$ and $\mathfrak{X}$ from Theorem \ref{second-moment-nonarch} and
Theorem \ref{second-moment-arch}. We fix $f$ to be a test function of type I as
in \S\ref{type-1-test-function} or of type II as in \S\ref{type-2-test-function}.
We also recall the corresponding choices of local test vectors there.

For every finite place we first perform our computations under the hypothesis
that the place is unramified and only later on we explain
the necessary changes for the ramified ones.

\section{The spectral weights}

In this section, $\Pi$ will be a generic unitary representation of $\overline{G_{n+1}}(F)$. Moreover, we assume that $\Pi$ appears as a local component of a certain generic standard automorphic representation, and thus is at most $\vartheta$-tempered (see \cite[\S3.5]{JaNu2021reciprocity}) for some $0\le\vartheta<1/2$. Recall the local spectral weights $H(\Pi)$ from \S\ref{sec:spectral-side}. We always assume $s_1=s_2=0$, unless otherwise mentioned.

\subsection{Uninteresting places}

\begin{lem}\label{bounds-spectral-nonarchimedean}
Let $v$ be unramfied and  
\begin{itemize}
    \item either $f$ be of either type I and $\infty>v\nmid\q$,
    \item or $f$ be of type II and $\infty>v\nmid \q\p_0\p_1$.
\end{itemize}
Then
$$H(\Pi;s_1,s_2)= \delta_{c(\Pi)=0}.$$
\end{lem}

\begin{proof}
We first write
$$H(\Pi;s_1,s_2)=\sum_{W\in\B(\Pi)}\frac{\Psi(1/2+s_1,\Pi(\mathbf{1}_{\overline{G_{n+1}}(\o)})W,W_{{1}}){\Psi(1/2+s_2,\tilde{\Pi}(\mathbf{1}_{\overline{G_{n+1}}(\o)})\overline{W},\overline{W_{{2}}})}}{L(1/2+s_1,\Pi\otimes{\pi}_{1})L(1/2+s_2,\tilde{\Pi}\otimes\tilde{\pi}_{2})}.$$
We choose $\B(\Pi)$ to be a $K$-type orthonormal basis. In this case, $\Pi(\mathbf{1}_{\overline{G_{n+1}}(\o)})W$ vanishes identically unless $\Pi$ is unramified and $W=\frac{W_{\Pi}}{\|W_\Pi\|}$, in which case, it equals $\frac{W_{\Pi}}{\|W_\Pi\|}$. Thus we have
$$H(\Pi;s_1,s_2)=\|W_\Pi\|^{-2}\frac{\Psi(1/2+s_1,W_\Pi,W_{{1}}){\Psi(1/2+s_2,\overline{W_\Pi},\overline{W_{{2}}})}}{L(1/2+s_1,\Pi\otimes{\pi}_{1})L(1/2+s_2,\tilde{\Pi}\otimes\tilde{\pi}_{2})}.$$
We conclude by noting that $\|W_\Pi\|=1$ (see \cite[eq.(16)]{JaNu2021reciprocity}) and using \eqref{unramified-zeta-integral-m-1}.
\end{proof}

\subsection{Level places}

\begin{lem}\label{upper-bounds-spectral-nonarchimedean}
  Let $v\mid\q$ be an unramified place and let $f=\tilde{\1}_{\overline{K_0}(\p^e)}$.
Then $H(\Pi)$ vanishes unless $c(\Pi)\le e$. Moreover, we have $H(\Pi)\ll_\epsilon N(\p^{e})^\epsilon$.
\end{lem}

\begin{proof}
As in the proof of Lemma \ref{bounds-spectral-nonarchimedean} we write
$$H(\Pi)=\sum_{W\in\B(\Pi)}\frac{\Psi(1/2,\Pi(f^0)W,W_{{1}}){\Psi(1/2,\tilde{\Pi}(f^0)\overline{W},\overline{W_{{2}}})}}{L(1/2,\Pi\otimes{\pi}_{1})L(1/2,\tilde{\Pi}\otimes\tilde{\pi}_{2})}.$$
where $f^0=\tilde{\1}_{\overline{K_0}(\p^e)}$. Thus applying newvector theory (see \S\ref{sec:newvectors}) we see that
$\Pi(f^0)W=0$ unless $c(\Pi)\le e$ in which case, if $W\in\Pi^{\overline{K_0}(\p^e)}$ then $\Pi(f_0)W=W$. This yields the first claim.

For the upper bound, we first restrict the sum as
$$H(\Pi)=\sum_{W\in\B\left(\Pi^{\overline{K_0}(\p^e)}\right)}\frac{\Psi(1/2,W,W_{{1}}){\Psi(1/2,\overline{W},\overline{W_{{2}}})}}{L(1/2,\Pi\otimes{\pi}_{1})L(1/2,\tilde{\Pi}\otimes\tilde{\pi}_{2})}.$$
using the above argument.
It follows from \cite[Lemma 7.3]{JaNu2021reciprocity} that for any $W\in \Pi^{\overline{K_0}(\p^e)}$, one has
\begin{equation*}
  \Psi(1/2,W,W_{i})\ll_\epsilon N(\p)^{\epsilon(e-c(\Pi))},\quad i=1,2
\end{equation*}
We conclude by noting that
$\dim(\Pi^{K_0(\p^{e})})$ is a polynomial in $c(\Pi)$ and $e$, hence is $\ll_\epsilon N(\p^e)^\epsilon$ for $c(\Pi)\le e$; see \cite{reeder1991old}. Combining these estimates we conclude.
\end{proof}

\begin{lem}\label{lower-bounds-spectral-nonarchimedean}
    Let $v\mid\q$ be an unramified place and $f$ be of type I. Moreover, assume that $\pi_1\cong\pi_2$ and $W_1=W_2$. Then $H(\Pi)\ge 0$ for every unitary $\pi$ and $H(\Pi)\gg 1$ if $c(\Pi)\le e$.
\end{lem}

\begin{proof}
    Working as in the proof of Lemma \ref{upper-bounds-spectral-nonarchimedean} we see that
    $$H(\Pi)=\sum_{W\in\B\left(\Pi^{\overline{K_0}(\p^e)}\right)}\left|\frac{\Psi(1/2,W,W_{{1}})}{L(1/2,\Pi\otimes{\pi}_{1})}\right|^2\ge 0.$$
    Now noting that $W_\Pi\in \Pi^{\overline{K_0}(\p^e)}$ and choosing $\B\left(\Pi^{\overline{K_0}(\p^e)}\right)\ni\frac{W_\Pi}{\|W_\Pi\|}$ we obtain
    $$H(\Pi)\ge \|W_\Pi\|^{-2}\left|\frac{\Psi(1/2,W_\Pi,W_{{1}})}{L(1/2,\Pi\otimes{\pi}_{1})}\right|^2.$$
    It follows from the argument in \cite[eq.(17)]{JaNu2021reciprocity} that $\|W_\Pi\|^2\asymp 1$.
    We conclude the proof using \eqref{unramified-zeta-integral-m-1}.
\end{proof}
We remark that one can make the implied constant in Lemma \ref{lower-bounds-spectral-nonarchimedean} arbitrary close to $1$ for sufficiently large $v$.

\subsection{Supercuspidal prime}

\begin{lem}\label{supercuspidal-projects-cuspidal}
Let $v=\p_0$ and $f$ be of type II. Recall the supercuspidal representation
$\sigma$ of $\overline{G_{n+1}}(F)$ underlying the definition of $f$. Then the
local weight $H(\Pi)$ vanishes identically unless $\Pi\cong\tilde{\sigma}$.
\end{lem}

\begin{proof}
It suffices to show that for any $W\in\Pi$ we have $\Pi(f)W(1)=0$ unless $\Pi\cong\tilde{\sigma}$. We have
\begin{equation*}
\Pi(f)W(1) =\int_{\overline{G_{n+1}}(F)}f(g)W(g)\d g =\int_{\overline{G_{n+1}}(F)}\langle\sigma(g)W_\sigma,W_\sigma\rangle W(g)\d g.
\end{equation*}
Now the bilinear pairing between $\Pi$ and $\sigma$ given by
$$\Pi\times\sigma\ni (W,W')\mapsto \int_{\overline{G_{n+1}}(F)}\langle\sigma(g)W',W_\sigma\rangle W(g)\d g$$
is well-defined, as the matrix coefficient $\langle\sigma(g)W',W_\sigma\rangle$
is compactly supported as a function of $g\in\overline{G_{n+1}}(F_v)$
and is clearly $\overline{G_{n+1}}(F_v)$-invariant. As $\Pi$ and $\sigma$ are irreducible, by Schur's lemma the pairing must vanish identically (\emph{i.e.}, for all $W'\in\sigma$ and $W\in\Pi$), unless $\Pi\cong\tilde{\sigma}$. Plugging in $W'=W_\sigma$ we conclude.
\end{proof}

\subsection{Auxiliary place}

\begin{lem}\label{aux-weight-majorize}
    Let $v=\p_1$ and $f$ be of type II. Recall $\nu\in\Z_{\ge 0}$ underlying the definition of $f$. Then the local weight $H(\Pi)$ vanishes identically unless $c(\Pi)\le\nu$.
\end{lem}

\begin{proof}
    It suffices to show that $f$ is right invariant by $\overline{K_0}(\p_1^\nu)$.
    Indeed, if that is the case then $\Pi(f)$ will project on to
    $\Pi^{\overline{K_0}(\p_1^\nu)}$. By newvector theory (see \S\ref{sec:newvectors}), the above space is zero unless $c(\Pi)\le\nu$ and the lemma follows.

    To see the claim on invariance, we show that
    $$f(g)=1\implies f(gk)=1$$
    for any element in $k\in\overline{K_0}(\p_1^\nu)$ and $g\in\overline{G_{n+1}}(F)$. The above also shows that
    $$f(gk)=1\implies f(gkk^{-1})=f(g)=1,$$
    consequently, $f(g)=f(gk)$.

    We note that any $k$ can be written as
    $$k:=\mat{\I_n&b'\\&1}\mat{a'&\\&1}\mat{\I_n&\\c'&1},\quad b',\p_1^{-\nu}c'\in\o^n,
    a'\in K_n$$
    and any $g$ with $g_{n+1,n+1}\neq 0$ can be written as
    $$g=\mat{\I_n&b\\&1}\mat{a&\\&1}\mat{\I_n&\\c&1},\quad b\in F^n, c\in F^n, a\in G_n(F).$$
    We record that
    $$gk=\mat{\I_n&b+d^{-1}ab'\\&1}\mat{d^{-2}a(d\I_n-b'c)a'&\\&1}\mat{\I_n&\\c'+d^{-1}ca'&1}\mod Z_{n+1}(F),$$
    where $d:=1+cb'$.
    If $f(g)=1$ then from the support condition
    $$b\in\o^n,\, a\in K_n\diag(\p_1^\mu)K_n,\, c\in\p_1^\nu\o^n.$$
    Consequently, $$d\in\o^\times,\,d\I_n-b'c\in K_n,$$
    and thus
    $$b+d^{-1}ab'\in\o^n,\,d^{-2}a(d\I_n-b'c)a'\in K_n\diag(\p_1^\mu)K_n,\,
    \p_1^{-\nu}(c'+d^{-1}ca')\in \o^n.$$
    Hence, $f(gk)=1$ and we conclude the proof.
\end{proof}

\begin{rmk}\label{rmk-ramified-weight-function}
  In any of the above cases, if $v$ is ramified, we change variables inside the
  zeta-integral as in \cite[proof of Lemma 2.1]{cogdellpiatetski-shapiro1994converse}
  and recall that we choose $W_1$ and $W_2$ to be shifted spherical vectors.
  This allows us to reduce to the unramified case. The price to pay is
  for all this change of variables is of the form $|\Delta|^{\mu(s_1,s_2)}$, where $\mu$ is an affine functions whose coefficients depend only on $n$.
\end{rmk}

\subsection{Archimedean places}

\begin{lem}\label{bound-spectral-archimedean-local-weight}
Let $v\mid\infty$ and $f$ be of type II. Then the local weight satisfies
$$h(\Pi)\ll_\epsilon \left(\frac{X}{C(\Pi)}\right)^{2n\sigma},$$
for any $0<\sigma<1/2-\vartheta$.
In particular, for fixed $X$ we have $h(\Pi)=o(1)$ as $C(\Pi)\to\infty$.
\end{lem}

To prove Lemma \ref{bound-spectral-archimedean-local-weight} we need a preparatory
lemma, which we describe below. The proofs of Lemma
\ref{bound-spectral-archimedean-local-weight} and Lemma
\ref{bounds-spectral-archimedean} have similarities with the proofs of
\cite[Theorem 8]{JaNe2019anv} and \cite[Proposition 7.1]{JaNe2019anv}, respectively.

First, let $U$ be a sufficiently small neighborhood of the identity element of
$G_n(F)$ and $\xi\in C_c^\infty(U)$ such that $\|\xi\|_{L^1(G_n(F))}=1$. We
define $V_\xi\in\Pi$ by
$$V_\xi\sbr{\mat{h&\\&1}}:=\int_{N_n(F)}\xi(nh)\overline{\psi(n)}\d n,\quad
h\in G_n(F).$$
Note the theory of the Kirillov model ensures that for a given $\xi$ there is a unique $V_\xi$ as above in $\Pi$; see \cite[Proposition 5]{jacquet2010distinction}.

\begin{lem}\label{bounds-spectral-archimedean}
Let $v\mid\infty$ and $f:=f_0\ast^{\rm{u}}\alpha$ be of type II. We have
\begin{multline*}
    \int_{N_n(F)\bs G_n(F)}W_1(h_1)\int_{N_n(F)\bs G_n(F)}\overline{W_2(h_2)}\hat{\alpha}(e_nh_2)\\
    \int_{\overline{G_{n+1}}(F)}f^0(g)V_\xi\sbr{\mat{h_1&\\&1}g\mat{h_2^{-1}&\\&1}}\d g\d h_2\d h_1\ll_\epsilon \left(\frac{X}{C(\Pi)}\right)^{2n\sigma}
\end{multline*}
for any $0<\sigma<1/2-\vartheta$, uniformly in $\xi$, as $\xi$ tends to the $\delta$-distribution at the identity.
\end{lem}

\begin{proof}
  We start by writing the $g$-integral in Bruhat coordinates \emph{i.e.},
  $$g:=\mat{a&b\\&1}\mat{\I_n&\\c&1};\quad a\in G_n(F), b,c^\top\in F^n;
  \quad\d g=\frac{\d a}{|\det(a)|}\d b\d c.$$
  Using the description of $f^0$ as in \S\ref{type-2-test-function},
  the unipotent equivariance of $V_\xi$, and calculating the $b$-integral,
  we write the main integral as
  \begin{multline*}
		X^n\int_{N_n(F)\bs G_n(F)}W_1(h_1)\hat{\Omega}_{12}(-e_nh_1)\int_{N_n(F)\bs
    G_n(F)}\overline{W_2(h_2)}\hat{\alpha}(e_nh_2)\\
		\int_{F^n}\Omega_{21}(cX) \int_{G_n(F)}\Omega_{11}(a)V_{\xi}
    \sbr{\mat{h_1a&\\c&1}\mat{h_2^{-1}&\\&1}}\frac{\d a}{|\det(a)|}\d c\d h_2\d h_1.
  \end{multline*}
	We now consider the $a,c$-integral above. Recalling that $\Omega_{11}=
  \Omega^1\ast\Omega^2$ and changing variables we write the above as
	\begin{equation}\label{g-integral}
		\int_{F^n}\Omega_{21}(cX)\int_{G_n(F)}\int_{G_n(F)}\Omega^1(a_1)\Omega^2(a_2)V_{\xi,h_2a_2}\sbr{\mat{h_1a_1&\\ca_2&1}}\frac{|\det(a_2)|}{|\det(a_1)|}\d a_1\d a_2\d c
	\end{equation}
	where $V_{\xi,h}:=V_\xi(\cdot\diag(h^{-1},1))$.
	Fix $0<\sigma<1/2-\vartheta$. We apply Whittaker--Plancherel formula for $|\det|^{-\sigma}V_{\xi,h_2a_2}
  \sbr{\mat{\cdot&\\&1}\mat{\I_n&\\ca_2&1}}$ 
  (see \emph{e.g.}, \cite[eq.(8)]{JaNu2021reciprocity}),
  $\GL(n+1)\times\GL(n)$ local functional equation (see \eqref{LFE-m-m-1}), and
  unipotent equivariance of $\tilde{V}_{\xi}$ to write $V_{\xi,h_2a_2}
  \sbr{\mat{h_1a_1&\\ca_2&1}}$ as
	\begin{multline*}
		\int_{\hat{G}_n(F)}\gamma(1/2+\sigma,\tilde{\Pi}\otimes\pi)\omega_\pi(-1)^n
    \sum_{W\in\B(\pi)}W(h_1a_1)|\det(h_1a_1)|^\sigma\\
		\int_{N_n(F)\bs G_n(F)}\tilde{V}_{\xi,h_2a_2}\sbr{\mat{h&\\&1}}
    \overline{\tilde{W}(h)}|\det(h)|^\sigma\psi(e_nha_2^\top c^\top)\d h\d\mu_{\loc}\pi.
	\end{multline*}
	Thus changing variable $h\mapsto C(\Pi)h$ we can write \eqref{g-integral} as
	\begin{multline}\label{g-integral-final}
    |\det(h_1)|^\sigma\int_{\hat{G}_n(F)}\omega_\pi((-1)^nC(\Pi))
    \gamma(1/2+\sigma,\tilde{\Pi}\otimes{\pi})C(\Pi)^{n\sigma}\\
		\sum_{W\in\B(\pi)}\int_{G_n(F)}\Omega^1(a_1)W(h_1a_1)
    |\det(a_1)|^{\sigma-1}{\d a_1}\\
		\int_{G_n(F)}\Omega^2(a_2)|\det(a_2)|\int_{N_n(F)\bs G_n(F)}
    \tilde{V}_{\xi,h_2a_2}\sbr{\mat{hC(\Pi)&\\&1}}\\
		\overline{\tilde{W}(h)}|\det(h)|^\sigma\hat{\Omega}_{21}
    \left(\frac{C(\Pi)}{X}e_nha_2^\top \right)\d h{\d a_2}\d\mu_{\loc}\pi.
	\end{multline}

We consider the $a_2$-integral first. Once again, fix $0<\sigma'<1/2-\vartheta$ sufficiently small, applying
  Whittaker--Plancherel formula and $\GL(n+1)\times\GL(n)$ local functional
  equation we write
	\begin{multline}\label{estimate-v-tilde}
    \int_{G_n(F)}\Omega^2(a_2)|\det(a_2)|\tilde{V}_{\xi,h_2a_2}
    \sbr{\mat{hC(\Pi)&\\&1}}\hat{\Omega}_{21}
    \left(\frac{C(\Pi)}{X}e_nha_2^\top \right){\d a_2}\\
		=\int_{\hat{G}_n(F)}\omega_{\pi'}((-1)^nC(\Pi))\gamma(1/2+\sigma',{\Pi}\otimes\pi')
    C(\Pi)^{n\sigma'}\sum_{W'\in\B(\pi')}W'(h)|\det(h)|^{\sigma'}\\
		\int_{G_n(F)}\Omega^2(a_2)|\det(a_2)|\hat{\Omega}_{21}
    \left(\frac{C(\Pi)}{X}e_nha_2^\top \right)\\
		\int_{N_n(F)\bs G_n(F)}{V}_{\xi,h_2a_2}\sbr{\mat{t&\\&1}}
    \overline{\tilde{W'}(t)}|\det(t)|^{\sigma'}\d t{\d a_2}\d\mu_{\loc}\pi'
	\end{multline}
	Using the definition of $V_{\xi,h}$, unfolding the $N_n$-integral, and
  changing variable $t\mapsto th_2a_2$ we write the two inner integrals above as
  $$\int_{G_n(F)}\int_{G_n(F)}\Omega^2(a_2){|\det(a_2)|}\hat{\Omega}_{21}
    \left(\frac{C(\Pi)}{X}e_nha_2^\top \right)\xi(t)
  \overline{\tilde{W'}(th_2a_2)}|\det(th_2a_2)|^{\sigma'}\d t{\d a_2}.$$
	We integrate by parts in the $a_2$-variable with respect to the differential
  operator $\mathfrak{D}$ as defined in \cite[\S3.4]{JaNu2021reciprocity} multiple
  times and apply \cite[Lemma 3.1]{JaNu2021reciprocity} to bound
  $\tilde{\pi}'(a_2)\tilde{W'}(th_2)$. Applying the compact support of $\Omega^2$
  and the using that $\hat{\Omega}_{21}$ is a Schwartz function,
  we bound the above by
	$$\ll S_{-L}(W')\int_{G_n(F)}\delta^{1/2-\eta}(th_2)|\det(th_2)|^{\sigma'}
  \xi(t)\d t \ll S_{-L}(W')\delta^{1/2-\eta}(h_2)|\det(h_2)|^{\sigma'}$$
	for any $L,\eta>0$ uniformly in $\frac{C(\Pi)}{X}e_nh$ and in $\xi$ as we
  shrink the support of $\xi$ to a delta distribution.

  Here, $S_d$ denotes the Sobolev norm as defined in \cite[\S3.4]{JaNu2021reciprocity}.
  In the last inequality above, we have used the fact that
  $t\in \operatorname{supp}(\xi)\implies \det(t)\asymp \delta(t)\asymp 1$ and
  $\|\xi\|_{L_1}=1$. Now to estimate \eqref{estimate-v-tilde} we first use this
  last estimate to bound the two innermost integrals. We apply the facts that
	$$W'(h)\ll S_d(W')\delta^{1/2-\eta}(h)$$ for some $d$ (follows from
  \cite[Lemma 3.1]{JaNu2021reciprocity}) and that
	$$\gamma(1/2+\sigma,\Pi\otimes\pi')\ll C(\Pi)^{-n\sigma'}C(\pi')^{O(1)}$$
	(follows from \cite[Lemma 3.1]{JaNe2019anv}). Finally, making $L$ sufficiently
  large and applying \cite[Lemma 3.2]{JaNu2021reciprocity} we bound
  \eqref{estimate-v-tilde} by
	$$\ll \delta^{1/2-\eta}(h)|\det(h)|^{\sigma'}\delta^{1/2-\eta}(h_2)
  |\det(h_2)|^{\sigma'}.$$
	Now we focus on \eqref{g-integral-final}. We first estimate the two innermost
  integrals. Integrating by parts we have
	$$\hat{\Omega}_{21}\left(\frac{C(\Pi)}{X}e_nha_2^\top \right)\ll_N
  (1+(C(\Pi)/X\|e_nh\|))^{-N}$$
	as $a_2$ varies over a compact set. Using this and the above estimate of
  \eqref{estimate-v-tilde} we bound the inner two integrals of
  \eqref{g-integral-final} by
	\begin{equation*}\delta^{1/2-\eta}(h_2)|\det(h_2)|^{\sigma'}
		\int_{N_n(F)\bs G_n(F)}\delta^{1/2-\eta}(h)|\det(h)|^{\sigma'+\sigma}|
    \tilde{W}(h)|\min(1,(C(\Pi)/X\|e_nh\|)^{-N})\d h.
	\end{equation*}
	Once again, applying \cite[Lemma 3.1]{JaNu2021reciprocity}
  we see that the
  above integral is absolutely convergent and is bounded by
	$$\ll S_{d}(W)\left(\frac{C(\Pi)}{X}\right)^{-n(\sigma+\sigma')}$$
	for some $d>0$. Now we proceed similarly as when we were estimating
  \eqref{estimate-v-tilde}. Integrating by parts the $a_1$-integral in
  \eqref{g-integral-final} with respect to $\mathfrak{D}$ sufficiently many times,
  bounding $\pi(a_1)W(h_1)\ll \delta^{1/2-\eta}(h_1)S_d(W)$ (as $a_1$
  varies over a compact set), using bounds of $\gamma$-factors, and applying
  \cite[Lemma 3.1]{JaNu2021reciprocity} we argue that \eqref{g-integral-final},
  and hence \eqref{g-integral}, are bounded by
	$$\ll |\det(h_1)|^\sigma\delta^{1/2-\eta}(h_1h_2)|\det(h_2)|^{\sigma'}
  \left(\frac{C(\Pi)}{X}\right)^{-n(\sigma+\sigma')}.$$
	Hence, we bound the main integral by
	\begin{multline*}
	  \left(\frac{C(\Pi)}{X}\right)^{-n(\sigma+\sigma')}\int_{N_n(F)\bs G_n(F)}
    |W_1(h_1)\hat{\Omega}_{12}(-e_nh_1)|\delta^{1/2-\eta}(h_1)
    |\det(h_1)|^{\sigma}\d h_1\\
		\int_{N_n(F)\bs G_n(F)}|W_2(h_2)\hat{\alpha}(e_nh_2)|\delta^{1/2-\eta}(h_2)
    |\det(h_2)|^{\sigma'}\d h_2.
	\end{multline*}
	Once again applying \cite[Lemma 3.1]{JaNu2021reciprocity} we see that the above
  two integrals converge for $\sigma,\sigma'>0$ as $\pi_{1},\pi_{2}$ are tempered.
  Choosing $\sigma=\sigma'$ and recalling their restriction we conclude.
\end{proof}

\begin{proof}[Proof of Lemma \ref{bound-spectral-archimedean-local-weight}]
	Changing basis we can write $h(\Pi)$ as
	\begin{multline*}
		\sum_{W\in\B(\Pi)}\int_{N_n(F)\bs G_n(F)}\int_{\overline{G_{n+1}(F)}}
    f^0(g)W\sbr{\mat{h_1&\\&1}g}W_1(h_1)\d g\d h_1\\
		\int_{N_n(F)\bs G_n(F)}\int_{F^n}\alpha(b)\overline{W\sbr{\mat{h_2&\\&1}
    \mat{\I_n&-b\\&1}}}\overline{W_2(h_2)}\d b\d h_2.
	\end{multline*}
	Using unipotent equivariance of $W$ and exchanging integrals we write the above as
	\begin{multline*}
		\sum_{W\in\B(\Pi)}\int_{N_n(F)\bs G_n(F)}W_1(h_1)\int_{N_n(F)\bs G_n(F)}
    \hat{\alpha}(-e_nh_2)\overline{W_2(h_2)}\\
		\int_{\overline{G_{n+1}(F)}}f^0(g)W\sbr{\mat{h_1&\\&1}g}
    \overline{W\sbr{\mat{h_2&\\&1}}}\d g\d h_2\d h_1.
	\end{multline*}
	Integrating by parts the $g$-integral we make sure the $W$-sum is absolutely convergent. The $h_1$ and $h_2$ integrals are also absolutely convergent which follows from the bounds of $W_i$ (see \cite[Lemma 3.1]{JaNu2021reciprocity}). Thus we may move the $W$-sum inside, change the basis to $\{\Pi(\diag(h_2,1))W\}_{W\in\B(\pi)}$, and change variables to write the above as
	\begin{equation*}
		\int_{N_n(F)\bs G_n(F)}W_1(h_1)\int_{N_n(F)\bs G_n(F)}\hat{\alpha}(e_nh_2)\overline{W_2(h_2)}
		\sum_{W\in\B(\Pi)}\Pi(F_{h_1,h_2})W(1)\overline{W(1)}\d h_2\d h_1,
	\end{equation*}
 where $F_{h_1,h_2}:=f^0\sbr{\mat{h_1^{-1}&\\&1}\cdot\mat{h_2&\\&1}}$.
 
	Let $\{\xi_j\}_j$ be a sequence of smooth $L^1$-normalized function on $G_n(F)$ supported on a sufficiently small neighborhood of the identity, such that the support of $\xi_j$ shrinks to the identity as $j\to\infty$. Thus the sequence $V_{\xi_j}$ tends to the delta mass at the identity in the dual of $C^\infty(N_n(F)\bs G_n(F),\psi)$ as $j\to\infty$ where $V_{\xi_j}$ is defined as in the statement of Lemma \ref{bounds-spectral-archimedean}. Arguing as in \cite[\S 7.1]{JaNe2019anv} (see also the proof of \cite[Theorem 8]{JaNe2019anv}) we conclude that the above display is the limit of
	\begin{equation*}
		\int_{N_n(F)\bs G_n(F)}W_1(h_1)\int_{N_n(F)\bs G_n(F)}\hat{\alpha}(e_nh_2)\overline{W_2(h_2)}\\
		\Pi(F_{h_1,h_2})V_{\xi_j}(1)\d h_2\d h_1
	\end{equation*}
 as $j\to\infty$. Changing variables and applying Lemma \ref{bounds-spectral-archimedean} we conclude.
\end{proof}

\begin{lem}\label{localization}
    Let $v\mid\infty$ and real, and $f$ be of type I. Also let $\pi_{1}\cong\pi_{2}$ and $W_{1}=W_{2}$. Then $h(\Pi)\ge 0$ for every unitary $\Pi$ and $h(\Pi)\gg 1$ if $C(\Pi)<X$.
\end{lem}

\begin{proof}
As $f=f^0\ast f^0$ we have
$$h(\Pi)=\sum_{W\in\B(\Pi)}|\Psi(1/2,\Pi(f^0)W,W_1)|^2\ge0,$$
which proves the first assertion.

To prove the second assertion we start with $\B(\Pi)$ containing an archimedean analytic newvector $W_0\in\Pi$, in the sense of \cite{JaNe2019anv}.     Thus in view of the above display dropping all terms but $W_0$ in the sum over $\B(\Pi)$ it suffices to show that
$$\Psi(1/2,\Pi(f^0)W_0,W_1)\gg 1,$$
if $C(\Pi)<X$. In the rest of the proof, the implied constants in $O$-notation will be assumed to be independent of $\tau$.

First, we record the relevant properties of an analytic newvector.
\begin{enumerate}
    \item We choose $W_0\mid_{G_n(F)}=\theta_1$ where $\theta_1\in C_c^\infty(N_n(F)\bs G_n(F),\psi)^{K_n}$ is as in \eqref{normalize-W-1}. Note that we have $\|W_0\|=1$. See the construction of an analytic newvector in \cite[\S4]{JaNe2019anv},
    \item Recalling that $\Pi$ is $\vartheta$-tempered we have
    $$\left|W_0\sbr{\mat{h&\\c/X&1}}-W_0\sbr{\mat{h&\\&1}}\right| \ll_\eta |\det(h)|^{-\vartheta}\delta^{1/2-\eta}\sbr{\mat{h&\\&1}}|c|,$$
    for $C(\Pi)< X$ and $|c|\ll 1$. This follows from \cite[Lemma 7.2]{Jana2020RS}.
    \item For every $\delta>0$ there is a $\tau>0$ such that for every $h\in\mathrm{supp}\left(W_0\mid_{G_n(F)}\right)$ we have
    $$\left|W_0\sbr{\mat{ha&\\&1}}-W_0\sbr{\mat{h&\\&1}}\right|<\delta$$
    whenever $\|a-\I_n\|<\tau$. This follows from \cite[Theorem 7]{JaNe2019anv}.
\end{enumerate}
Writing the definition and using Bruhat coordinates we write $\Psi(1/2,\Pi(f^0)W_0,W_1)$ as
\begin{multline*}
   X^n\int_{N_n(F)\bs G_n(F)}\int_{F^n}\int_{F^n}\int_{G_n(F)}\Omega_{11}(a)\Omega_{12}(b)\Omega_{21}(cX)\\W_0\sbr{\mat{h&\\&1}\mat{a&b\\&1}\mat{\I_n&\\c&1}}W_1(h)\frac{\d a}{|\det(a)|} \d b\d c\d h.
\end{multline*}
Changing variables and using unipotent equivariance we rewrite the above as
\begin{multline*}
   \int_{N_n(F)\bs G_n(F)}\int_{F^n}\int_{F^n}\int_{G_n(F)}\Omega_{11}(a)\Omega_{12}(b)\Omega_{21}(c)\psi(e_nhb)\\
   W_0\sbr{\mat{ha&\\c/X&1}}W_1(h)\frac{\d a}{|\det(a)|} \d b\d c\d h.
\end{multline*}
We write the above as
\begin{multline}\label{main-term-analytic-newvector-1}
   \int_{N_n(F)\bs G_n(F)}\int_{F^n}\int_{F^n}\int_{G_n(F)}\Omega_{11}(a)\Omega_{12}(b)\Omega_{21}(c)\psi(e_nhb)\\W_0\sbr{\mat{ha&\\&1}}W_1(h)\frac{\d a}{|\det(a)|} \d b\d c\d h+E_1,
\end{multline}
where $E_1$, using property (2) above, can be bounded by
\begin{multline*}
   \ll\int_{N_n(F)\bs G_n(F)}\left|\int_{F^n}\Omega_{12}(b)\psi(e_nhb)\d b\right|\int_{F^n}\Omega_{21}(c)|c|\d c\\
   \int_{G_n(F)}\Omega_{11}(a)
   |\det(ha)|^{-\vartheta}\delta^{1/2-\eta}\sbr{\mat{ha&\\&1}}\frac{\d a}{|\det(a)|}|W_1(h)|\d h
\end{multline*}
Integrating by parts with respect to $b$ in the resulting integral, and making
use of the support conditions (which, in particular, imply $|c|<\tau$ for $c\in \mathrm{supp}(\Omega_{21})$) and normalizations of $\Omega_{ij}$ we see that the above integral is bounded by
$$\ll \tau\int_{N_n(F)\bs G_n(F)}(1+\|e_nh\|)^{-N}|\det(h)|^{-\vartheta}\delta^{1/2-\eta}\sbr{\mat{h&\\&1}}|W_1(h)|\d h.$$
Applying \cite[Lemma 3.1]{JaNu2021reciprocity} we check that the above integral is absolutely convergent and is $O(1)$, consequently, we have $E_1=O(\tau)$.

Now we focus on \eqref{main-term-analytic-newvector-1}.
First, note that for $b\in\mathrm{supp}(\Omega_{12})$ we have
$$|\psi(e_nhb)-1| \le \|e_nh||b\| = O(\tau).$$
This follows as in the integral in \eqref{main-term-analytic-newvector-1}, $ha\in \mathrm{supp}\left(W_0\mid_{G_n(F)}\right)$ which is a fixed compact set (see property (1) above) and $a\in\mathrm{supp}(\Omega_{11})$. Thus $h$ varies in a fixed compact set, hence $e_nh=O(1)$. 
Thus, using normalization of $\Omega_{21}$, we can write the main term of \eqref{main-term-analytic-newvector-1} as
\begin{equation}\label{main-term-analytic-newvector-2}
  \int_{N_n(F)\bs G_n(F)}\int_{F^n}\int_{G_n(F)}\Omega_{11}(a)\Omega_{12}(b)\\W_0\sbr{\mat{ha&\\&1}}W_1(h)\frac{\d a}{|\det(a)|} \d b\d h+E_2,
\end{equation}
where $E_2$, using normalization of $\Omega_{12}$, can be bounded by
$$\ll\tau\int_{N_n(F)\bs G_n(F)}\int_{G_n(F)}\Omega_{11}(a)\left|W_0\sbr{\mat{ha&\\&1}}W_1(h)\right|\frac{\d a}{|\det(a)|} \d b\d c\d h.$$
As $h$ in the integral varies over a fixed compact set, the normalization of $\Omega_{11}$ implies that $E_2=O(\tau)$.
Thus, using normalization of $\Omega_{12}$, we can write the main term in \eqref{main-term-analytic-newvector-2} as
$$\int_{N_n(F)\bs G_n(F)}W_0\sbr{\mat{h&\\&1}}W_1(h)\d h \int_{G_n(F)}\Omega_{11}(a)\frac{\d a}{|\det(a)|}+E_3$$
where $E_3$, using property (3) above (after making $\tau$ smaller, if necessary) and noting that $h$ varies over a fixed compact set, can be bounded by $O(\tau)$.
Once again, from the normalization $\Omega_{11}$ and the construction of $W_0$ in property (1) above, the integral in the above display equals
$$\int_{N_n(F)\bs G_n(F)}\theta_1(h)W_1(h)\d h=1$$
which follows from \eqref{normalize-W-1}. All in all, we obtain
$$\Psi(1/2,\Pi(f^0)W_0,W_1)=1+O(\tau).$$
Making $\tau$ sufficiently small we conclude.
\end{proof}

\section{The main term}

In this section, we will analyze the main terms, namely, $\tilde{I}^+(s_1,s_2)$
defined in \eqref{Iv+tilde} and $I(n^-,s_1,s_2)$ defined in \eqref{Iv-tilde},
for various choices of $f$. Here, as in the previous section, we mostly work
under the hypothesis that all finite places are unramified and briefly mention
what needs to be changed to deal with the ramified ones.

\subsection{Computation of $\tilde{I}^+(s_1,s_2)$}

We recall the definition of $\hat{f}^u$, namely,
\begin{equation*}
\hat{f}^u(c,a):=\int_{F^n}f\sbr{\mat{a&b\\&1}}\overline{\psi(cb)}\d b,\quad c^\top\in F^n, a\in G_n(F).
\end{equation*}
which is the local component of \eqref{def-fu-hat}.

\subsubsection{Level places}

\begin{lem}\label{level-primes-I+}
  Let $v\mid\q$ and $f=\tilde{\1}_{\overline{K_0}(\p^e)}$. Then we have
    $$\tilde{I}^+(s_1,s_2)=\vol(\overline{K_0}(\p^e))^{-1}L(1+s_1+s_2,\pi_{1}\otimes\tilde{\pi}_{2})$$
    for all $s_1,s_2\in\C$.
\end{lem}
\begin{proof}
As $f=\vol(\overline{K_0}(\p^e))^{-1}\1_{\overline{K_0}(\p^e)}$ we readily have
(recall \eqref{def-fu-hat})
$$\hat{f}^u(c,y)=\vol(\overline{K_0}(\p^e))^{-1}\1_{K_n}(y)\1_{\o^n}(c).$$
As $W_2$ is spherical,
$$\tilde{I}^+(s_1,s_2)=\vol(\overline{K_0}(\p^e))^{-1}\Psi(1+s_1+s_2,W_1,\overline{W_2},\mathbf{1}_{\o^n}).$$
Applying \eqref{unramified-zeta-integral-m}, first for large positive $\Re(s_i)$, and
then by meromorphic continuation, we conclude.
\end{proof}

\subsubsection{Supercuspidal prime}

\begin{lem}\label{sc-prime-I+}
    Let $v=\p_0$ and $f$ be of type II. Then we have
    $$\tilde{I}^+(s_1,s_2)=1$$
    for all $s_1,s_2\in\C$.
\end{lem}

\begin{proof}
Unraveling the definitions we write
\begin{multline*}
\tilde{I}^+(s_1,s_2)=\int_{N_n(F) \backslash G_n(F)}\int_{G_n(F)}W_{1}(x)\overline{W_{2}(xy)}
\int_{F^n}f\sbr{\mat{\I_n&b\\&1}\mat{y&\\&1}}\overline{\psi(e_nxb)}\\ \d b\d_{s_2}y\d_{1+s_1+s_2}x.
\end{multline*}
Changing variables $y\mapsto yx^{-1}$ and $xb\mapsto b$ we obtain
\begin{multline*}
\tilde{I}^+(s_1,s_2)=\int_{N_n(F) \backslash G_n(F)}\int_{G_n(F)}W_{1}(x)\overline{W_{2}(y)}
\int_{F^n}f\sbr{\mat{x^{-1}&\\&1}\mat{\I_n&b\\&1}\mat{y&\\&1}}\overline{\psi(e_nb)}\\ \d b\d_{s_2}y\d_{s_1}x.
\end{multline*}
We fold the integral over $y\in G_n(F)$ as integrals over $n'\in N_n(F)$ and $y\in N_n(F)\bs G_n(F)$. Then putting $n:=\mat{n'&b\\&1}\in N_{n+1}(F)$ and using unipotent equivariance of $W_2$ we write the above integral as
$$\int_{N_n(F) \backslash G_n(F)}\int_{N_n(F)\bs G_n(F)}W_{1}(x)\overline{W_{2}(y)}\int_{N_{n+1}(F)}f\sbr{\mat{x^{-1}&\\&1}n\mat{y&\\&1}}\overline{\psi(n)} \d n\d_{s_2}y\d_{s_1}x$$
Recalling the definition of $f$ and using unitarity of $\sigma$ we see that
\begin{align*}
    f\sbr{\mat{x^{-1}&\\&1}n\mat{y&\\&1}}
    &=\langle \sigma(\diag(x^{-1},1)n\diag(y,1))W_\sigma,W_\sigma\rangle\\
    &=\langle \sigma(n)\sigma(\diag(y,1))W_\sigma,\sigma(\diag(x,1))W_\sigma\rangle.
\end{align*}
Using \cite[Lemma 6.1]{jana2021applications} we compute
$$\int_{N_{n+1}(F)}f\sbr{\mat{x^{-1}&\\&1}n\mat{y&\\&1}}\overline{\psi(n)} \d n=W_\sigma(\diag(y,1))\overline{W_\sigma(\diag(x,1))}.$$
We then deduce for large $\Re(s_i)$ and then by meromorphic continuation that
\begin{align}\label{I+=PsiPsi}
\tilde{I}^+(s_1,s_2) & =\int_{N_n(F) \backslash G_n(F)} \overline{W_{\sigma}
\sbr{\mat{x&\\&1}}}W_1(x)\d_{s_1} x\int_{N_n(F) \backslash G_n(F)} W_{\sigma}
\sbr{\mat{y&\\&1}}\overline{W_{2}(y)}\d_{s_2} y\notag\\
&=\Psi(1/2+s_1,\overline{W_{\sigma}},W_{1})\Psi(1/2+s_2,W_{\sigma},\overline{W_{2}}),
\end{align}
for all $s_1,s_2$.
Applying \eqref{unramified-zeta-integral-m-1} the above evaluates to
$$L(1/2+s_1,\tilde{\sigma}\otimes \pi_{1})L(1/2+s_2,\sigma\otimes  \tilde{\pi}_{2}).$$
We conclude from the fact that $\pi_i$ are unramified and $\sigma$ is supercuspidal.
\end{proof}

\subsubsection{Auxiliary prime}

\begin{lem}\label{aux+I}
    Let $v=\p_1$ and $f$ be of type II. Then we have
    $$\tilde{I}^+(s_1,s_2)=\overline{\lambda_{\pi_2}(\p_1^\mu)}N(\p_1)^{-|\mu|s_2}
    {L(1+s_1+s_2,\pi_1\otimes\tilde{\pi}_2)}$$
    for all $s_1,s_2\in\C$.
\end{lem}

\begin{proof} 
    Recalling the expression of $f$ (see \S \ref{subsubsec:auxiliary}),
    applying the fact that
    $$\int_{G_n(F)}\1_{K_n\diag(\p_1^\mu)K_n}(y)W_2(xy)\d y=\lambda_{\pi_2}
    (\p_1^\mu)W_2(x),$$
    and that $\hat{\1}_{\o^n}=\1_{\o^n}$ we obtain
    $$\tilde{I}^+(s_1,s_2)=\overline{\lambda_{\pi_2}(\p_1^\mu)}
    N(\p_1)^{-|\mu|s_2}\int_{N_n(F)\bs G_n(F)}W_1(x)\overline{W_2(x)}
    \1_{\o^n}(e_nx)\d_{1+s_1+s_2} x.$$
    We conclude by applying \eqref{unramified-zeta-integral-m} for large
    $\Re(s_i)$ and then by meromorphic continuation.
\end{proof}

\begin{rmk}\label{remark-ramified-I+}
When $v$ is ramified, we change variables inside the zeta integral in \eqref{Iv+tilde}
in order to account for the fact that we now have a ramified additive characters and
that $W_1$ and $W_2$ are now shifted spherical vectors. We do not require very
precise informations about the correction factor coming from this change of
variables, except that it has the form $\Delta^{\mu^+(s_1,s_2)}$, where
$\mu^+$ is an affine function whose coefficients only depend on $n$.
\end{rmk}

\subsubsection{Archimedean places}

\begin{lem}\label{archimedean-I+type-1}
    Let $v\mid\infty$ and $f$ be of type I. Then $\tilde{I}^+(s_1,s_2)$ is holomorphic at a neighbourhood of $(0,0)$ and we have $$\partial_{s=0}\tilde{I}^+(s,0)=O(X^n).$$
    Moreover, if $\pi_1=\pi_2$ and $W_1=W_2$ then we have
    $$\tilde{I}^+(0,0)=\langle W_1,\pi_1(f\mid_{G_n(F)}) W_1\rangle=X^nI_0$$
    where $I_0$ is a constant independent of $X$.
\end{lem}

\begin{proof}
Note that the support condition of $f$ ensures that
$\hat{f}^u(c,y)$ is a Schwartz function in $c\in F^n$ and compactly supported
in $y\in G_n(F)$. Thus the zeta integral in \eqref{Iv+tilde} is absolutely
convergent for $\Re(s_1+s_2)>-1$ (by \cite[Lemma 3.1]{JaNu2021reciprocity})
and by compactness of the $y$-integral, $\tilde{I}^+(s_1,s_2)$ defines a
holomorphic function in the same region. We write
\begin{equation*}
\tilde{I}^+(s,0)=\int_{G_n(F)}\int_{N_n(F)\backslash G_n(F)}W_{1}(x)\overline{W_2(xy)}\hat{f}^u(e_nx,y)|\det(x)|^{1+s}\d x \d y.
\end{equation*}
Thus it follows that
$$\partial_{s=0}\tilde{I}^+(s,0)=\int_{G_n(F)}\int_{N_n(F)\backslash G_n(F)}W_{1}(x)\overline{W_{2}(xy)}\hat{f}^u(e_nx,y)|\det(x)| \log|\det(x)|\d x\d y.$$
Integrating by parts and applying the majorant property (2) of $f$ we see that
$$\hat{f}^u(e_nx,y)\ll_N X^n(1+|e_nx|)^{-N}\text{ uniformly in $y$},\quad y\asymp \I_n.$$
Applying 
$$\log|\det(x)|\ll \max(|\det(x)|^\epsilon,|\det(x)|^{-\epsilon})$$ and \cite[Lemma 3.1]{JaNu2021reciprocity} we see that the above integral is absolutely convergent and is $O(X^n)$.

Now assume that $\pi_1\cong\pi_2$ and $W_1=W_2$, and plug-in $s=0$ in
\eqref{Iv+tilde}. We use Bruhat coordinates
$$x=z\mat{h&\\&1}\mat{\I_{n-1}&\\c&1};\quad \d x= \d^\times z \frac{\d h}{|\det(h)|}\d c;$$
$$ z\in F^\times, h\in N_{n-1}(F)\bs G_{n-1}(F), c\in F^{n-1};$$
to rewrite the $x$-integral in $\tilde{I}^+(0,0)$ as
\begin{multline*}
\int_{F^{n-1}}\int_{N_{n-1}(F)\backslash G_{n-1}(F)}W_{1}\sbr{\mat{h&\\&1}\mat{\I_{n-1}&\\c&1}}
\overline{W_1\sbr{\mat{h&\\&1}\mat{\I_{n-1}&\\c&1}y}}\d h\\
\int_{F^\times}\hat{f}^u(z(c,1),y)|z|^n\d^\times z\d c.
\end{multline*}
The above $h$-integral defines a $G_n(F)$-invariant inner product on $\pi_1$ and thus equals
$$\int_{N_{n-1}(F)\backslash G_{n-1}(F)}W_{1}\sbr{\mat{h&\\&1}}\overline{W_1\sbr{\mat{h&\\&1}y}}\d h=\langle W_1,\pi_1(y)W_1\rangle,$$
which is $c$-independent. Thus we can move the $c$-integral inside. After changing variable $c\mapsto c/z$ and taking Fourier inversion we calculate the $(c,z)$-integral
(up to a constant that depends only on whether $F_v$ is areal or complex) as
$$\int_{F^n} \hat{f}^u(c',y)\d c'=f(\diag(y,1)).$$
Thus the first expression of $\tilde{I}^+(0,0)$ follows.

To prove the second expression it suffices to prove that $f\mid_{G_n(F)}=X^nf_1$ where $f_1\in C_c^\infty(G_n(F))$ is independent of $X$. In other words,
$$I_0=\langle W_1, \pi_1(f_1)W_1\rangle.$$
Note that for $a\in G_n(F)$ and $b\in F^n$ we can write $f\sbr{\mat{a&\\&1}}$ as
\begin{align*}
&\int_{G_n(F)}\int_{F^n}\int_{F^n}f^0\sbr{\mat{aa'&ab'\\&1}\mat{\I_n&\\c'&1}}f^0\sbr{\mat{a'&b'\\&1}\mat{\I_n&\\c'&1}}\d c'\d b'\frac{\d a'}{|\det(a')|}\\
&=X^{2n}\int_{G_n(F)}\Omega_{11}(aa')\Omega_{11}(a')\frac{\d a'}{|\det(a')|}\int_{F^n}\Omega_{12}(ab')\Omega_{12}(b')\d b'\int_{F^n}\Omega_{21}(c'X)^2\d c'.
\end{align*}
Changing variable $c'\mapsto c'/X$ and defining
$$f_1(a):=\int_{G_n(F)}\Omega_{11}(aa')\Omega_{11}(a')\frac{\d a'}{|\det(a')|}\int_{F^n}\Omega_{12}(ab')\Omega_{12}(b')\d b'\int_{F^n}\Omega_{21}(c')^2\d c'$$
we conclude the proof.
\end{proof}

\begin{rmk}
    We will later see that $I_0$ appears as the constant in the leading term in
    the asymptotic expansion of the second moment such as Theorem
    \ref{second-moment-nonarch} and Theorem \ref{second-moment-arch}. The vector
    $W_1$ and the test function $f$ can be further modified in a way such that
    $I_0$ will be a non-zero constant. However, we do not need this fact in this
    paper so we do not give a proof here.
\end{rmk}

Next, we analyze $\tilde{I}^+(0,0)$ for type II test function $f$. Recall from \S\ref{type-2-test-function} that in this case $f=f^0\ast^{\rm{u}}\alpha$.

\begin{lem}\label{archimedean-I+type-2}
    Let $v\mid\infty$ and $f$ be of type II. Then $\tilde{I}^+(s_1,s_2)$ is holomorphic at a neighbourhood of $(0,0)$. Moreover, there is a choice of $\alpha$ independent of $X$ so that
    $$\tilde{I}^+(0,0)\asymp X^n$$
    as $X\to\infty$.
\end{lem}

We first need a few auxiliary lemmata to prove Lemma \ref{archimedean-I+type-2}.

\begin{lem}\label{archimedean-I+type-2-f0}
Let $v\mid \infty$ and $f$ be of type II. Then we have
\begin{equation*}
\tilde{I}^+(0,0,f^0)=X^n\overline{\lambda_{\pi_2}(\Omega_{11})}\Omega_{12}(0)\Omega_{21}(0)
\frac{L(1,\pi_1\otimes\widetilde{\pi}_2)}{\zeta(n)}.
\end{equation*}
In particular, $\tilde{I}^+(0,0,f^0)\asymp X^n$.
\end{lem}

\begin{proof}
The proof is similar to but longer than the proof of Lemma \ref{aux+I}. Indeed, 
recalling the expression of $f^0$ and using that 
\[\int_{G_n(F)}W_{2}(xy)\Omega_{11}(y)\d y=\lambda_{\pi_2}(\Omega_{11})W_2(x),\]
we may write $\tilde{I}^+(0,0,f^0)$ as
\[X^n\Omega_{21}(0)\overline{\lambda_{\pi_2}(\Omega_{11})}\int_{N_n(F)\bs G_n(F)}
W_{1}(x)\overline{W_2(x)}\hat{\Omega}_{12}(e_nx)\det(x)\d x.\]
Now using Iwasawa coordinates
$$x=z\diag(a,1)k;\quad z\in F^\times, a\in A_{n-1}(F), k\in K_n;\quad \d x=
\d^\times z\frac{\d^\times a}{\delta(\diag(a,1))};$$
and sphericality of $W_i$ we obtain
\begin{multline*}
\tilde{I}^+(0,0,f^0)=X^n\Omega_{21}(0)\overline{\lambda_{\pi_2}(\Omega_{11})}
\int_{K_n}\int_{F^\times}\hat{\Omega}_{12}(ze_nk)|z|^n\d^\times z\d k\\
\int_{A_{n-1}(F)} W_{1}(\diag(a,1))\overline{W_2(\diag(a,1))}
\frac{\d^\times a}{\delta(a)}.
\end{multline*}
Now by Stade's formula \cite{stade2002archimedean} for the unramified computation
for the $\GL(n)\times \GL(n)$ Rankin--Selberg integrals and working as above we
compute that
\begin{align*}
    L(1,\pi_1\otimes\tilde{\pi}_2)
    &=\int_{N_n(F)\bs G_n(F)}W_1(x)\overline{W_2(x)}\exp(-\pi|e_nx|^2)|\det(x)|\d x\\
    &=\int_{A_{n-1}(F)} W_{1}(\diag(a,1))\overline{W_2(\diag(a,1))}\frac{\d^\times a}{\delta(a)}\int_{F^\times}e^{-\pi |z|^2}|z|^n\d^\times z
\end{align*}
Noting that the last $z$-integral evaluates to $\zeta(n)$ we conclude that the $a$-integral in the last expression of $\tilde{I}^+(0,0,f^0)$ equals
$$\frac{L(1,\pi_1\otimes\widetilde{\pi}_2)}{\zeta(n)}.$$
Now realizing the integral over $K_n\times F^\times$ in the last expression of $\tilde{I}^+(0,0,f^0)$ as the integral over $F^n$ in polar coordinates we evaluate (up to a
constant depending only on whether $F_v$ is real or complex) 
$$\int_{K_n}\int_{F^\times}\hat{\Omega}_{12}(ze_nk)|z|^n\d^\times z\d k=\int_{F^n}\hat{\Omega}_{12}(b)\d b=\Omega_{12}(0)$$
by Fourier inversion and conclude. The asymptotic of $\tilde{I}^+(0,0)$ follows from the assumptions on $\Omega_{ij}$ and the fact that the $L$-functions are non-vanishing.
\end{proof}

\begin{lem}\label{lem:alpha-to-delta}
    Let $v\mid\infty$ and $f$ be of type II. Let $\epsilon>0$. Then there exists an $\alpha\in C_c^\infty(F^n)$ as above such that for any multi-indices $j,k$ we have $$\partial^j_b\partial^k_c\left(f-f^0\right)\sbr{\mat{\I_n&b\\&1}\mat{y&\\&1}\mat{\I_n&\\c&1}}\ll_{j,k}\epsilon X^{n+|k|}$$
    for $y\in G_n(F)$ and $b,c^\top\in F^n$ varying on fixed compact sets.
\end{lem}

\begin{proof}
Note that as $\alpha$ tends to the $\delta$-distribution at the origin of $F^n$, the function $f$ tends to $f^0$, uniformly in $y,b,c$ as they vary over fixed compact sets.
 
The lemma follows by choosing the test function $\alpha$ sufficiently close to the $\delta$-distribution and applying the estimates of the derivatives of $f,f^0$ from majorant property (2); see \S\ref{type-2-test-function}.
\end{proof}

\begin{lem}\label{error-main-term-I+}
Let $v\mid\infty$ and $f$ be of type II. We have
$$|\tilde{I}^+(0,0;f)-\tilde{I}^+(0,0;f^0)|<\epsilon X^n$$
where $\epsilon$ and $\alpha$ are as in Lemma \ref{lem:alpha-to-delta}. 
\end{lem}

\begin{proof}
First, as $\alpha$ has a fixed compact support on $F^n$ and $f^0$ has a fixed compact support in $\overline{G_{n+1}}(F)$, we have $f\sbr{\mat{y&b\\&1}}$ has fixed compact supports in $G_n(F)$ and $F^n$ as functions of $y$ and $b$, respectively. So the same is true for $(f-f_0)\sbr{\mat{y&b\\&1}}$. 

Now we write
\begin{multline*}
\tilde{I}^+(0,0,f)-\tilde{I}^+(0,0,f^0)=\int_{\G_n(F)}\int_{N_n(F)\bs G_n(F)}{W}_{1}(x)\overline{{W}_{2}(xy)}\\
\left(\int_{F^n}(f-f^0)\mat{y&b\\&1}\overline{\psi(e_nx b)}\d b\right)|\det(x)|\d x\d y.
\end{multline*}
We integrate by parts the $b$-integral and apply Lemma \ref{lem:alpha-to-delta} to bound the innermost integral by
$\ll_N \epsilon X^n (1+|e_nx|)^{-N}$ uniformly in $y$ varying over a fixed compact set. Applying this bound and the bounds of Whittaker functions ${W}_i$ from \cite[Lemma 3.1]{JaNu2021reciprocity} we see that the above integrals are absolutely convergent for a sufficiently large $N$ and is $O_{N,\pi_{1},\pi_{2}}(\epsilon X^n)$. Modifying $\epsilon$ (and $\alpha$) we conclude the proof.
\end{proof}

\begin{proof}[Proof of Lemma \ref{archimedean-I+type-2}]
From the proof of Lemma \ref{error-main-term-I+} it follows that $\hat{f}^u(c,y)$ is compactly supported in $y\in G_n(F)$ and Schwartz in $c\in F^n$. Thus the same argument as in the proof of Lemma \ref{archimedean-I+type-1} yields the first assertion.

The second assertion immediately follows from Lemma \ref{error-main-term-I+}, Lemma \ref{archimedean-I+type-2-f0}, and taking $\epsilon$ sufficiently small.
\end{proof}

\subsection{Computation of $I^-(s_1,s_2)$}

We recall the definition of $\hat{f}^l$, namely,
\begin{equation*}
\hat{f}^l(b,a):=\int_{F^n}f\sbr{\mat{a&\\c&1}}\overline{\psi(cb)}\d c,\quad b\in F^n, a\in G_n(F).
\end{equation*}
which is the local component of \eqref{def-fl-hat}.

\subsubsection{Level primes}
\begin{lem}\label{level-primes-I-}
    Let $f=\tilde{\1}_{\overline{K_0}(\p^e)}$. Then we have
    $${I}^-(s_1,s_2)=\vol(\overline{K_0}(\p^e))^{-1}N(\p)^{-en(s_1+s_2)}
    L(1-s_1-s_2,\tilde{\pi}_{1}\otimes{\pi}_{2})$$
    for all $s_1,s_2\in\C$.
\end{lem}

\begin{proof}
As $f=\tilde{\1}_{\overline{K_0}(\p^e)}$, it follows from
\eqref{def-fl-hat} and using the fact that 
$$\hat{\1_{\o^n}(\p^{-e}\cdot)}=N(\p)^{-en}\1_{\o^n}(\p^e\cdot)$$
we readily have
$$\hat{f}^l(b,x)=\vol(\overline{K_0}(\p^e))^{-1}N(\p)^{-en}\1_{K_n}(x)\1_{\o^n}(\p^eb).$$
Thus for large negative $\Re(s_i)$ we have
\begin{align*}
&{I}^-(s_1,s_2)
=\vol(\overline{K_0}(\p^e))^{-1}N(\p)^{-en}\Psi(1-s_1-s_2,\tilde{W}_1,\overline{\tilde{W}_2},\1_{\o^n}(\p^e\cdot))\\
&=\vol(\overline{K_0}(\p^e))^{-1}N(\p)^{-en}\int_{N_n(F)\bs G_n(F)}\tilde{W}_1(y)\overline{\tilde{W}_2(y)}\1_{\o^n}(\p^ee_ny)|\det(y)|^{1-s_1-s_2}\d y.\\
&=\vol(\overline{K_0}(\p^e))^{-1}N(\p)^{-en(s_1+s_2)}\int_{N_n(F)\bs G_n(F)}\tilde{W}_1(y)\overline{\tilde{W}_2(y)}\1_{\o^n}(e_ny)|\det(y)|^{1-s_1-s_2}\d y.
\end{align*}
The last equality follows from the change of variables $y\mapsto \p^{-e}y$.
Applying \eqref{unramified-zeta-integral-m} and uniqueness of meromorphic
continuation we conclude.
\end{proof}

\subsubsection{Supercuspidal prime}

\begin{lem}\label{sc-prime-I-}
    Let $v=\p_0$ and $f$ be of type II. Recall the supercupsidal representation $\sigma$ underlying the definition of $f$. Then we have
    $${I}^-(s_1,s_2)=C(\sigma)^{-n(s_1+s_2)}$$
    for all $s_1,s_2\in\C$.
\end{lem}

\begin{proof}
  Unraveling the expression in \eqref{Iv-tilde} we have
\begin{equation*}
I^-(s_1,s_2)=\int_{G_n(F)}\int_{N_n(F)\backslash G_n(F)}
\tilde{W}_{1}(yx^\top)\overline{\tilde{W}_{2}(y)}
\int_{F^n}f\mat{x&\\c&1}\overline{\psi(e_nyc^\top)}
\d c\d_{1-s_1-s_2}y\d_{-s_1}x.
\end{equation*}
Using that in an unitary representation the map $W\mapsto \tilde{W}$ is an isometry, it follows that for all $g$,
\begin{equation*}
f(g)=\IP{\sigma(g)W_\sigma,W_\sigma}=\IP{\tilde{\sigma}(g^{-\top})\tilde{W}_\sigma,\tilde{W}_\sigma}=:\tilde{f}(g^{-\top}),
\end{equation*}
where $\tilde{f}$ denotes the matrix coefficient of $\tilde{W}_\sigma\in\tilde{\sigma}$.
Thus we can write the above integral as
$$\int_{G_n(F)}\int_{N_n(F)\backslash G_n(F)}
\tilde{W}_{1}(yx^\top)\overline{\tilde{W}_{2}(y)}\\
\int_{F^n}\tilde{f}\mat{x^{-\top}&-x^{-\top}c^\top\\&1}\overline{\psi(e_nyc^\top)}
\d c\d_{1-s_1-s_2}y\d_{-s_1}x.$$
Changing variables $c\mapsto -cx$, $x\mapsto x^{-\top}$, and $y\mapsto yx$
we write the above as
$$\int_{G_n(F)}\int_{N_n(F)\backslash G_n(F)}
\tilde{W}_{1}(y)\overline{\tilde{W}_{2}(yx)}\\
\int_{F^n}\tilde{f}\mat{x&c^\top\\&1}{\psi(e_nyc^\top)}
\d c\d_{1-s_1-s_2}y\d_{-s_2}x$$
A similar computation as in the proof of Lemma \ref{sc-prime-I+} leads to
\begin{align*}
I^-(s_1,s_2) & =\Psi(1/2-s_1,\overline{\tilde{W}_{\sigma}},\tilde{W}_{1})\Psi(1/2-s_2,\tilde{W}_{\sigma},\overline{\tilde{W}_{2}}) \\
&=\gamma(1/2+s_1,\tilde{\sigma}\otimes\pi_{1})\gamma(1/2+s_2,\sigma\otimes\tilde{\pi}_{2})\tilde{I}^+(s_1,s_2).
\end{align*}
The last equality follows from \eqref{LFE-m-m-1} and \eqref{I+=PsiPsi}.
As $\sigma$ is supercuspidal and $\pi_{i}$ are unramified with trivial central
characters, we have
$$\gamma(1/2+s_i,\sigma\otimes\tilde{\pi}_{i})=C(\sigma)^{-s_i}\epsilon(1/2,\sigma\otimes\tilde{\pi}_{i})=C(\sigma)^{-s_i}\epsilon(1/2,\sigma)^n.$$
Applying that $\epsilon(1/2,\tilde{\sigma})=\epsilon(1/2,\sigma)^{-1}$ we conclude
by meromorphic continuation.
\end{proof}

\subsubsection{Auxiliary prime}

\begin{lem}\label{aux-I}
    Let $v=\p_1$ and $f$ be of type II. Then we have
    $$I^-(s_1,s_2)=\lambda_{\tilde{\pi}_1}(\p_1^\mu)N(\p_1)^{|\mu|s_1-n\nu(s_1+s_2)}
    {L(1-s_1-s_2,\tilde{\pi}_1\otimes{\pi}_2)}$$
    for all $s_1,s_2\in\C$.
\end{lem}

\begin{proof}
    The proof is quite similar to the proof of Lemma \ref{aux+I}. Indeed,
    observing that
    $$\int_{G_n(F)}\1_{K_n\diag(\p_1^\mu)K_n}(x)\tilde{W}_1(yx^\top)\d x=\int_{G_n(F)}\1_{K_n\diag(\p_1^\mu)K_n}(x)\tilde{W}_1(yx)\d x=\lambda_{\tilde{\pi}_1}(\p_1^\mu)\tilde{W}_1(y),$$
    and that $\hat{\1_{\o^n}(\p_1^{-\nu}\cdot)}=N(\p_1)^{-n\nu}\1_{\o^n}
    (\p_1^{\nu}\cdot)$ we obtain
    $$I^-(s_1,s_2)=N(\p_1)^{|\mu|s_1-n\nu}\lambda_{\tilde{\pi}_1}(\p_1^\mu)
    \int_{N_n(F)\bs G_n(F)}\tilde{W}_1(y)\overline{\tilde{W}_2(y)}
    \1_{\o^n}(\p_1^{\nu}e_ny)\d_{1-s_1-s_2} y.$$
    We change variable $y\mapsto\p_1^{-\nu} y$ to obtain the above equals
    $$\lambda_{\tilde{\pi}_1}(\p_1^\mu)N(\p)^{|\mu|s_1-n\nu(s_1+s_2)}
    \int_{N_n(F)\bs G_n(F)}\tilde{W}_1(y)\overline{\tilde{W}_2(y)}
    \1_{\o^n}(e_ny)\d_{1-s_1-s_2} y.$$
    We conclude by applying \eqref{unramified-zeta-integral-m} for sufficiently
    negative $\Re(s_i)$ and then by meromorphic continuation.
\end{proof}

\begin{rmk}\label{remark-ramified-I-}
When $v$ is ramified, we proceed exactly as in Remark \ref{remark-ramified-I+},
changing variables inside the zeta integral. This brings in a correction factor of the shape $\Delta^{\mu^-(s_1,s_2)}$, where, as before, $\mu^-$
is an affine function whose coefficients only depend on $n$. Moreover, a careful
analysis of the correction term reveals that $\mu^+(0,0)=\mu^-(0,0)$.
\end{rmk}

\subsubsection{Archimedean place}

\begin{lem}\label{archimedean-I-type-1}
    Let $v\mid\infty$ and $f$ be of type I. Then $\tilde{I}^0(n^-,s_1,s_2)$ is holomorphic at a neighbourhood of $(0,0)$. Moreover, assume that ${\pi}_1\cong{\pi}_2$ and $W_1=W_2$. Then we have
    $$I^-(0,0)=\tilde{I}^+(0,0)=X^nI_0$$
    and
    $$\partial_{s=0}I^-(0,s)=-nI_0 X^n\log X +O(X^n)$$
    where $I_0$ is as in Lemma \ref{archimedean-I+type-1}.
\end{lem}

\begin{proof}
Arguing analogously to as in the proof of Lemma \ref{archimedean-I+type-1},
we may deduce that $I^-(s_1,s_2)$ defines a holomorphic function in the region
$\Re(s_1+s_2)<1$.

Recalling \eqref{I+=PsiPsi}, we obtain that $\partial_{s=0}I^-(0,s)$ equals
$$-\int_{G_n(F)}\int_{N_n(F)\backslash G_n(F)}\tilde{W}_{1}(yx^\top) \overline{\tilde{W}_1(y)}\hat{f}^l(e_n y,x)|\det(y)|\log|\det(y)|\d y \d x.$$
We change variable $y\mapsto yX$ and write $\partial_{s=0}I^-(0,s)$ as
\begin{multline*}
-nX^n\log X\int_{G_n(F)}\int_{N_n(F)\backslash G_n(F)}\tilde{W}_{1}(yx^\top)
\overline{\tilde{W}_{1}(y)}\hat{f}^l(e_n yX,x)
|\det(y)|\d y \d x\\
-X^n\int_{G_n(F)}\int_{N_n(F)\backslash G_n(F)}\tilde{W}_{1}(yx^\top)
\overline{\tilde{W}_{1}(y)}\hat{f}^l(e_n yX,x)
|\det(y)|\log |\det(y)|\d y \d x.
\end{multline*}
Working as in the proof of Lemma \ref{archimedean-I+type-1} the first summand above can be written as
\begin{multline*}
-nX^n\log X\int_{G_n(F)}\int_{N_{n-1}(F)\backslash G_{n-1}(F)}\tilde{W}_{1}\sbr{\mat{h&\\&1}x^\top}
\overline{\tilde{W}_{1}\sbr{\mat{h&\\&1}}}\d h\\
\int_{F^{n-1}}\int_{F^\times}\hat{f}^l((c,1)zX,x)|z|^n\d^\times z\d c \d x
\end{multline*}
Changing variables $z\mapsto z/X$ and applying Fourier inversion the $(c,z)$-integral above evaluates to $X^{-n}f(\diag(x,1))$. Using the isometric property of the contragredient map we see the $h$-integral above equals
$$\langle\tilde{\pi}_1(x^\top)\tilde{W}_1,\tilde{W}_1\rangle=\langle\pi_1(x^{-1})W_1,W_1\rangle=\langle W_1,\pi_1(x)W_1\rangle.$$
Thus we evaluate the first summand in the last expression of $\partial_{s=0}I^-(0,s)$ as
$$-n\log X\langle W_1,\pi_1(f\mid_{G_n(F)})W_1\rangle=-nI_0X^n\log X.$$
The last equality follows from Lemma \ref{archimedean-I+type-1}.

For the second summand in the last expression of $\partial_{s=0}I^-(0,s)$,
we first claim that the function
$$(c,x)\mapsto X^n\hat{f}^l(cX,x)$$
is $\ll_N X^n(1+|c|)^{-N}$ and is compactly supported in $x$.
Indeed, we note that
$$X^n\hat{f}^l(cX,x)=\int_{F^{n-1}}f\sbr{\mat{x&\\c'/X&1}}
\overline{\psi(c'c^\top)}\d c'.$$
Integrating by parts in $c'$ and applying majorant property (2) in
\S\ref{type-1-test-function} yields the claim.
Now, using the fact that $\log|\det(y)|\ll \max(|\det(y)|^\epsilon,
|\det(y)|^{-\epsilon})$, and applying the Whittaker bounds from
\cite[Lemma 3.1]{JaNu2021reciprocity} we see that the second summand is
absolutely convergent and is $O(X^n)$. This proves the second assertion.

On the other hand, we have
$$I^-(0,0)=\int_{G_n(F)}\int_{N_n(F)\backslash G_n(F)}\tilde{W}_{1}(yx^\top) \overline{\tilde{W}_1(y)}\hat{f}^l(e_n y,x)|\det(y)|\d y\d x.$$
Following the same computations as above we deduce the first assertion.
\end{proof}

Next, we analyze ${I}^-(0,0)$ for a type II test function $f$. Recall from \S\ref{type-2-test-function} that in this case $f=f^0\ast^{\rm{u}}\alpha$.

\begin{lem}\label{archimedean-I-type-2}
    Let $v\mid\infty$ and $f$ be of type II. Then $I^-(s_1,s_2)$ is holomorphic at a neighbourhood of $(0,0)$. Moreover, there is a choice of $\alpha$ independent of $X$ so that
    $${I}^-(0,0)\asymp X^n$$
    as $X\to\infty$.
\end{lem}

We first need a few auxiliary lemmata to prove Lemma \ref{archimedean-I-type-2}.

\begin{lem}\label{archimedean-I-type-2-f0}
Let $v\mid\infty$ and $f$ be type II. Then we have
\begin{equation*}
{I}^-(0,0,f^0)=X^n\Omega_{12}(0)\Omega_{21}(0)\lambda_{\tilde{\pi}_1}(\Omega_{11})\frac{L(1,\tilde{\pi}_1\times\pi_2)}{\zeta(n)}.
\end{equation*}
In particular, ${I}^-(0,0,f^0)\asymp X^n$.
\end{lem}

\begin{proof}
The proof is analogous to the proof of Lemma \ref{archimedean-I+type-2-f0}.
We write $I^-(0,0,f^0)$ as
\begin{equation*}
 X^n\Omega_{12}(0)\int_{N_n(F)\bs G_n)F)}\left(\int_{\G_n(F)}\tilde{W}_{1}(yx^\top)
\Omega_{11}\d x \right)\overline{\tilde{W}_{2}(y)}\hat{\Omega}_{21}(e_ny)|\det(y)|\d y
\end{equation*}
Noting that that $\Omega_{11}(x)=\Omega_{11}(x^\top)$, changing variable $x\mapsto x^\top$, and using Iwasawa coordinates as in the proof of Lemma \ref{archimedean-I+type-2-f0} we write the above as
\begin{equation*}
    X^n\Omega_{12}(0)\lambda_{\tilde{\pi}_1}(\Omega_{11})\int_{K_n}\int_{F^\times}\hat{\Omega}_{21}(ze_nk)|z|^n\d^\times z\d k
    \int_{A_{n-1}(F)}\tilde{W}_{1}(\diag(a,1))\overline{\tilde{W}_{2}(\diag(a,1))}\frac{\d^\times a}{\delta(a)}.
\end{equation*}
Proceeding as in the proof of Lemma \ref{archimedean-I+type-2-f0}, we evaluate
the $a$-integral as
$$\frac{L(1,\tilde{\pi}_1\otimes\pi_2)}{\zeta(n)}$$
and two outermost integrals (up to a constand depending only onf $F_v$) as
$\Omega_{21}(0)$. We conclude by recalling the assumptions on the
non-vanishing of $\Omega_{ij}$ and the non-vanishing of the $L$-values.
\end{proof}

\begin{lem}\label{error-main-term-I-}
Let $v\mid\infty$ and $f$ be of type II. Then we have
$$|I^-(0,0,f)-I^-(0,0,f^0)|<\epsilon X^n$$
where $\epsilon$ and $\alpha$ are as in Lemma \ref{lem:alpha-to-delta}. 
\end{lem}

\begin{proof}
The proof is similar to the proof of Lemma \ref{error-main-term-I+}. First of all, following a similar argument we see that both $f\sbr{\mat{x&\\c&1}}$ and $(f-f^0)\sbr{\mat{x&\\c&1}}$ have fixed compact supports in $G_n(F)$ and $F^n$ as functions of $x$ and $c$, respectively.

Changing variables $c\mapsto c/X$ and $y\mapsto yX$ we can write
\begin{multline*}
I^-(0,0,f)-I^-(0,0,f^0)=\int_{\G_n(F)}\int_{N_n(F)\bs G_n(F)}\tilde{W}_{1}(yx^\top)\overline{\tilde{W}_{2}(y)}\\
\left(\int_{F^n}(f-f^0)\mat{x&\\c/X&1}\overline{\psi(e_ny c^\top)}\d c\right)|\det(y)|\d y\d x.
\end{multline*}
We integrate by parts the $c$-integral and apply Lemma \ref{lem:alpha-to-delta} to bound the innermost integral by
$\ll_N \epsilon X^n (1+|e_nx|)^{-N}$ uniformly in $x$ varying over a fixed compact set. We conclude by arguing as in the proof of Lemma \ref{error-main-term-I+}.
\end{proof}

\begin{proof}[Proof of Lemma \ref{archimedean-I-type-2}]
From the proof of Lemma \ref{error-main-term-I-} it follows that $\hat{f}^l(b,x)$ is compactly supported in $x\in G_n(F)$ and Schwartz in $b\in F^n$. Thus the same argument as in the proof of Lemma \ref{archimedean-I-type-1} shows the first assertion.

The second assertion immediately follows from Lemma \ref{error-main-term-I-}, Lemma \ref{archimedean-I-type-2-f0}, and taking $\epsilon$ sufficiently small.
\end{proof}

\section{The residue term}

The goal of this section is to estimate the local components of the residue term
$\R(s_1,s_2)$ as defined in \eqref{residue-term}. As per our custom, we will tacitly
assume that any finite place in unramified and we indicate the necessary changes
for ramified places later on. Moreover, we assume that $f$ is of type I;
see \S\ref{type-1-test-function}.

\subsection{Level places}

Recall $\q$ from Theorem \ref{second-moment-nonarch} and let $v\mid\q$ for the rest of the subsection. Also, recall $H(\Pi)$ from \S\ref{sec:spectral-side}. Our main result of this subsection is the following. The result shares some similarities with those of \cite[\S 9]{Tsuzuki2021Hecke}, from which our proof is heavily inspired.

\begin{prop}\label{local-comp-residue}
We have
\begin{equation*}
H(\mathcal{I}(\tilde{\pi}_1,1/2-s_1))=(-1)^n\frac{\zeta(1)}{\zeta(n+1)}\frac{L(\frac{n+3}{2}-(n+1)s_1,\tilde{\pi}_1)}{L(\frac{n+1}{2}-ns_1+s_2,\tilde{\pi}_2)}
\end{equation*}
and
\begin{equation*}
H(\mathcal{I}(\tilde{\pi}_2,s_2-1/2))=(-1)^n\frac{\zeta(1)}{\zeta(n+1)}\frac{L(\frac{n+3}{2}-(n+1)s_2,\pi_2)}{L(\frac{n+1}{2}-ns_2+s_1,\pi_1)}
\end{equation*}
for all $s_1,s_2\in\C$.
\end{prop}

We prove Proposition \ref{local-comp-residue} in two steps. Recall that the
underlying test function here is $f=\vol\left(\overline{K_0}(\p^e)\right)^{-1}
\1_{\overline{K_0}(\p^e)}$. In \S\ref{local-residue-e-equal-1}, we prove Proposition
\ref{local-comp-residue} when $e=1$. In \S\ref{stability-e}, we show that the
value of $H(\cdot)$ does not change as $e$ increases, thus concluding the proof.

Even though a direct proof for general $e$ is possible by enhancing the methods
in \S\ref{local-residue-e-equal-1}, this approach seemed (to us) not very illuminating.

Finally, in the proofs below we will implicitly assume the Satake parameters of $\pi_i$ are regular and $s_i$ are in generic positions. The general result will follow by
the principle of uniqueness of meromorphic continuation.

\subsubsection{Proof of Proposition \ref{local-comp-residue} when $e=1$}
\label{local-residue-e-equal-1}

First, from \cite{reeder1991old} we record that a (not-necessarily orthonormal)
basis of $\Pi^{\overline{K_0}(\p)}$ is given by $\{W^{(j)}\}_{0\leq j\leq n}$, where
\begin{equation*}
  W^{(j)}(g):=\int_{G_n(F)}W_\Pi\left[g\begin{pmatrix} h^{-1} & \\&1\end{pmatrix}\right]\phi_j(h)|\det(h)|^{1/2}\d h,
\end{equation*}
and
\begin{equation*}
  \phi_j:=N(\p)^{-j(n-j)/2}\mathbf{1}_{X_j},\quad X_j:=K_n\diag(\p1_j,1_{n-j})K_n.
\end{equation*}
At first we assume that $\Pi=\mathcal{I}(\tilde{\pi}_1,s)$ is unitary, that is, $\Re(s)=0$. Applying the Gram--Schmidt process to the above basis, we find an orthonormal basis
$\{V^{(j)}\}_{0\leq j\leq n}$ along with a matrix $A=(a_{ij})_{0\leq i,j\leq n}$
such that $V^{(j)}=\sum_{i=0}^n a_{ij}W^{(i)}$. In particular, $A$ satisfies
$$A^\top wGw\overline{A}=\I_{n+1}\Leftrightarrow G^{-1}=w\overline{A}A^\top w,$$
where $w$ is the long Weyl element in $G_{n+1}$ and
\begin{equation*}
	G=(\IP{W^{n+1-i},W^{n+1-j}})_{1\leq i,j\leq n+1}.
\end{equation*}
We thus have
\begin{align}\label{H-equals-trace-FG}
H(\Pi)=\sum_{j}&\frac{\Psi(1/2+s_1,V^{(j)},W_1)\Psi(1/2+s_2,\overline{V^{(j)}},\overline{W_2})}
{L(1/2+s_1,\Pi\otimes\pi_1)L(1/2+s_2,\widetilde{\Pi}\otimes\widetilde{\pi}_2)}  =
\sum_{j}\sum_{\alpha}\sum_{\beta}a_{\alpha j}\overline{a_{\beta j}}F_{\alpha\beta}\notag\\
&=\operatorname{tr}(A^\top wFw\overline{A})=\operatorname{tr}(w\overline{A}A^\top  wF)=\operatorname{tr}(G^{-1}F) =\operatorname{tr}(F^\top G^{-\top}).
\end{align}
where
\begin{equation*}
    F_{ij}=\frac{\Psi(1/2+s_1,W^{(n+1-i)},W_1) \Psi(1/2+s_2,\overline{W^{(n+1-j)}},\overline{W_2})}{L(1/2+s_1,\Pi\otimes\pi_1)L(1/2+s_2,\tilde{\Pi}\otimes\tilde{\pi}_2)}.
\end{equation*}
Expanding the definition of $\Psi(1/2+s_1,W^{(i)},W_1)$ we compute it as (\emph{cf.} \cite[Proposition 9.3]{Tsuzuki2021Hecke})
\begin{align*}
  &\int_{N_n(F)\bs G_n(F)}W_\Pi\left[\begin{pmatrix} g & \\&1\end{pmatrix}\right]\int_{G_n(F)}\phi_i(h)|\det(h)|^{1/2+s_1}W_1(gh)|\det(g)|^{s_1}\d h\d g\\
  &=N(\p)^{-i(1/2+s_1)}\lambda_{\pi_1}(\phi_i)\int_{N_n(F)\bs G_n(F)}W_\Pi\left[\begin{pmatrix} g & \\&1\end{pmatrix}\right]W_1(g)|\det(g)|^{s_1}\d g\\
  &=N(\p)^{-i(1/2+s_1)}\lambda_{\pi_1}(\phi_i)L(1/2+s_1,\Pi\otimes \pi_1).
\end{align*}
Here we have used that $\phi_i\in C_c^\infty(K_n\bs G_n(F)/K_n)$ and \eqref{unramified-zeta-integral-m-1}.
Similarly, we have
$$ \Psi(1/2+s_2,\overline{W^{(j)}},\overline{W_2})=N(\p)^{-j(1/2+s_2)}\lambda_{\tilde{\pi}_2}(\phi_j)L(1/2+s_2,\tilde{\Pi}\otimes \tilde{\pi}_2).$$
Let $\mu_{\pi}$ denote the Satake parameters of $\pi$. It is known that $\lambda_\pi(\phi_i)$ equals the $i$-th symmetric polynomial evaluated at $\mu_{\pi}$; see \cite[eq.(8.14)]{satake1963spherical}.
Thus, we have that
\begin{equation}\label{formula-for-F-eq}
    F_{ij}=N(\p)^{-((n+1-i)(1/2+s_1)+(n+1-j)(1/2+s_2))}e_{n+1-i}(\mu_{\pi_1})e_{n+1-j}(\mu_{\tilde{\pi}_2}).
\end{equation}
Now we denote
$$\mu_\Pi:=(x_1,\ldots,x_{n+1}),\quad \mu_{\pi_1}:=(a_1,\ldots,a_n),\quad \mu_{\tilde{\pi}_2} := (b_1,\ldots,b_n).$$
We now apply \cite[Prop.9.6]{Tsuzuki2021Hecke} to evaluate $G^{-\top}$. We first point out that our $n$ is $n-1$ in \cite{Tsuzuki2021Hecke}. Also, the inner product in \eqref{inner-product-normalization}, which we use here, differs from the inner product in \cite[eq.(4.7)]{Tsuzuki2021Hecke} by a factor of $L(1,\Pi\otimes\tilde{\Pi})$. Accounting for these different normalizations we obtain
\begin{align*}
G^{-\top}&=(-1)^n\zeta(n+1)^{-1}\mathbb{H}(\mu_\Pi)\mathbb{D}(N(\p)^{-1},\mu_\Pi)R(-N(\p))\\
&=(-1)^n\zeta(n+1)^{-1}\mathbb{E}(\mu_\Pi)^{-1}\diag((Q(N(\p)^{-1}x_k;\mu_\Pi)^{-1})_k)\mathbb{V}(\mu_\Pi)R(-N(\p))
\end{align*}
following \cite[eq.(9.13), eq.(9.12), eq.(9.4)]{Tsuzuki2021Hecke}. Recalling the definitions from \cite[\S9]{Tsuzuki2021Hecke} we write
\begin{multline*}
  G^{-\top}=(-1)^{n}\zeta(n+1)^{-1}\left(\frac{(-1)^n(-x_j)^{n+1-k}}{\prod_{\alpha\neq j}(x_j-x_\alpha)}\right)_{ij}\diag\left(\left(\prod_{\beta=1}^{n+1}(1-N(\p)^{-1}x_jx_\beta^{-1})^{-1}\right)_j\right)\\
  \times\left(x_i^{j-1}\right)_{ij}\diag\left(\left((-N(\p))^{n+1-j}\right)_j\right).
\end{multline*}
Thus we have
\begin{equation}\label{formula-for-G-eq}
	\left(G^{-\top}\right)_{ij} =\zeta(n+1)^{-1}N(\p)^{n+1-j} \sum_{k=1}^{n+1} \frac{(-x_k)^{n+j-i}}
	{\prod_{\alpha\neq k}(x_k-x_\alpha)} \prod_{\beta=1}^{n+1}(1-N(\p)^{-1}x_kx_\beta^{-1})^{-1},
\end{equation}
We also abbreviate $T_i:=N(\p)^{-1/2-s_i}$. Hence, combining \eqref{H-equals-trace-FG}, \eqref{formula-for-F-eq}, and \eqref{formula-for-G-eq} we obtain
\begin{multline*}
	H(\Pi)=\zeta(n+1)^{-1}\sum_{i,j,k}e_{n+1-i}(\mu_{\pi_1})e_{n+1-j}(\mu_{\tilde{\pi}_2})
	T_1^{n+1-i}T_2^{n+1-j}\\
	\times\frac{(-x_k)^{n+1-i}}{\prod_{\alpha\neq k}(x_k-x_\alpha)}
	\prod_{\beta=1}^{n+1}(1-N(\p)^{-1}x_kx_\beta^{-1})^{-1}(-x_k)^{j-1}N(\p)^{n+1-j}.
\end{multline*}
Applying Vieta's formula to the sums over $i$ and $j$, we see that the above can be written as
\begin{equation}\label{general-H-Pi}
    H(\Pi)=(-1)^n\zeta(n+1)^{-1}\sum_{k=1}^{n+1}\frac{\prod_{\gamma=1}^n(1-x_kT_1a_\gamma)\prod_{\delta=1}^n(x_k-T_2N(\p)b_\delta)}{\prod_{\alpha\neq k}(x_k-x_\alpha)\prod_{\beta=1}^{n+1}(1-N(\p)^{-1}x_kx_\beta^{-1})}.
\end{equation}
Now we prove the first equality in Proposition \ref{local-comp-residue}.

Note that both $H(\Pi)=H(\mathcal{I}(\tilde{\pi}_1,s))$ and the above expression are meromorphic functions in $\mu_\Pi=(x_i)_i$ \emph{i.e.}, in $s\in\C$. Thus we may equate $H(\mathcal{I}(\tilde{\pi}_1,1/2-s_1))$ and the above expression for $(x_i)_i=\mu_{\mathcal{I}(\tilde{\pi}_1,1/2-s_1)}$. Notice that the condition $\Pi=\mathcal{I}(\tilde{\pi}_1,1/2-s_1)$, translates to
$$T_1a_\beta=N(\p)^{-1}x_\beta^{-1},\quad 1\leq \beta\leq n;\qquad\qquad x_{n+1}=N(\p)^{n(1/2-s_1)}=T_1^nN(\p)^{n}.$$
Hence, we can rewrite \eqref{general-H-Pi} as
\begin{align*}
H(\Pi)&={(-1)^n}\zeta(n+1)^{-1}\sum_{k=1}^{n+1}\frac{\prod_{\delta=1}^n(x_k-N(\p)T_2b_\delta)}{(1-N(\p)^{-1}x_kx_{n+1}^{-1})\prod_{\alpha\neq k}(x_k-x_\alpha)}\\
&=\frac{(-1)^n\zeta(n+1)^{-1}(N(\p)x_{n+1})^{-n}}{\prod_{\beta=1}^{n+1}(1-N(\p)^{-1}x_\beta x_{n+1}^{-1})}\sum_{k=1}^{n+1}\prod_{\delta=1}^n(x_k-N(\p)T_2b_\delta)\prod_{\gamma\neq k}\frac{N(\p)x_{n+1}-x_\alpha}{x_k-x_\alpha}.
\end{align*}
Now consider the polynomial
$$f(X):=\prod_{\delta=1}^n(X-N(\p)T_2b_\delta).$$
As the degree of $f$ is $n$ and $x_i$ are distinct, by the Lagrange interpolating polynomial we have
$$f(X)=\sum_{k=1}^{n+1}f(x_k)\prod_{\alpha\neq k}\frac{X-x_\alpha}{x_k-x_\alpha}.$$
Thus we have
\begin{equation*}
    \prod_{\delta=1}^n(N(\p)x_{n+1}-N(\p)T_2b_\delta)=f(N(\p)x_{n+1})=\sum_{k=1}^{n+1}\prod_{\delta=1}^n(x_k-N(\p)T_2b_\delta)\prod_{\alpha\neq k}\frac{N(\p)x_{n+1}-x_\alpha}{x_k-x_\alpha}.
\end{equation*}
Consequently, $H(\Pi)$ equals
$$\frac{(-1)^n\zeta(n+1)^{-1}\prod_{\delta=1}^n(1-x_{n+1}^{-1}T_2b_\delta)}{\prod_{\beta=1}^{n+1}(1-N(\p)^{-1}x_\beta x_{n+1}^{-1})}=(-1)^n\frac{\zeta(1)}{\zeta(n+1)}\frac{\prod_{\delta=1}^n(1-x_{n+1}^{-1}T_2b_\delta)}{\prod_{\beta=1}^{n}(1-N(\p)^{-1}x_\beta x_{n+1}^{-1})}.$$
Substituting for $x_k$ their values in terms of $a_k$ implies that $H(\Pi)$ equals
\begin{equation*}
    (-1)^n\frac{\zeta(1)}{\zeta(n+1)}\frac{\prod_{\delta=1}^n(1-T_2T_1^{-n}N(\p)^{-n}b_\delta)}{\prod_{\beta=1}^{n}(1-T_1^{-n-1}N(\p)^{-n-2}a_\beta^{-1})}.
\end{equation*}
Finally, recalling the definitions of $a_\beta,\,b_\delta,\,T_1,\,T_2$ we see the above equals
\begin{equation*}
	(-1)^n\frac{\zeta(1)}{\zeta(n+1)}\frac{L(\frac{n+3}{2}-(n+1)s_1,\tilde{\pi}_1)}
	{L(\frac{n+1}{2}-ns_1+s_2,\tilde{\pi}_2)}.
\end{equation*}
This concludes the proof of the first equality in Proposition \ref{local-comp-residue}
when $e=1$.

Proof of the second equality in Proposition \ref{local-comp-residue} is similar. Indeed, For $\Pi=\mathcal{I}(\tilde{\pi}_2,s_2-1/2)$ we have
$$x_\delta=b_\delta T_2 N(\p),\quad 1\leq \delta\leq n;\qquad\qquad x_{n+1}=N(\p)^{-n}T_2^{-n}.$$
Thus we can rewrite \eqref{general-H-Pi} as
\begin{align*}
H(\Pi)&=(-1)^n\zeta(n+1)^{-1}\frac{\prod_{\gamma=1}^n(1-x_{n+1}T_1a_\gamma)\prod_{\delta=1}^n(x_{n+1}-x_\delta)}{\prod_{\alpha\neq n+1}(x_k-x_\alpha)\prod_{\beta=1}^{n+1}(1-N(\p)^{-1}x_{n+1}x_\beta^{-1})}\\
&=(-1)^n\frac{\zeta(1)}{\zeta(n+1)}\frac{\prod_{\gamma=1}^n(1-N(\p)^{-n}T_2^{-n}T_1a_\gamma)}{\prod_{\beta=1}^{n}(1-N(\p)^{-1}N(\p)^{-n}T_2^{-n}x_\beta^{-1})}.
\end{align*}
Once again, recalling $a_\beta,\,b_\delta,\,T_1,\,T_2$, we verify the second equality
in Proposition \ref{local-comp-residue} when $e=1$.

\subsubsection{Stability with respect to the depth}\label{stability-e}

Here we show that for $e\ge 1$
$$H\left(\mathcal{I}(\tilde{\pi}_1,1/2-s_1),\vol\left(\overline{K_0}(\p^e)\right)^{-1}\1_{\overline{K_0}(\p^e)}\right)=H\left(\mathcal{I}(\tilde{\pi}_1,1/2-s_1),\vol\left(\overline{K_0}(\p)\right)^{-1}\1_{\overline{K_0}(\p)}\right)$$
and the same for $H(\mathcal{I}(\tilde{\pi}_2,s_2-1/2))$. This together with the result in \S\ref{local-residue-e-equal-1} proves Proposition \ref{local-comp-residue}.
We only show the first case, the second one is similar. 

We abbreviate $1/2-s_1$ as $s$. We assume $s\in i\mathbb{R}$ for now, consequently, $\mathcal{I}(\tilde{\pi}_1,s)$ is unitary. First note that, the difference between the left-hand side and the right-hand side above is given by
$$\sum_{W\in\B\left(\mathcal{I}(\tilde{\pi}_1,s)^{\overline{K_0}(\p^e)}\cap\left(\mathcal{I}(\tilde{\pi}_1,s)^{\overline{K_0}(\p)}\right)^\perp\right)}\frac{\Psi(1-s,W,W_{{1}}){\Psi(1/2+s_2,\overline{W},\overline{W_{{2}}})}}{L(1-s,\mathcal{I}(\tilde{\pi}_1,s)\otimes{\pi}_{1})L(1/2+s_2,\mathcal{I}(\pi_1,-s)\otimes\tilde{\pi}_{2})}.$$
Following the proof of \cite[Lemma 9.1]{JaNu2021reciprocity}, we construct a basis $\B$ in the above sum consisting of elements of the form $\frac{W}{\|W\|}$ such that 
\begin{itemize}
    \item $W$ is entire in $s$ and
    \item $\|W\|^2$ is an $s$-independent non-zero constant multiple of $L(1,\mathcal{I}(\tilde{\pi}_1,s)\otimes I({\pi}_1,-s))^{-1}$.
\end{itemize}
We infer the meromorphic property of the above sum in the $s$ variable from Lemma \ref{analytic-continuation-local-factor} and note that $L(1,\mathcal{I}(\tilde{\pi}_1,s)\otimes I({\pi}_1,-s))$ is holomorphic in $s\in i\Rr$.
Thus, it suffices to show that for every $s\in i\Rr$ and every such $W$ the expression
$$\frac{\Psi(1-s-z,W,W_1)}{L(1-s-z,\mathcal{I}(\tilde{\pi}_1,s)\otimes \pi_1)}$$
vanishes at $z=0$.

By the $\GL(n+1)\times\GL(n)$ local functional equation as in \eqref{LFE-m-m-1} and recalling that $\pi_1$ is unramified, the above equals
\begin{equation*}
\frac{\Psi(s+z,\tilde{W},\tilde{W}_1)}{L(s+z,\mathcal{I}({\pi}_1,-s)\otimes\tilde{\pi}_1)}=\frac{\Psi(s+z,\tilde{W},\tilde{W}_1)}{L(z,\pi_1\otimes\tilde{\pi}_1)L(z+(n+1)s,\tilde{\pi}_1)}.
\end{equation*}
Note that $L(z,\pi_1\otimes\tilde{\pi}_1)$ has a pole at $z=0$ of order at least $n$. Thus to prove the claim it is enough to show that the zeta integral in the numerator can have a pole at $z=0$ of order at most $n-1$.

Now as $W$ is $\overline{K_0}(\p^e)$-invariant we have that for $a\in A_n(F)$
\begin{equation}\label{support-1}
\text{if }\tilde{W}\left[\begin{pmatrix}a&\\&1\end{pmatrix}\right]\neq 0 \text{ then } |a_n|\le N(\p)^e.
\end{equation}
The above follows as $\tilde{W}$ is $\overline{K_0}(\p^e)^\top$-invariant, hence, for $v\in (\o^\times)^n$
$$\tilde{W}\sbr{\mat{a&\\&1}\mat{\I_n&\p^ev\\&1}}-\tilde{W}\sbr{\mat{a&\\&1}}=\left(\psi(a_n\p^e e_nv)-1\right)\tilde{W}\sbr{\mat{a&\\&1}}.$$
On the other hand, as $W$ is in the orthogonal complement of the $\overline{K_0}(\p)$-invariant vectors, taking contragredient we have
$$0=\int_{\overline{K_0}(\p)}{W}(\cdot k)\d k =\int_{\overline{K_0}(\p)^\top}\tilde{W}(\cdot k)\d k=\int_{\overline{G_{n+1}}(F)}\tilde{W}(\cdot g)\mathbf{1}_{\overline{K_0}(\p)^\top}(g)\d g.$$
We write the $g$-integral in the Bruhat coordinate as
$$g=\begin{pmatrix}\I_n&b\\&1\end{pmatrix}\begin{pmatrix}a&\\c&1\end{pmatrix},\quad a\in G_n(F),\quad b,c^\top\in F^n;\quad\quad \d g=\frac{\d a}{|\det a|}\d b\d c.$$
Note that $g\in \overline{K_0}(\p)^\top$ implies that 
$$a\in\mathrm{Mat}_n(\o),\det(a)\in\o^\times\implies a\in K_n;\quad b\in \p\o^n;\quad c\in\o^n.$$
Thus applying $\left(\overline{K_0}(\p^e)\cap P_{n+1}\right)^\top$-invariance of $\tilde{W}$ we get that for $a\in A_n(F)$
$$0=\int_{\o^n}\tilde{W}\left[\mat{a&\\&1}\begin{pmatrix}\I_n&\p b\\&1\end{pmatrix}\right]\d b=\tilde{W}\left[\mat{a&\\&1}\right]\int_{\o^n}\psi(a_n\p e_nb)\d b.$$
This implies that
\begin{equation}\label{support-2}
    \text{if }\tilde{W}\left[\begin{pmatrix}a&\\&1\end{pmatrix}\right]\neq 0\text{ then }|a_n|\ge p.
\end{equation}
Now using $K_n$ properties of $\tilde{W}$ and $\tilde{W}_1$ and noting that $\mathcal{I}(\pi_1,-s)$ is tempered, for $\Re(z)>0$
we can write $\Psi(s+z,\tilde{W},\tilde{W}_1)$ as an absolutely convergent integral
\begin{align}\label{further-integral}
&\int_{A_n(F)}\tilde{W}\sbr{\mat{a&\\&1}}\tilde{W}_1(a)|\det(a)|^{s+z-1/2}\frac{\d a}{\delta(a)}\notag\\
&=\int_{F^\times}\int_{A_{n-1}(F)}\tilde{W}\left[\begin{pmatrix}a'a_n &  & \\&a_n&\\&&1\end{pmatrix}\right]\tilde{W}_1\left[\begin{pmatrix}a' & \\&1\end{pmatrix}\right]|a_n^n\det(a')|^{s+z-1/2}\frac{\d a'}{|\det(a')|\delta(a')}\d^\times a_n\notag\\
&=\int_{F^\times}\int_{N_{n-1}(F)\backslash G_{n-1}(F)}\tilde{W}\left[\begin{pmatrix}ha_n &  & \\&a_n&\\&&1\end{pmatrix}\right]\tilde{W}_1\left[\begin{pmatrix}h & \\&1\end{pmatrix}\right]|a_n^n\det h|^{s+z-1/2}\frac{\d h}{|\det h|}\d^\times a_n.
\end{align}

Now in the next lemma we evaluate $\tilde{W}_1(\diag(h,1))=\W_{\tilde{\pi}_1}(\diag(h,1))$ in terms of spherical Whittaker function on $G_{n-2}(F)$. This is a recursion formula that might be of independent interest, so we write it in general notations\footnote{The archimedean analog of this formula, in a more convoluted form, is an important part of the computations in \cite{JaNe2019anv}.}.

\begin{lem}\label{decomposition-lowerdegree}
Let $\pi$ be the unramified representation of $\overline{G_m}(F_v)$ with Langlands parameters $\{\mu_i\}_{i=1}^m$ and 
let $\pi_j$ be the unramified representation of $G_{m-1}(F)$ with Langlands parameters given by $\{\mu_1,\dots,\mu_{j-1},\mu_{j+1},\dots,\mu_m\}$. Then we have
$$W_\pi\sbr{\mat{h&\\&1}}= |\det(h)|^{1/2}\sum_{j=1}^m\frac{\mu_j^{-1}W_{\pi_j}(h)}{\prod_{i\neq j}(\mu_i-\mu_j)},$$
for any $h\in N_{m-1}(F)\backslash G_{m-1}(F)$ with $e_{m-1}h\in \o^{m-1}$.
\end{lem}

\begin{proof}
Applying sphericality of $W_\pi$ and $W_{\pi_j}$, and \eqref{shintani} it suffices to show that
$$W_\pi(\diag(\p^\nu,1))= |\det(\diag(\p^\nu))|^{1/2}\sum_{j=1}^m\frac{\mu_j^{-1}W_{\pi_j}(\diag(\p^\nu))}{\prod_{i\neq j}(\mu_i-\mu_j)},$$
for any $\nu\in\Z^{m-1}$ dominant with $\nu_{m-1}\ge 0$.
    We again apply \eqref{shintani} to explicate $W_\pi(\diag(\p^{\tilde{\nu}}))$ for $\tilde{\nu}:=(\nu,0)$. Writing out $$\chi_{\tilde{\nu}}(\mu)=\frac{\det((\mu_j^{m-i+\tilde{\nu}_i})_{i,j})}{\prod_{i<j}(\mu_i-\mu_j)},$$ expanding the determinant in the numerator with respect to the last row, and 
    using $\prod_j\mu_j =1$ we see
    \begin{align*}
    \chi_{\tilde{\nu}}(\mu)
    &=\sum_{k=1}^m\frac{\prod_{j\neq k}\mu_j(-1)^{m+k}\det((\mu_j^{m-1+i+\nu_i})_{i,j\neq k})}{\prod_{i<j}(\mu_i-\mu_j)}\\
    &=\sum_{k=1}^m\frac{\mu_k^{-1}(-1)^{m+k}\chi_{\nu}(\mu_1,\dots,\mu_{k-1},\mu_{k+1},\dots,\mu_m)}{(-1)^{m-k}\prod_{i\neq k}(\mu_i-\mu_k)}.
    \end{align*}
    Applying \eqref{shintani} once again and noting that $\delta(\diag(\p^{\tilde{\nu}}))=|\det(\diag(\p^\nu))|\delta(\diag(\p^{{\nu}}))$
    we conclude.
\end{proof}

Applying \eqref{further-integral}, Lemma \ref{decomposition-lowerdegree}, \eqref{support-1}, and \eqref{support-2} we see that $\psi(s+z,\tilde{W},\tilde{W}_1)$ can be written as a \emph{finite} linear combination of the integrals of the form
$$\int_{N_{n-1}\backslash G_{n-1}}\tilde{W}_{a_n}\left[\begin{pmatrix}h&&\\&1&\\&&1\end{pmatrix}\right]\tilde{W}_{\tilde{\pi}_{1,j}}(h)|\det(h)|^{s+z-1}\d h$$
where $\tilde{W}_{a_n}(\cdot):=\tilde{W}(\cdot\diag(a_n\I_n,1))$. The coefficients of the linear combination depend only on $a_n$, $z$, and $\mu_{\pi_1}$. Moreover, as a function of $z$ the coefficients are entire.

We notice that the integral in the last display is a
$\GL(n+1)\times\GL(n-1)$ zeta integral. Thus, as a function of $z$,
can only have poles at the poles of $L(s+z,\mathcal{I}(\pi_1,-s)\otimes
\tilde{\pi}_{1,j})$; see \cite[Corollary 6.3.1]{cogdell2004lectures}. In particular, the order of the pole at $z=0$ is at most $n-1$. Hence the same is true for the finite linear combination, as required.

\begin{rmk}
  When $v$ is ramified, by the same changes of variables as in Remark
  \ref{rmk-ramified-zeta-function}, we reduce to the unramified
  computation at the cost of the extra factor $\Delta^{\mu^R(s_1,s_2)}$ for some
  affine function $\mu^R$ whose coefficients only depend on $n$.
\end{rmk}

\subsection{Archimedean place}

Recall $\mathfrak{X}$ from Theorem \ref{second-moment-arch} and let $v\mid\infty$ for the rest of the subsection. Also, recall $h(\Pi)$ from \S\ref{sec:spectral-side}. Our main result of this subsection is as follows.

\begin{prop}\label{main-prop-residue-arch}
We have
$$h(\mathcal{I}(\tilde{\pi}_1,1/2-s_1))\ll 1$$
and
$$h(\mathcal{I}(\tilde{\pi}_2,s_2-1/2))\ll 1,$$
for all $s_1,s_2\in\C$ a fixed distance away from the poles\footnote{Here and elsewhere in this subsection, by ``poles of $h$'' we mean ``poles of $s\mapsto h(\mathcal{I}(\tilde{\pi}_1, 1/2-s))$ or $s\mapsto h(\mathcal{I}(\tilde{\pi}_2,s-1/2))$''.} of $h$.
\end{prop}

We only prove the first estimate above; the second one follows similarly.

The proof of this proposition has some resemblance with its non-archimedean counterpart in the previous section. We first define a smooth compactly supported function on $\overline{G_{n+1}}(F)$ by
\begin{equation}\label{def-f-1}
  f_1\sbr{\mat{\I_n&\\c&1}p}:=X^{-n}f\sbr{\mat{\I_n&\\c/X&1}p},\quad c\in F^n, p\in Q_{n+1}.
\end{equation}
$f_1$ can be thought as a \emph{essentially fixed} smooth characteristic function of the $\tau$-radius ball around the identity of $\overline{G_{n+1}}(F)$. We write
$$h(\mathcal{I}(\tilde{\pi}_1,1/2-s_1),f)=h(\mathcal{I}(\tilde{\pi}_1,1/2-s_1),f-f_1)+h(\mathcal{I}(\tilde{\pi}_1,1/2-s_1),f_1).$$
In Lemma \ref{lem:residue-main-term-bound} we show that
$$h(\mathcal{I}(\tilde{\pi}_1,1/2-s_1),f_1)\ll 1$$
which can be understood as the analogue of the result in \S\ref{local-residue-e-equal-1}.
In Lemma \ref{lem:reduce-to-bound-main-term} we show that
$$h(\mathcal{I}(\tilde{\pi}_1,1/2-s_1),f-f_1)=0.$$
Heuristically, $\mathcal{I}(\tilde{\pi}_1,1/2-s_1)(f-f_1)$ can be thought of as an archimedean analogue of the orthogonal projection onto $\overline{K_0}(\p^e)$-invariant vectors which lie in the orthogonal complement of $\overline{K_0}(\p)$-invariant vectors in the non-archimedean case. Thus the above display can be understood as the analogue of the results in \S\ref{stability-e}. Clearly, Lemma \ref{lem:residue-main-term-bound} and Lemma \ref{lem:reduce-to-bound-main-term} prove the first estimate in Proposition \ref{main-prop-residue-arch}.

\begin{lem}\label{lem:residue-main-term-bound}
We have
$$h(\mathcal{I}(\tilde{\pi}_1,1/2-s_1), f_1) = O(1)$$
for every $s_1,s_2$ a fixed distance away from the poles of $h$.
\end{lem}

\begin{proof}
Recall the derivative bounds of $f$ from \S\ref{type-1-test-function}, which follows from the majorant property of $f$. It thus follows from \eqref{def-f-1} that
$$\partial_a^\alpha\partial_b^\beta\partial_c^\gamma\partial_d^\delta \,f_1\sbr{\mat{a&b\\c&d}} \ll 1,$$
for any fixed multi-indices $\alpha,\beta,\gamma,\delta$. Consequently, we have for any $j\ge 0$
$$\mathfrak{D}^jf_1\ll 1.$$
Here $\mathfrak{D}:=\mathfrak{D}_0+R$ where $\mathfrak{D}_0$ is the elliptic differential operator as defined in \cite[\S 3.4]{JaNu2021reciprocity} and $R$ is a sufficiently large constant.

We recall the meromorphic continuation of $h(\mathcal{I}(\tilde{\pi}_1,s),f_1,s_1,s_2)$ in $s\in\C$ from the proof of Lemma \ref{analytic-continuation-local-factor} (also see \cite[Lemma 9.1]{JaNu2021reciprocity}). For $s\in i\Rr$ we construct an orthonormal basis of $\B(\mathcal{I}(\tilde{\pi}_1,s))\ni W$ such that 
\begin{itemize}
    \item $W$ is entire in $s$;
    \item $\langle W,\tilde{W}\rangle$ is a non-zero constant, independent of $s$ (here $\langle,\rangle$ denotes a non-zero pairing between the dual representations);
    \item moreover, we may choose $W$ to be eigenvector of $\mathfrak{D}$.
\end{itemize}
The last requirement above can be satisfied by choosing $K$-type induced vectors underlying $W$ in proof of \cite[Lemma 9.1]{JaNu2021reciprocity}. We thus understand $h(\mathcal{I}(\tilde{\pi}_1,s),f_1,s_1,s_2)$ for $|\Re(s)|\le 2$ as the meromorphic continuation of
$$\sum_W\frac{\Psi(1/2+s_1,\mathcal{I}(\tilde{\pi}_1,s)(f_1)W,W_1)\Psi(1/2+s_2,\tilde{W},\overline{W_2})}{\langle W,\tilde{W}\rangle}$$
where $W$ runs over vectors with the above properties.

Let $\mathfrak{D}W=\lambda_W W$. Then $\lambda_W$ is an entire function in $s$. Moreover, making $R$ sufficiently large we make sure $\lambda_W\neq 0$ for $|\Re(s)|\le 2$. Finally, the trace class property of $\mathfrak{D}^{-L}$ ensure that 
$$\sum_W\frac{|\lambda_W|^{-L}}{\langle W, \tilde{W}\rangle}<\infty$$
for some large $L>0$.

First, using \cite[Lemma 3.3]{JaNu2021reciprocity} (\emph{loc.\ cit.}, recall the definition of the Sobolev norm \cite[\S3.4]{JaNu2021reciprocity}) we record that
there is a $d\ge 0$ such that
$$\Psi(1/2+s_2,\tilde{W},\overline{W_2})\ll_{s_2,W_2} |\lambda_W|^{d}$$
for all $W$ and for all fixed $s_2\in\C$ that is regular for the above zeta integral.
Now let $s_1'\in\C$ such that $\Re(s_1')$ is sufficiently positive. Now writing the $\Psi$ in the integral representation and integrating by parts many times we obtain
\begin{equation*}
  \Psi(1/2+s_1',\mathcal{I}(\tilde{\pi}_1,s)(f_1)W,W_1)=|\lambda_W|^{-N}\Psi(1/2+s_1',\mathcal{I}(\tilde{\pi}_1,s)(\D^Nf_1)W, W_1).
\end{equation*}
Working as above and applying the derivative bound of $f_1$ as above we obtain for any large $N$
$$\Psi(1/2+s_1',\mathcal{I}(\tilde{\pi}_1,s)(f_1)W,W_1)\ll_N\lambda_W^{-N},$$
uniformly in $X$.
Similarly, applying the functional equation \eqref{LFE-m-m-1}
we see that
$$\Psi(1/2-s_1',\mathcal{I}(\tilde{\pi}_1,s)(f_1)W,W_1)=\gamma(1/2+s_1',\mathcal{I}(\pi_1,-s)\otimes\tilde{\pi}_1)\Psi(1/2+s_1',\mathcal{I}(\pi_1,-s)(f_1^\iota)\tilde{W},\tilde{W}_1)$$
where $f_1^\iota(g):=f_1(g^{-\top})$. Once again applying derivative bounds of $f_1$ and working as above we obtain for all large $N$
$$\Psi(1/2+s_1,\mathcal{I}(\tilde{\pi}_1,s)(f_1)W,W_1)\ll_N\lambda_W^{-N},$$
uniformly in $X$.
Thus making $N$ sufficiently large we see that
$$h(\mathcal{I}(\tilde{\pi}_1,s), f_1)\ll\sum_W\frac{|\lambda_W|^{-N}}{\langle W,\tilde{W}\rangle}\ll 1,$$
which follows from the trace class property. Choosing $s=1/2-s_1$, avoiding a pole of $h$, we conclude. 
\end{proof}

\begin{lem}\label{lem:reduce-to-bound-main-term}
Let $f_2:=f-f_1$. We have
$$h(\mathcal{I}(\tilde{\pi}_1,1/2-s_1),f_2)=0$$
for every $s_1,s_2$.
\end{lem}

We first need some preparatory results to prove this lemma.

\begin{lem}\label{lem:homolomorphic-gln-gln}
Let $\pi$ be a tempered representation of $G_n(F)$, and $W\in\pi$ and $W'\in\tilde{\pi}$. Let $\Phi$ be any Schwartz function on $F^n$ such that $\Phi(x)\ll |x|$ as $x\to 0$. Then the quantity
$${\gamma(s,\pi\otimes\tilde{\pi})}\Psi(s,W,W',\Phi)$$ vanishes at $s=0$.
\end{lem}

\begin{proof}
Applying $\GL(n)\times\GL(n)$ local functional equation we write ${\gamma(s,\pi\otimes\tilde{\pi})}\Psi(s,W,W',\Phi)$ for $\Re(s)<1$ as
$$\omega_{\pi}(-1)^{n-1}\int_{N_n(F) \bs G_n(F)}\tilde{W}(g)\tilde{W}'(g)\hat{\Phi}(e_ng)|\det(g)|^{1-s}\d g=:\Psi(s),$$
where $\hat{\Phi}$ is the Fourier transform of $\Phi$ and the above integral is absolutely convergent (follows from \cite[Lemma 3.1]{JaNu2021reciprocity}).

We use Bruhat coordinates
$$g=z\mat{h&\\&1}\mat{\I_{n-1}&\\c&1};\quad z\in F^\times, h\in N_{n-1}(F)\bs G_{n-1}(F), c\in F^{n-1};$$
$$\d g= \d^\times z \frac{\d h}{|\det(h)|}\d c;$$
to write $\Psi(s)\omega_{\pi}(-1)^{n-1}$ as
\begin{multline*}
  \int_{F^{n-1}}\int_{N_{n-1}(F)\bs G_{n-1}(F)}\tilde{W}\sbr{\mat{h&\\&1}\mat{\I_{n-1}&\\c&1}}\tilde{W}'\sbr{\mat{h&\\&1}\mat{\I_{n-1}&\\c&1}}|\det(h)|^{-s}\d h\\
  \int_{F^{\times}}\hat{\Phi}(z(c,1))|z|^{n(1-s)}\d^\times z\d c.
\end{multline*}
We claim that ${\gamma(1-ns){\Psi}(s)}$ is holomorphic at $s=0$. As ${\gamma(1-ns){\Psi}(s)}$ is a meromorphic function of $s$ it is enough to show that $$\lim_{s\to 0}{\gamma(1-ns){\Psi}(s)}$$ exists. Indeed, using the $G_n(F)$-invariance of the standard pairing between $\pi$ and $\tilde{\pi}$ in the Whittaker model we can write this limit as
\begin{multline*}
  \int_{N_{n-1}(F)\bs G_{n-1}(F)}\tilde{W}\sbr{\mat{h&\\&1}}\tilde{W}'\sbr{\mat{h&\\&1}}\d h\\
  \lim_{s\to 0}\gamma(1-ns)\int_{F^{n-1}}\int_{F^{\times}}\hat{\Phi}(z(c,1))|z|^{n(1-s)}\d^\times z\d c.
\end{multline*}
Note that the $h$-integral is absolutely convergent. We change variable $c\to c/z$ in the inner integral above to write the above limit as
$$\lim_{s\to 0}\gamma(1-ns)\int_{F^{n-1}}\int_{F^\times}\hat{\Phi}(c,z)|z|^{1-ns}\d^\times z\d c.$$
Now we apply the local Tate functional equation to the inner integral above to write the limit as
$$\lim_{s\to 0}\int_{F^\times}\Phi(ze_n)|z|^{ns}\d^\times z.$$
The above integral is absolutely convergent at $s=0$ thanks to the decay condition of $\Phi$ near zero. Now noting that
$$\gamma(1-ns)=\frac{\zeta(ns)}{\zeta(1-ns)}$$
has a simple pole at $s=0$ we conclude.
\end{proof}

\begin{lem}\label{lem:holomorphic-gln+1-gln}
Let $\pi$, $W'$, and $\Phi$ be as in Lemma \ref{lem:homolomorphic-gln-gln}. Assume that $\pi$ is a principal series representation. Also, let $\mu\in\C$ with $\Re(\mu)$ sufficiently large and $W\in\pi_\mu:=\pi\boxplus|.|^\mu$. Then
$${\gamma(s,\pi\otimes\tilde{\pi})}\int_{N_n(F) \bs G_n(F)}W\sbr{\mat{g&\\&1}}W'(g)\Phi(e_ng)|\det(g)|^{s-1/2}\d g$$
vanishes at $s=0$.
\end{lem}

\begin{proof}
Let $\Re(s)$ be sufficiently large. Note that $\pi_\mu$ is also a principal series. We work as in \cite[\S7-8]{jacquet2009archimedean} to assume that there is a $W_1\in\pi$ that is entire in its Langlands parameters and Schwartz functions $\Phi_1$ on $\mathrm{Mat}_{n\times n}(F)$ and $\Phi_2$ on $F^n$ so that
$$W\sbr{\mat{g&\\&1}}=|\det(g)|^{\frac{n}{2}+\mu}\int_{G_n(F)}\Phi_1(h^{-1}g)\Phi_2(e_nh)W_1(h)|\det(h)|^{-\frac{n-1}{2}-\mu}\d h.$$
The above integral representation is absolutely convergent.
Thus the integral in the lemma after the change of variables $h\mapsto gh$, $h\mapsto h^{-1}$, and $g\mapsto gh$ can be written as
$$\int_{G_n(F)}\Phi_1(h)|\det(h)|^{\frac{n-1}{2}+\mu+s}\int_{N_n(F)\bs G_n(F)}W_1(g)W'(gh)\Phi(e_ngh)\Phi_2(e_ng)|\det(g)|^{s}\d g\d h.$$
The above integral representation is again absolutely convergent. Justification of interchange of integrals is exactly as in \cite[\S8.3]{jacquet2009archimedean}. Note that for $x\in F^n$ we have $\Phi(xh)\ll |xh| \le |x|\|h\|$, where $\|.\|$ denotes a fixed operator norm on $G_n(F)$. Let
$$\Psi_\pi(s,h):=\gamma(s,\pi\otimes\tilde{\pi})\int_{N_n(F)\bs G_n(F)}W_1(g)W'(gh)\Phi(e_ngh)\Phi_2(e_ng)|\det(g)|^{s}\d g$$
which is originally defined for sufficiently large $\Re(s)$ and can be meromorphically continued for $s\in \C$ via the theory of zeta integral; see \cite[Lecture 2]{cogdell2007functions}.
Applying Lemma \ref{lem:homolomorphic-gln-gln} we see that $\Psi_\pi(s,h)$ vanishes at $s=0$ for each $h\in G_n(F)$. Thus to prove the lemma it suffices to show that
$$\int_{G_n(F)}\Phi_1(h)|\det(h)|^{\frac{n-1}{2}+\mu+s}\Psi_\pi(s,h)\d h$$
is absolutely convergent for $\Re(s)\ge 0$.

We claim that $\Psi_\pi(s,h)$ as a function of $h$ is at most polynomially bounded in terms of $\|h\|$ at a neighbourhood of $s=0$. This claim is sufficient. Indeed,
$$\int_{G_n(F)}\Phi_1(h)|\det(h)|^{\frac{n-1}{2}+\mu+s}\Psi_\pi(s,h)\d h\ll\int_{G_n(F)}|\Phi_1(h)||\det(h)|^{\frac{n-1}{2}+\Re(\mu)+\Re(s)}\|h\|^{O(1)}\d h.$$
The last integral is absolutely convergent at a neighbourhood of $s=0$ for sufficiently large $\Re(\mu)$; see \emph{e.g.}, \cite[Lemma 3.3 (ii)]{jacquet2009archimedean}.

Now we turn to prove the above claim. We apply \cite[Lemma 3.1]{JaNu2021reciprocity} and using properties of Sobolev norm as in \cite[\S2.4.1, S1b.]{MV2010subconvexity} to see that $\Psi_\pi(s,h)\ll\|h\|^{O(1)}$ for large $\Re(s)$, as long as, $s$ is away from a pole of $\gamma(\cdot,\pi\otimes\tilde{\pi})$. 
If $\Re(s)$ is sufficiently negative then using $\GL(n)\times\GL(n)$ local functional equation (see \eqref{LFE-m-m}) we can write $\Psi_\pi(s,h)$ as
$$\omega_\pi(-1)^{n-1}\int_{N_n(F)\bs G_n(F)}\tilde{W}_1(g)\tilde{W}'(gh^{-\top})\Phi_3(e_ng)|\det(g)|^{1-s}\d g,$$
where for $c\in F^n$
$$\Phi_3(c):=\int_{F^n}\Phi(xh)\Phi_2(x)\overline{\psi(cx^\top)}\d x\ll_N\|h\|^{O_N(1)}(1+|c|)^{-N}.$$ 
Working as before we see that $\Psi_\pi(s,h)\ll \|h\|^{O(1)}$ if $\Re(s)$ is sufficiently negative. Using the boundedness $\Psi_\pi(\cdot,h)$ in vertical strips (follows from \emph{e.g.}, \cite[Theorem 2.1 (ii)]{jacquet2009archimedean}) and applying Phragm\'en--Lindel\"of we conclude that $\Psi_\pi(s,h)\ll \|h\|^{O(1)}$ as a function of $h$ in a neighborhood of $s=0$.
\end{proof}

\begin{proof}[Proof of Lemma \ref{lem:reduce-to-bound-main-term}]
We abbreviate $1/2-s_1$ as $s$. Recall that (\emph{e.g.}, from the proof of Lemma \ref{lem:residue-main-term-bound}) that $h(\mathcal{I}(\tilde{\pi}_1,s),f_2)$ is an absolute convergent sum of the terms of the form
$$\langle W,\tilde{W}\rangle^{-1}\Psi(1-s,\mathcal{I}(\tilde{\pi}_1,s)(f_2)W,W_1)\Psi(1/2+s_2,\tilde{W},\overline{W_2})$$
where $W$ is entire in $s$.
Thus to prove the lemma it is enough to show that each summand vanishes. Moreover, recalling that $\langle W,\tilde{W}\rangle^{-1}$ is an $s$-independent constant and appealing to the meromorphicity of the above summands it suffices to show that
$$\Psi(1-s,\mathcal{I}(\tilde{\pi}_1,s)(f_2)W,W_1)=0$$
for $\Re(s)=0$.
Letting $f_2^\iota(g):=f_2(g^{-\top})$ and 
applying the $\GL(n+1)\times\GL(n)$ local functional equation \eqref{LFE-m-m-1} to write the above zeta integral as
\begin{align*}
&{\Psi(s,\mathcal{I}({\pi}_1,-s)(f^\iota_2)\tilde{W},\tilde{W}_1)}{\gamma(s,\mathcal{I}(\pi_1,-s)\otimes\tilde{\pi}_1)}\\
&={\Psi(0,\mathcal{I}({\pi}_1,-s)(f^\iota_2)\tilde{W},\tilde{W}_1|\det|^{s})}\gamma(0,\pi_1\otimes\tilde{\pi}_1)\gamma(ns,\tilde{\pi}_1|\det|^s)
\end{align*}

Let us abbreviate ${\pi}_1|\det|^{-s}$ as $\pi$. Then it suffices to show that for a tempered principal series representation $\pi$ of $G_n(F)$ and $\mu\in\C$ away from a pole of $\gamma(\cdot,\tilde{\pi})$ and $\Re(\mu)$ sufficiently large, the quantity
$$\gamma(z,\pi\otimes\tilde{\pi}){\Psi(z,\pi_\mu(f^\iota_{2})W,W')}$$
vanishes at $z=0$ for $W\in\pi_\mu$ and $W'\in\tilde{\pi}$, both holomorphic in their corresponding Langlands parameters. Indeed, if that is the case, being a meromorphic function of $\mu$ the above vanishes for all $\mu$, in particular, for $\mu=ns$ as $\gamma(ns,\tilde{\pi})$ is holomorphic on $\Re(s)=0$.

We write
$$\pi_\mu(f^\iota_{2}){W}\sbr{\mat{g&\\&1}}=\int_{\overline{G_{n+1}}}f_2(h){W}\sbr{\mat{g&\\&1}h^{-\top}}\d h.$$
We parameterize $h$ as
$$h=\mat{\I_n&\\c&1}p,\quad c\in F^n, p\in P_{n+1}(F),\quad \d h= \d c \d_{\mathrm{right}} p.$$
Applying $N_n$-equivariance of ${W}$ we write the above integral as
$$\int_{P_{n+1}(F)}\left(\int_{F^n}\psi(e_ngc^\top)f_2\sbr{\mat{\I_n&\\c&1}p}\d c\right){W}\sbr{\mat{g&\\&1}p^{-\top}}\d_{\mathrm{right}} p.$$
Recalling \eqref{def-f-1} we can write the inner integral above as
$$\hat{f}^{21}_2(e_ng,p)=\hat{f}_1^{21}(e_ng/X,p)-\hat{f}_1^{21}(e_ng,p)$$
where for any function $f$ we write
$$\hat{f}^{21}(c,p):=\int_{F^n}f\sbr{\mat{\I_n&\\c'&1}p}\psi(c'c^\top)\d c'.$$
F
rom the compact support condition of $f_1$, we see that $\hat{f}_2^{21}(c,p)$ is compactly supported in $p$ and is Schwartz in $c$. Indeed, if $\mat{\I_n&\\c'&1}p\in B_\tau\mod Z_{n+1}(F)$ where $B_\tau$ is the ball of radius $\tau$ around the identity in $G_{n+1}(F)$ then $c'$ and $p$ vary in compact sets in their respective domains.

Moreover, using the mean value theorem we see that
$$\hat{f}_2^{21}(e_ng,p)=(1/X-1)(e_ng)\cdot\nabla_{\xi e_n g}\hat{f}^{21}_1(\cdot,p) \ll_X |e_ng|,$$
for some $1/X<\xi<1$.

Thus for $\Re(z)$ sufficiently large we can write the zeta integral $\Psi(z,\pi_\mu(f^\iota_2){W},W')$ as
\begin{equation*}
\int_{P_{n+1}(F)}\int_{N_n(F) \bs G_n(F)}\hat{f}^{21}_2(e_ng,p)\pi_\mu(p){W}\sbr{\mat{g&\\&1}}W'(g)|\det(g)|^{z-1/2}\d g \d_{\mathrm{right}}p,
\end{equation*}
Using the decay property of $\hat{f}^{21}_2(\cdot,p)$ and applying Lemma \ref{lem:holomorphic-gln+1-gln} we see that the inner integral vanishes at $z=0$. We conclude the proof by noting that $p$ varies over a compact set.
\end{proof}

\section{The degenerate term}

The goal of this section is to estimate the local components of the summands of the degenerate term $\D(s_1,s_2)$ as defined in \eqref{def-degenerate}. Recall the description of the local components from \eqref{final-I1-local} and \eqref{final-I2-local}.

\subsection{Uninteresting primes}

\begin{lem}\label{degenerate-unramified}
Let $v$ be an unramified place and let 
\begin{itemize}
    \item either $f$ be of either type I and $\infty>v\nmid\q$,
    \item or $f$ be of type II and $\infty>v\nmid \q\p_0\p_1$.
\end{itemize}

Then $$I^1(n^-,s_1,s_2)=\frac{L(1+s_1+s_2,\pi_1\otimes\tilde{\pi}_2)L(\frac{n+1}{2}+n s_1-s_2,\pi_2)}{L(\frac{n+3}{2}+(n+1)s_1,\pi_1)}$$
whenever $\Re(s_1),\Re(s_2)$ lie in the region $\mathcal{C}_1$, defined in \eqref{def-region-C1}, and
$$I^2(n^-,s_1,s_2)=\frac{L(1+s_1+s_2,{\pi}_1\otimes\tilde{\pi}_2)L(\frac{n+1}{2}+n s_1-s_2,\tilde{\pi}_1)}{L(\frac{n+3}{2}+(n+1)s_1,\widetilde{\pi}_2)}$$
whenever $\Re(s_1),\Re(s_2)$ lie in the region $\mathcal{C}_2$, defined in \eqref{def-region-C2}.
\end{lem}

\begin{proof}
We only prove the first equality. The second equality follows similarly as $f^-=f$. 

Suppose first that $v$ is unramified. We essentially follow the proof of Lemma \ref{meromorphic-cont-degenerate}. Let $\Re(s_1),\Re(s_2)\in\mathcal{C}_1$. The definition of $f$ yields that the domain of integration consists of those
$x \in G_n(F),y\in N_n(F)\bs G_n(F), b\in F^n$ such that
\begin{equation}\label{domain-of-integration-I1n}
  \begin{pmatrix} x^{-1}+be_ny&b\\ e_ny&1 \end{pmatrix}=\mat{x^{-1}&b\\&1}\mat{\I_n&\\e_ny&1}\in K_{n+1}\mod Z_{n+1}(F).
\end{equation}
We divide the integral in \eqref{final-I1-local} in two parts, namely, if $e_ny\in\o^n$ and $e_ny\notin\o^n$ and call them $A$ and $B$, respectively.

If $e_ny\in\o^n$ then
\eqref{domain-of-integration-I1n} is equivalent to $$x\in G_n(\o),\quad b\in\o^n.$$
As $W_{1}$ is spherical we obtain
\begin{align*}
A&=\int_{G_n(\o)}\d_{1+s_1} x\int_{\o^n}\d b\int_{N_n(F)\backslash G_n(F)} W_{1}(y)\overline{W_{2}(y)}\1_{\o^n}(e_ny)\d_{1+s_1+s_2}y\\
&=L(1+s_1+s_2,\pi_1\otimes\tilde{\pi}_2).
\end{align*}
The last equality follows from \eqref{unramified-zeta-integral-m}.

Now we consider $B$. Suppose that $e_ny\notin\o^n$. In that case, in Iwasawa coordinates we write 
$$N_n(F)\bs G_n(F)\ni y=z\diag(a,1)k;\quad z\in F^\times, a\in A_{n-1}(F), k\in K_n;$$
$$\d y=\d^\times z\frac{\d a}{\delta(a)|\det(a)|}\d k.$$
Note that 
$$e_ny\notin\o^n\iff ze_nk\notin\o^n\iff z\notin\o\iff z^{-1}\in\p,$$ 
and
$$\mat{\I_n&\\e_ny&1}=\mat{k^{-1}&\\&1}\mat{\I_{n-1}&&\\&z^{-1}&1\\&&z}\mat{\I_{n-1}&&\\&&-1\\&1&z^{-1}}\mat{k&\\&1}.$$
Hence, \eqref{domain-of-integration-I1n} is equivalent to
$$\mat{(kx)^{-1}&b\\&1}\mat{\I_{n-1}&&\\&z^{-1}&1\\&&z}\in K_{n+1}\mod Z_{n+1}(F).$$
Applying sphericality of $W_i$, changing variable $x\mapsto k^{-1}x$ we can write the contribution of the $y$-integral with $e_ny\notin\o^n$ in $I^1(n^{-},s_1,s_2)$ as
\begin{multline*}
\int_{G_n(F)}\int_{F^\times}\int_{F^n}{f}\left[\mat{x^{-1}&b\\&1}\mat{\I_{n-1}&&\\&z^{-1}&1\\&&z}\right]\1_{\p}(z^{-1})|z|^{n(1+s_1+s_2)}\d^\times z\d b\\
\int_{A_{n-1}(F)}W_{1}\sbr{\mat{a&\\&1}x}\overline{W_{2}\sbr{\mat{a&\\&1}}}\frac{\d_{s_1+s_2} a}{\delta(a)}\d_{s_1+1} x.
\end{multline*}
Once again applying the support condition of $f$ we obtain
$$\mat{\I_n&b+z^{-1}x^{-1}e_n^\top\\&1}\mat{z^{-1}x^{-1}\mat{\I_{n-1}&\\&z^{-1}}&\\&1}\in K_{n+1}\mod Z_{n+1}(F).$$
We see that the above is equivalent to
\begin{equation*}
   b+z^{-1}x^{-1}e_n^\top\in\o^n,\quad z^{-1}x^{-1}\mat{\I_{n-1}&\\&z^{-1}}\in K_n.
\end{equation*}
Changing variable $b\mapsto b-z^{-1}x^{-1}e_n^\top$ we evaluate the $b$-integral to $1$. Now applying $Z_nK_n$-invariance of $W_1$, and changing variables $a\mapsto az^{-1}$ and $z\mapsto z^{-1}$ we rewrite the above integral as
$$B=\int_{F^\times}\1_{\p}(z)|z|^{1+ns_1-s_2}
\int_{A_{n-1}(F)}W_{1}\sbr{\mat{a&\\&1}}\overline{W_{2}\sbr{\mat{az&\\&1}}}\frac{\d_{s_1+s_2} a}{\delta(a)}\d^\times z.$$
Now the support condition of $W_{1}$ as in \eqref{shintani} and the condition $z\in\p$ we automatically have that the $za_n\in\o$. Thus we can use Lemma \ref{decomposition-lowerdegree}, which in this context (using similar notations) tells us that
\begin{equation*}
	W_{2}\sbr{\mat{az&\\&1}}= |\det(za)|^{1/2}\sum_{j=1}^n\frac{\mu_j^{-1}}
	{\prod_{i\neq j}(\mu_i-\mu_j)}\chi_j^{-1}(z)W_{\pi_{2,j}}(a),
\end{equation*}
where $\chi_j(z):=\mu_j^{\nu_\p(z)}$ and $\nu_\p$ denotes the $\p$-adic valuation.
Then we have
\begin{equation*}
  B=\sum_{j=1}^n\frac{\overline{\mu_j}^{-1}}{\prod_{i\neq j}(\overline{\mu}_i-\overline{\mu}_j)} B_{1,j}B_{2,j}.
\end{equation*}
Here using sphericality of $W_1$ and $W_{\pi_{2,j}}$, and applying \eqref{unramified-zeta-integral-m-1} we write
\begin{align*}
B_{1,j}: &=\int_{N_{n-1}(F)\backslash G_{n-1}(F)} W_{1}(\diag(h,1))
\overline{W_{\pi_{2,j}}(h)}|\det(h)|^{s_1+s_2+1/2}\d h\\
&=L(1+s_1+s_2,\pi_1\otimes\overline{\pi}_{2,j})
\end{align*}
From the definition of the $L$-factor and the fact that $\pi_2$ is unitary,
we have
\begin{align*}
B_{1,j} &=L(1+s_1+s_2,\pi_1\otimes \tilde{\pi}_2)L(1+s_1+s_2,\pi_1
\otimes\overline{\chi}_j)^{-1}\\
&=L(1+s_1+s_2,\pi_1\otimes \tilde{\pi}_2)\sum_{k=0}^n\frac{(-1)^ke_k(\mu_{\pi_1})
\overline{\mu}_j^k}{N(\p)^{k(1+s_1+s_2)}}
\end{align*}
where $e_k(\cdot)$ denotes the $k$'th elementary symmetric polynomial.

On the other hand,
\begin{equation*}
  B_{2,j}:=\int_{z\in\p}\overline{\chi}^{-1}_j(z)|z|^{ns_1-s_2+(n+1)/2}\d^\times z
  =\sum_{m\ge 1}\frac{\overline{\mu}_j^{-m}}{N(\p)^{m(ns_1-s_2+(n+1)/2)}}
\end{equation*}
Thus $B$ equals
\begin{equation*}
L(1+s_1+s_2,\pi_1\otimes\widetilde{\pi}_2)\sum_{m\ge 1}\sum_{k=0}^n\frac{(-1)^ke_k(\mu_{\pi_1})}
{N(\p)^{m(ns_1-s_2+(n+1)/2)+k(1+s_1+s_2)}}
        \sum_{j=1}^n\frac{\overline{\mu_j}^{k-1-m}}
  {\prod_{i\neq j}(\overline{\mu}_i-\overline{\mu}_j)}
\end{equation*}

It now follows from a computation done in
\cite[Proof of Lemma 9.4]{JaNu2021reciprocity} that the inner sum vanishes
unless $m\ge k$, in which case it equals $\chi_{(m-k,0,\dots,0)}(\mu_{\pi_2})$.
Therefore, changing variables, we have that
$$B=C-L(1+s_1+s_2,\pi_1\otimes\tilde{\pi}_2),$$
where
\begin{align*}
	C & =L(1+s_1+s_2,\pi_1\otimes\tilde{\pi}_2)\sum_{m \ge 0}
	\frac{\chi_{(m,0,\dots,0)}(\mu_{\pi_2})}{N(\p)^{m (\frac{n+1}2+ns_1-s_2)}}
	\sum_{k=0}^{n}\frac{(-1)^ke_k(A_{\pi_1})}{N(\p)^{k(\frac{n+3}2+(n+1)s_1)}}\\
	  & =\frac{L(1+s_1+s_2,\pi_1\otimes\tilde{\pi}_2)L(\frac{n+1}2+ns_1-s_2,\pi_2)}
	{L(\frac{n+3}2+(n+1)s_1,\pi_1)}.
\end{align*}
Recalling that $I^1(n^-,s_1,s_2)=B+L(1+s_1+s_2,\pi_1\otimes\tilde{\pi}_2)$ we conclude.
\end{proof}

\subsection{Level primes}

\begin{lem}\label{degenerate-level}
Let $f=\tilde{\1}_{\overline{K_0}(\p^{e})}$. Then we have
$$I^j(n^-,s_1,s_2)=\frac{\zeta(1)}{\zeta(n+1)}N(\p)^{-ne(s_1+s_2)}L(1+s_1+s_2,\pi_1\otimes\tilde{\pi}_2)$$
for $j=1,2$ and whenever $\Re(s_1),\Re(s_2)\in\mathcal{C}_j$.
\end{lem}

\begin{proof}
We only prove the $j=1$ case. The $j=2$ case follows similarly as $f^-=f$.

Let $\Re(s_1),\Re(s_2)$ be sufficiently large. Recall that from Lemma \ref{global-I1} that $I^1(n^{-},s_1,s_2)$ equals
\begin{equation*}
\int_{G_n(F)}\int_{N_n(F)\backslash G_n(F)}\int_{F^n}{f}\left[\begin{pmatrix} x^{-1}+be_ny&b\\ e_ny&1 \end{pmatrix}\right] W_{1}(yx)\overline{W_{2}(y)}\d b\d_{s_1+s_2+1} y \d_{s_1+1} x.
\end{equation*}
The support condition of $f$ yields that the domain of integration is $x \in G_n(F)$, $y\in N_n(F)\bs G_n(F)$, $b\in F^n$ such that
$$\begin{pmatrix}x^{-1}+be_ny&b\\ e_ny&1 \end{pmatrix}=\mat{x^{-1}&b\\&1}\mat{\I_n&\\e_ny&1}\in \overline{K_0}(\p^e)\mod Z_{n+1}(F).$$
Note that we necessarily have $e_ny\in\p^e\o^n$. Thus automatically we have
$$x\in G_n(\o),\quad b\in\o^n.$$
As $W_1$ is spherical $I^1(n^{-},s_1,s_2)$ equals
\begin{equation*}
\vol(\overline{K_0}(\p^e))^{-1}\int_{G_n(\o)}\d_{1+s_1} x\int_{\o^n}\d b\int_{N_n(F)\backslash G_n(F)} W_{1}(y)\overline{W_{2}(y)}\1_{\p^e\o^n}(e_ny)\d_{1+s_1+s_2}y.
\end{equation*}
Changing variables $y\rightarrow \p^ey$ and applying \eqref{unramified-zeta-integral-m} we obtain that the above equals
$$\vol(\overline{K_0}(\p^e))^{-1}N(\p)^{-ne(1+s_1+s_2)}L(1+s_1+s_2,\pi_1\otimes\tilde{\pi}_2).$$
We conclude the proof by noting that
$$\frac{N(\p)^{-en}}{\vol(\overline{K_0}(\p^e))}=\frac{\zeta(1)}{\zeta(n+1)}$$
and meromorphic continuation.
\end{proof}

\subsection{Supercuspidal place}

\begin{lem}\label{degenerate-sc}
    Let $f$ be of type II and $v=\p_0$. Then $I^j(n^-,s_1,s_2)=0$ for $s_1,s_2$ with $\Re(s_1),\Re(s_2)\in\mathcal{C}_j$ for $j=1,2$,
\end{lem}

\begin{proof}
    As $f$ is a matrix coefficient of a unitary representation $f^-$ is also the same. Thus proving the lemma for $j=1$ is enough for $f$ being a matrix coefficient of a supercuspidal representation. 
    
    On the other hand, if $f$ is a matrix coefficient then $f(\cdot g)$ for any $g\in\overline{G_{n+1}}(F)$ is also the same. Thus recalling \eqref{final-I1-local} it suffices to show that
    $$\int_{F^n}f\sbr{\mat{\I_n&b\\&1}}\d b=0$$
    where $f$ is a matrix coefficient of a supercuspidal representation.

    Apply compact support of $f$ we write the above integral as a limit of $N\to\infty$ of
    $$\int_{F^n}\1_{\o}(\p_0^Nb)f\sbr{\mat{\I_n&b\\&1}}\d b.$$
    Let $f(g)=\langle\sigma(g)V_1,V_2\rangle_\sigma$ for $V_i\in\sigma$. Then the above integral, using unipotent equivariance of $V_1$, can be written as
    \begin{multline*}
    \int_{N_{n-1}(F)\bs G_{n-1}(F)}V_1\overline{V_2}\sbr{\mat{h&\\&1}}\int_{F^n}\1_{\o}(\p_0^Nb)\psi(e_{n-1}hb)\d b\d h\\
    =N(\p_0)^{-N}\int_{N_{n-1}(F)\bs G_{n-1}(F)}V_1\overline{V_2}\sbr{\mat{h&\\&1}}\1_\o(e_{n-1}h\p_0^{-N})\d h.
    \end{multline*}
    As $V_i\mid_{G_{n-1}(F)}$ are compactly supported the right-hand side above vanishes for large $N$. Hence we conclude.
\end{proof}

\begin{rmk}
  When $v$ is ramified, changing variables in \eqref{final-I1-local} or 
  \eqref{final-I2-local}, we can reduce to the unramified computation
  up to a factor of $\Delta^{\mu^D(s_1,s_2)}$, where $\mu^D$ is affine and its
  coefficients only depend on $n$. 
\end{rmk}

\subsection{Archimedean places}

\begin{lem}\label{degenerate-arch}
Let $f$ be of type I and $v\mid\infty$. Then $I^j(n^-,s_1,s_2)$ are holomorphic for $\Re(s_1),\Re(s_2)\in\mathcal{C}_j$ and have meromorphic continuation for $s_1,s_2\in\C$. Moreover, we have $$I^j(n^-,s_1,s_2)=O\left(X^{-n\Re(s_1+s_2)}\right)$$ for $(s_1,s_2)$ at a neighbourhood of $(0,0)$.
\end{lem}

\begin{proof}
We prove the lemma for $j=1$ only using the support condition of $f$. As $f^-$ has a similar support condition as $f$ a similar proof works for $j=2$.

Recall \eqref{final-I1-local}. Note that if
\begin{equation*}
   \begin{pmatrix} x^{-1}+be_ny&b\\e_ny&1 \end{pmatrix}\in K_0(X,\tau)\mod Z_{n+1}(F)
\end{equation*}
then we immediately find fixed compact sets $C_1, C_2\subset F^n$ and a sufficiently small compact neighbourhood $C_3$ of the identity in $G_n(F)$ so that 
$$f\sbr{\begin{pmatrix} x^{-1}&b\\&1\end{pmatrix}\mat{\I_n&\\ c&1}}\neq 0\text{ only if }(Xc,b,x)\in C_1\times C_2 \times C_3.$$
Consequently, $I^1(n^-,s_1,s_2)$ can be written as
$$\int_{C_3}\int_{C_2}\Psi\left(1+s_1+s_2,\pi_1(x)W_1,\overline{W_2},f\sbr{\begin{pmatrix} x^{-1}&b\\&1\end{pmatrix}\mat{\I_n&\\ \cdot&1}}\right)\d b\d_{1+s_1} x.$$
The above integral is absolutely convergent if $\Re(s_1+s_2)>-1$ and uniformly in $s_1,s_2$ in compact sets in this region (see \cite[Lemma 3.1]{JaNu2021reciprocity}) and thus can be meromorphically continued to all $s_1,s_2$; see \cite[Lecture 2]{cogdell2007functions}. We conclude the proof of the first assertion by appealing to the compact support of $x$ and $b$ integrals.

Now for $\Re(s_1+s_2)>-1$, after change of variable $y\mapsto X^{-1} y$, we estimate $I^1(n^-,s_1,s_2)$ by
$$\ll X^{-n\Re(s_1+s_2)}\int_{x\in C_3}\int_{e_ny\in C_1}W_{1}(yx)\overline{W_{2}(y)}\d_{\Re(s_1)+\Re(s_2)+1}y \d x.$$
Noting that the last $y$-integral is absolutely convergent for $\Re(s_1+s_2)>-1$ we conclude the proof of the second assertion.
\end{proof}

\part{Endgame}

In this part, we prove the main theorems, namely, Theorem \ref{second-moment-nonarch}, Theorem \ref{second-moment-arch}, Theorem \ref{mixed-moment-2}, and Corollary \ref{non-vanishing}. In the rest of the paper, we assume $F$ is a global field. We also recall $\q=\prod_{v\mid\q}\p_v^{e_v}$ and $X:=\prod_{v\mid\infty} X_v$. We assume that $N(\q)X$ is sufficiently large.

\section{Assembling everything}

Let $f$ be of either type I or of type II; see \S\ref{type-1-test-function} and \S\ref{type-2-test-function}, respectively. Let $\pi_1,\pi_2$ be cuspidal automorphic representations for $G_n(F)$ with trivial central characters, such that $\pi_i$ are tempered and unramified at every place. Let $s_1,s_2\in\C$ such that $-\epsilon<\Re(s_1),\Re(s_2)<1/2$. At this point, we assume that $s_1,s_2$ are in generic position, in particular, $s_1+s_2$ avoids any pole of $\Lambda(\cdot, \pi_1\otimes\tilde{\pi}_2)$. In this section, we will assemble two different expressions of $\mathcal{P}^*$, as defined in \eqref{def-P*}, obtained in the previous parts.

\subsection{The relative trace formula}

We start by recalling the expression \eqref{regularizing-P*} for $\mathcal{P}^\ast$. Applying \eqref{decomposition-P-eta}, Proposition \ref{doublecosets-0} and Proposition \ref{doublecosets-with-unipotent} we obtain that $\mathcal{P}^\ast$ is the limit as $T\to\infty$ of
\begin{multline*}
     I^0(e,T)+I^0(n^+,T)+I^0(n^-,T)+I^0(\xi^\perp,T)+\sum_{0,1\neq t\in F}I^0(\xi(t),T)\\
     +I^0(w',T)+I^0(n^+w',T)+I^0(w'n^+,T) -I^1(e,T)-I^1(n^-,T)-I^1(w',T)\\
     -I^2(e,T)-I^2(n^-,T)-I^2(w',T) +I^{12}(e,T)+I^{12}(w',T).\\
\end{multline*}

We apply \eqref{apply-poisson}, \eqref{vanish-xi-t-contribution}, Lemma
\ref{vanish-genuine-orbital}, Lemma \ref{some-unip-orbital-integral-equal},
Lemma \ref{vanish-unip-orbital-integral} to reduce the above to
\[
    \tilde{I}^+(T)+I^0(n^-,T)-I^1(n^-,T)-I^2(n^-,T)+I^0(\xi^\perp,T).
\]
We now let $T\to\infty$. The limit exists thanks to Lemma \ref{final-for-Itilde+},
Lemma \ref{final-for-I-}, Lemma \ref{global-I1}, Lemma \ref{global-I2}, Proposition
\ref{rapid-decay-xi-perp-type-1}, and Proposition \ref{rapid-decay-xi-perp-type-2}.
Applying \eqref{def-degenerate} and Proposition \ref{main-expression-spectral-side}, and letting
$$\MT(s_1,s_2):=\tilde{I}^+(s_1,s_2)+I^-(s_1,s_2)$$
we obtain
\begin{equation}\label{main-junction-equation}
    \M=\MT-\D -\R + I^0(\xi^\perp),
\end{equation}
all evaluated at $s_1,s_2\in\C$ with $-\epsilon<\Re(s_1),\Re(s_2)<1/2$.

\subsection{The main term}

\begin{lem}\label{lem:main-term-type-1}
Let $f$ be of type I. Moreover, assume that $\pi_1\cong\pi_2$ and $W_1=W_2$. Then $\MT(s_1,s_2)$ is holomorphic in a neighbourhood of $(0,0)$. Moreover,
\begin{multline*}
\frac{\MT(0,0)}{I_\infty\Delta^n\vol(\overline{K_0}(\q))^{-1}X^n}=
\zeta^*(1)L(1,\pi_1,\Ad)\left(n\log N(\q)-\sum_{v\mid\infty}
\frac{\partial_{s=0}I^-_v(0,s)}{I_{0,v}X_v^n}\right)\\
+\sum_{v\mid\infty}\frac{\partial_{s=0}
\tilde{I}^+_v(s,0)}{I_{0,v}X_v^n}+J_F+2
\partial_{s=0}\left(sL(1+s,\pi_1\otimes\tilde{\pi}_1)\right)
\end{multline*}
where $J_F$ only depends on the discriminant of $F$ and $n$, and $I_\infty$ depends only on the choice of test vectors and test functions at the archimedean places.
\end{lem}

\begin{proof}
Applying Lemma \ref{final-for-Itilde+}, Lemma \ref{final-for-I-},
Lemma \ref{level-primes-I+}, and Lemma \ref{level-primes-I-},
along with Remarks \ref{remark-ramified-I+} and \ref{remark-ramified-I-},
we obtain that $\MT(s_1,s_2)$ equals
\begin{multline*}
\vol(\overline{K_0}(\q))^{-1}\left(\Delta^{\mu^+(s_1,s_2)}L(1+s_1+s_2,\pi_1\otimes\tilde{\pi}_1)\tilde{I}^+_\infty(s_1,s_2)\right.\\
+\left.\Delta^{\mu^-(s_1,s_2)}N(\q)^{-n(s_1+s_2)}L(1-s_1-s_2,\pi_1\otimes\tilde{\pi}_1)I^-_\infty(s_1,s_2)\right).
\end{multline*}
From the theory of global $L$-function (see \cite[Theorem 4.2]{cogdell2007functions}) we know that $L(s,\pi_1\otimes\tilde{\pi}_1)$ is holomorphic in $\C\setminus\{1\}$ and has a simple pole at $s=1$ with residue $\zeta^*(1)L(1,\pi_1,\Ad)$. Thus applying Lemma \ref{archimedean-I+type-1} and Lemma \ref{archimedean-I-type-1} we see that the above is holomorphic in $s_1,s_2$ at a neighbourhood of $(0,0)$, avoiding the hyperplane $s_1+s_2=0$.

Now we claim that the above is also holomorphic on $s_1+s_2=0$. 
We write 
$$L(1+s,\pi_1\otimes\tilde{\pi}_1)=\frac{\zeta^*(1)L(1,\pi_1,\Ad)}{s}+C_{\pi_1}+K_{\pi_1}(s)$$
where $K_{\pi_1}$ is an entire function on $\C$ with $K_{\pi_1}(0)=0$, and $C_{\pi_1}\in\C$. Thus the quantity $\vol(\overline{K_0}(\q))\MT(s_1,s_2)$ can be rewritten as
\begin{multline*}
\zeta^*(1)L(1,\pi_1,\Ad)\frac{\Delta^{\mu^+(s_1,s_2)}\tilde{I}^+_\infty(s_1,s_2)-\Delta^{\mu^-(s_1,s_2)}N(\q)^{-n(s_1+s_2)}I^-_\infty(s_1,s_2)}{s_1+s_2}\\
+\left(C_{\pi_1}+K_{\pi_1}(s_1+s_2)\right)\Delta^{\mu^+(s_1,s_2)}\tilde{I}^+_\infty(s_1,s_2)\\
+\left(C_{\pi_1}+K_{\pi_1}(-s_1-s_2)\right)\Delta^{\mu^-(s_1,s_2)}N(\q)^{-n(s_1+s_2)}I^-_\infty(s_1,s_2).
\end{multline*}
The last two summands are holomorphic at a neighbourhood of $(0,0)$. Thus to prove the holomorphicity of $\MT(s_1,s_2)$ at a neighbourhood of $(0,0)$ it suffices to show that the limit of the first summand exists as $s_1+s_2\to 0$. Applying Lemma \ref{archimedean-I-type-1} we have
$$\tilde{I}^+_\infty(0,0)=I^-_\infty(0,0)=\prod_{v\mid\infty}I_{0,v}X_v^n=I_\infty X^n,\quad I_\infty:=\prod_{v\mid\infty}I_{0,v}.$$
Employing this fact and the fact that $\mu^+(0,0)=\mu^-(0,0)$, we may write the limit as $s_1+s_2\to 0$ of the first summand above as
\begin{multline*}
    \zeta^*(1)L(1,\pi_1,\Ad)\left(\lim_{s\to0}\frac{\Delta^{\mu^+(s,0)}\tilde{I}^+_\infty(s,0)-\Delta^{\mu^+(0,0)}\tilde{I}^+_\infty(0,0)}{s}\right.\\
    \left.-\lim_{s\to 0}\frac{\Delta^{\mu^-(0,s)}N(\q)^{ns}I^-_\infty(0,s)-\Delta^{\mu^-(0,0)}I^-_\infty(0,0)}{s}\right).
\end{multline*}
The above limits exist and the above expression equals
$$\zeta^*(1)L(1,\pi_1,\Ad)\left(\partial_{s=0}\Delta^{\mu^+(s,0)}\tilde{I}^+_\infty(s,0)-\partial_{s=0}\Delta^{\mu^-(0,s)}N(\q)^{-ns}I^-_\infty(0,s)\right).$$
This completes the proof of the holomorphicity of $\MT(s_1,s_2)$ in a neighbourhood of $(0,0)$.

The above computation along with Lemma \ref{archimedean-I+type-1} also shows that
\begin{multline*}
    \frac{\MT(0,0)}{\vol(\overline{K_0}(\q))^{-1}}=2\partial_{s=0}
    \left(sL(1+s,\pi_1\otimes\tilde{\pi}_1)\right)\Delta^{\mu^\pm}I_\infty X^n\\
    +\zeta^*(1)L(1,\pi_1,\Ad)\partial_{s=0}\left(\Delta^{\mu^+(s,0)}
    \tilde{I}^+_\infty(s,0)-\Delta^{\mu^-(0,s)}N(\q)^{-ns}I^-_\infty(0,s)\right),
\end{multline*}
where $\mu^{\pm}$ is the common value of $\mu^\pm(0,0)$ and $\mu^-(0,0)$.
Applying the product differentiation rule and employing Lemma \ref{archimedean-I+type-1} and Lemma \ref{archimedean-I-type-1} the derivative in the second summand above can be evaluated as
\begin{multline*}
    \left(\partial_{s=0}\Delta^{\mu^+(s,0)}\right)I_\infty X^n+\Delta^{\mu^\pm}
    I_\infty X^n\sum_{v\mid\infty}\frac{\partial_{s=0}\tilde{I}^+_v(s,0)}{I_{0,v}X_v^n}\\
    -\left(\partial_{s=0}\Delta^{\mu^-(0,s)}\right)I_\infty X^n+n\Delta^{\mu^\pm}
    \log N(\q) I_\infty X^n - \Delta^{\mu^\pm}I_\infty X^n\sum_{v\mid\infty}\frac{\partial_{s=0}I^-_v(0,s)}{I_{0,v}X_v^n}.
\end{multline*}
Combining everything completes the proof.
\end{proof}

\begin{lem}\label{lem:main-term-type-2}
Let $f$ be of type II. Moreover, assume that $\pi_1\ncong\pi_2$.
Then $\MT(s_1,s_2)$ is holomorphic at a neighbourhood of $(0,0)$.
Moreover, there is a choice of $\mu$ (see the choice of $f_{\p_1}$ in
\S\ref{type-2-test-function}) so that
$$\MT(0,0)=J_{\mathfrak{X}}\,\vol(\overline{K_0}(\q))^{-1}X^n$$
with $J_{\mathfrak{X}}\asymp 1$.
\end{lem}

\begin{proof}
Applying Lemma \ref{final-for-Itilde+}, Lemma \ref{final-for-I-}, Lemma \ref{level-primes-I+}, and Lemma \ref{level-primes-I-}, Lemma \ref{aux+I}, Lemma \ref{aux-I}, Lemma \ref{sc-prime-I+}, Lemma \ref{sc-prime-I-}, along with Remarks \ref{remark-ramified-I+} and \ref{remark-ramified-I-}. We obtain that $\MT(s_1,s_2)$ equals
\begin{multline*}
\vol(\overline{K_0}(\q))^{-1}\Big(\Delta^{\mu^+(s_1,s_2)}\overline{\lambda_{\pi_{2,\p_1}}(\p_1^\mu)}N(\p_1)^{-|\mu|s_2}L^{\p_0}(1+s_1+s_2,\pi_1\otimes\tilde{\pi}_2)\tilde{I}^+_\infty(s_1,s_2)\Big.\\
+\Delta^{\mu^-(s_1,s_2)}\lambda_{\tilde{\pi}_{1,\p_1}}(\p_1^\mu)N(\p_1)^{|\mu|s_1-\nu(s_1+s_2)}C(\sigma)^{-n(s_1+s_2)}N(\q)^{-n(s_1+s_2)}\\
\times\Big.L^{\p_0}(1-s_1-s_2,\tilde{\pi}_1\otimes\pi_2)I^-_\infty(s_1,s_2)\Big).
\end{multline*}
From the theory of global $L$-function (see \cite[Theorem 4.2]{cogdell2007functions}) we know that $L(s,\pi_1\otimes\tilde{\pi}_2)$ is entire as $\pi_1\ncong\pi_2$. Thus applying Lemma \ref{archimedean-I+type-2} and Lemma \ref{archimedean-I-type-2} we see that $\MT(s_1,s_2)$ is holomorphic at a neighbourhood of $(0,0)$.

Moreover, plugging in $(0,0)$ in the above expression we obtain
$$\MT(0,0)=\vol(\overline{K_0}(\q))^{-1}X^n\left(\overline{\lambda_{\pi_{2,\p_1}}(\p_1^\mu)}A+\lambda_{\tilde{\pi}_{1,\p_1}}(\p_1^\mu)B\right),$$
where
$$A:=\Delta^{\mu^\pm} L^{\p_0}(1,\pi_1\otimes\tilde{\pi}_2)
\frac{\tilde{I}^+_\infty(0,0)}{X^n},\quad
B:=\Delta^{\mu^\pm} L^{\p_0}(1,\tilde{\pi}_1\otimes\pi_2)
\frac{I^-_\infty(0,0)}{X^n}.$$
It follows from the facts (see \cite{shahidi1980nonvanishing})
$$L^{\p_0}(1,\pi_1\otimes\tilde{\pi}_2), L^{\p_0}(1,\tilde{\pi}_1\otimes\pi_2)\neq 0,$$
and Lemma \ref{archimedean-I+type-2} and Lemma \ref{archimedean-I-type-2} that
$$A\asymp 1,\quad B\asymp 1,\quad\text{as } N(\q)X\to\infty.$$
Now we choose $\mu'\in\Z_{\ge 0}^n$ so that
$$\lambda_{\pi_1}(\p_1^{\mu'})\neq \lambda_{\pi_2}(\p_1^{\mu'}).$$
Such a $\mu'$ exists as $\pi_1\ncong\pi_2$. Therefore, both 
$$A+B=o(1),\quad\overline{\lambda_{\pi_{2,\p_1}}(\p_1^{\mu'})}A+\lambda_{\tilde{\pi}_{1,\p_1}}(\p_1^{\mu'})B=o(1)$$
can not be valid simultaneously. Hence there is a choice of $\mu$ so that
$$\overline{\lambda_{\pi_{2,\pi_1}}(\p_1^{\mu})}A+\lambda_{\tilde{\pi}_{1,\p_1}}(\p_1^{\mu})B\asymp 1$$
and we conclude.
\end{proof}

\section{Proof of the theorems}

\subsection{Proof of Theorem \ref{second-moment-nonarch} and Theorem \ref{second-moment-arch}}

In this subsection, we assume $\pi_1\cong\pi_2$ and that $f$ is of type I.

From the analytic properties of the $L$-functions (see \cite[Theorem 4.2]{cogdell2007functions}) and recalling that the sum in \eqref{spectral-side} converges absolutely we see that $\mathcal{M}(s_1,s_2)$ is holomorphic in a neighbourhood of $(0,0)$. Thus, applying Proposition \ref{rapid-decay-xi-perp-type-1} and Lemma \ref{lem:main-term-type-1} to \eqref{main-junction-equation}, we conclude that $\D(s_1,s_2)+\R(s_1,s_2)$ is holomorphic in a neighbourhood of $(0,0)$. Moreover, from Proposition \ref{main-prop-degenerate} and Remark \ref{upgrade-Ijn-holomorphic} we conclude that only possible pole of $\D(s_1,s_2)$ in a sufficiently small neighbourhood of $(0,0)$ is along the hyperplane $s_1+s_2=0$. Thus the same must be true for $\R(s_1,s_2)$. Thus using Cauchy's residue theorem we have
\begin{align}\label{remove-singularity-D-R}
    (\D+\R)(0,0) &= \partial_{s=0}\, s\left(\D(s/2,s/2)+\R(s/2,s/2)\right)\notag\\
    &= \frac{1}{2\pi i}\int_{|s|=\epsilon}\frac{\D(s/2,s/2)+\R(s/2,s/2)}{s}{\d s}
\end{align}
for some $\epsilon>0$ sufficiently small but fixed. 

\begin{proof}[Proof of Theorem \ref{second-moment-nonarch}]
    We choose $X=1$ (\emph{i.e.}, $X_v=1,\,\forall\, v\mid\infty$) and $f$ to be of type I and name $\pi_1\cong\pi_2=\pi$. We define $\H=\H_{\q}$ so that (see \S\ref{sec:spectral-side})
    \begin{equation}\label{recallin-H-and-h}
      H_v:=H_v(\cdot;0,0,W_{1,v},W_{1,v},f_v), v<\infty;\quad
      h_v:=h_v(\cdot; 0,0,W_{1,v},W_{1,v},f_v), v\mid\infty.
    \end{equation}
    By definition,
    $$\mathbb{M}(\pi,\H_\q)=\M(0,0,\phi_1,\phi_2,f).$$
    Now from \eqref{main-junction-equation}, Lemma \ref{lem:main-term-type-1}, and Proposition \ref{rapid-decay-xi-perp-type-1} we obtain
    \begin{align*}
        \M(0,0)&=\MT(0,0)-(\D+\R)(0,0) + O_A(N(\q)^{-A})\\
        &=n\Delta^nI_\infty\zeta^*(1)L(1,\pi,\Ad)\vol(\overline{K_0}(\q))^{-1}\left(\log N(\q)+D_{\pi}\right) \\
        &+\mathrm{SMT}(\pi,\q)+  O_A\left(N(\q)^{-A}\right),
    \end{align*}
    where
    \begin{multline*}
        \mathrm{SMT}(\pi,\q):=\frac{\zeta_\q(1)}{\zeta_\q(n+1)}
        \partial_{s=0}\,s\Bigg[N(\q)^{-ns}L(1+s,\pi\otimes\tilde{\pi})\\
        \left(\frac{L^\q(\frac{n+1}{2}+\frac{n-1}{2}s,\pi)}
        {L^\q(\frac{n+3}{2}+\frac{n+1}{2}s,\pi)}
        I^1_\infty(n^-,s/2,s/2)+\frac{L^\q(\frac{n+1}{2}+\frac{n-1}{2}s,\tilde{\pi})}
        {L^\q(\frac{n+3}{2}+\frac{n+1}{2}s,\tilde{\pi})}I^2_\infty(n^-,s/2,s/2)\right)\\
        +r_FL(s,\pi\otimes\tilde{\pi}) \left(\frac{L^\q(\frac{n+1}{2}-\frac{n-1}{2}s,
        \tilde{\pi})}{L^\q(\frac{n+3}{2}-\frac{n+1}{2}s,\tilde{\pi})}
        h_\infty(\mathcal{I}(\tilde{\pi},1/2)_\infty)\right.\\
        +\left.\frac{L^\q(\frac{n+1}{2}+\frac{n-1}{2}s,\pi)}{L^\q(\frac{n+3}{2}
        -\frac{n+1}{2}s,\pi)} h_\infty(\mathcal{I}(\tilde{\pi},-1/2)_\infty)\right)\Bigg].
    \end{multline*}
The expression of $\mathrm{SMT}$ follows from \eqref{remove-singularity-D-R}, Proposition \ref{main-prop-degenerate}, and \eqref{residue-term}. Thus writing $L^\q=L/L_\q$ we have
$$\mathrm{SMT}(\pi,\q)=C\log N(\q)+D$$
where $C$ is given by a $\q$-independent linear combination of 
\eqref{quotients-of-Lfunctions} and $D$ is given by a $\q$-independent linear combination of the terms in
\eqref{quotients-of-Lfunctions} and \eqref{derivatives-quotients-Lfunctions}.
On the other hand, the required properties of $\H_\f$ follow from Lemma \ref{upper-bounds-spectral-nonarchimedean} and Lemma \ref{lower-bounds-spectral-nonarchimedean}. Finally, noting that 
$$\vol(\overline{K_0}(\q))^{-1}=\frac{\zeta_\q(1)}{\zeta_\q(n+1)}N(\q)^n$$we conclude the proof.
\end{proof}

\begin{proof}[Proof of Theorem \ref{second-moment-arch}]
    We follow a similar proof as of Theorem \ref{second-moment-nonarch}.  We choose $\q=1$ and $f$ to be of type I and name $\pi_1\cong\pi_2=\pi$. We define $\H=\H_{\mathfrak{X}}$
    as in \eqref{recallin-H-and-h}. Again by definition,
    $$\mathbb{M}(\pi,\H_\mathfrak{X})=\M(0,0,\phi_1,\phi_2,f).$$
    Now from \eqref{main-junction-equation}, Lemma \ref{lem:main-term-type-1}, Lemma \ref{archimedean-I+type-1}, Lemma \ref{archimedean-I-type-1}, and Proposition \ref{rapid-decay-xi-perp-type-1} we obtain
    \begin{align*}
        \M(0,0)&=\MT(0,0)-(\D+\R)(0,0) + O_A(N(\q)^{-A})\\
        &=n\Delta^n X^n\zeta^*(1)L(1,\pi,\Ad)\left(\log X+O(1)\right) +(\D+\R)(0,0)+  O_A(X^{-A}).
    \end{align*}
    Now we estimate $(\D+\R)(0,0)$ using the integral representation in \eqref{remove-singularity-D-R}. We choose $\epsilon$ such that $\R(s,0)$ is fixed distance away from all poles of it and the $h_\infty$ in \eqref{residue-term}. Applying Proposition \ref{main-prop-degenerate}, Lemma \ref{degenerate-arch}, and Proposition \ref{main-prop-residue-arch}
    $$(\D+\R)(0,0)\ll_\epsilon X^\epsilon.$$
    Finally, the required properties of $\H_{\mathfrak{X}}$ follow from Lemma \ref{localization}.
\end{proof}

\begin{proof}[Proof of Theorem \ref{hybrid-moment}]
    This is immediate from the proofs of Theorem \ref{second-moment-nonarch} and Theorem \ref{second-moment-arch}. Indeed, we choose $f$ to be of type I with parameter $\q$ and $\mathfrak{X}$. Then for the corresponding $\H(f)=\H_\f\H_\infty$ we have from the proofs of Theorem \ref{second-moment-nonarch} and Theorem \ref{second-moment-arch}
    $$\M(\pi,\H)\ll \vol(\overline{K_0}(\q))^{-1}X^n(\log N(\q)+\log X)\ll_\epsilon (N(\q)X)^{n+\epsilon}.$$
    On the other hand, for cuspidal $\Pi$ one has $\ell(\Pi)\asymp L(1,\Pi,\Ad)$, and consequently,
    $$\M(\pi,\H)\ge\sum_{\substack{\Pi\text{ cuspidal}\\\mathfrak{C}(\Pi_\f)\mid\q\\C(\Pi_v)\le X_v,\,v\mid\infty}}\frac{|L(1/2,\Pi\otimes\pi)|^2}{L(1,\Pi,\Ad)}\H_\f(\Pi_\f)\H_\infty(\Pi_\infty).$$
    We conclude the proof using properties of $\H_\f$ and $\H_\infty$ from Theorem \ref{second-moment-nonarch} and Theorem \ref{second-moment-arch}, respectively.
\end{proof}

\subsection{Proof of Theorem \ref{mixed-moment-2}}

In this subsection, we assume $\pi_1\ncong\pi_2$ and that $f$ is of type II.
We define $\H$ again as in \eqref{recallin-H-and-h}. From Lemma \ref{supercuspidal-projects-cuspidal} we yield that if $\H(\Pi)\neq 0$ then $\Pi$ must be cuspidal. We define
$$\mathcal{F}:=\{\Pi\text{ cuspidal for }\PGL_{n+1}(F)\mid\Pi_v\text{ unramified for } \infty>v\nmid\q\p_0\p_1\}$$
Applying Lemma \ref{bounds-spectral-nonarchimedean} and Remark \ref{rmk-ramified-weight-function} we write
$$\mathbb{M}(\pi,\H)=J_1\sum_{\Pi\in\mathcal{F}}\frac{L(1/2,\Pi\otimes\pi_1)L(1/2,\tilde{\Pi}\otimes\tilde{\pi}_2)}{L(1,\Pi,\Ad)}\mathcal{H}(\Pi)
H_\q(\Pi_\q)H_{\p_0}(\Pi_{\p_0})H_{\p_1}(\Pi_{\p_1})h_\infty(\Pi_\infty).$$
for some $J_1\neq 0$ depending only on $F$.
Hence, the right-hand side of \eqref{spectral-side} is entire in $s_1,s_2$, and consequently, in this case $\R(s_1,s_2)$ vanishes identically. Thus applying Proposition \ref{main-prop-degenerate}, Proposition \ref{rapid-decay-xi-perp-type-2}, and Lemma \ref{lem:main-term-type-2} we rewrite \eqref{main-junction-equation} as
\begin{equation}\label{main-formula-II}
\sum_{\Pi\in\mathcal{F}}\frac{L(1/2,\Pi\otimes\pi_1)L(1/2,\tilde{\Pi}\otimes\tilde{\pi}_2)}{L(1,\Pi,\Ad)}\mathcal{H}(\Pi)
=J_{\mathfrak{X}}\,\vol(\overline{K_0}(\q))^{-1} X^n + O_A\left((N(\q)X)^{-A}\right)
\end{equation}
for some $J_{\mathfrak{X}}\asymp 1$.

We now show that we can disregard the contribution of the representations which have slightly lower conductor. 

\begin{lem}\label{spectral-side-truncation}
For every $\eta>0$ there is a $\delta>0$ such that
$$\sum\frac{L(1/2,\Pi\otimes\pi_1)L(1/2,\tilde{\Pi}\otimes\tilde{\pi}_2)}{L(1,\Pi,\Ad)}
\mathcal{H}(\Pi)
\ll (N(\q)X)^{n-\delta},
$$
where the sum is over all $\Pi\in\mathcal{F}$ with $C(\Pi)\le (N(\q)X)^{1-\eta}$.
\end{lem}

\begin{proof}
First, Lemma \ref{supercuspidal-projects-cuspidal} and Lemma \ref{lower-bounds-spectral-nonarchimedean} imply that
$$H_{\p_0}(\Pi_{\p_0})=H_{\p_0}(\Pi_{\p_0};f_{\p_0})\ll 1 \ll H_{\p_0}(\Pi_{\p_0};f_0)$$
where $f_0:=\vol(\overline{K_0}(\p_0^{c(\sigma)}))^{-1}\1_{\overline{K_0}(\p_0^{c(\sigma))}}$. Similarly, from Lemma \ref{aux-weight-majorize} and Lemma \ref{lower-bounds-spectral-nonarchimedean} we obtain
$$H_{\p_1}(\Pi_{\p_1})=H_{\p_1}(\Pi_{\p_1};f_{\p_1})\ll 1 \ll H_{\p_1}(\Pi_{\p_1};f_1)$$
where $f_1:=\vol(\overline{K_0}(\p_1^\nu))^{-1}\1_{\overline{K_0}(\p_1^\nu)}$.

For an integral ideal $\mathfrak{e}$ of $F$, we define
$$\mathcal{F}(\mathfrak{e}):=\{\Pi\in \mathcal{F} : \mathfrak{C}(\Pi_\f)\mid\mathfrak{e}\}.$$
Thus using the two estimates above, along with Lemma \ref{bound-spectral-archimedean-local-weight} and Lemma \ref{upper-bounds-spectral-nonarchimedean} we bound the sum in consideration by
\begin{equation*}
    \ll_\epsilon (N(\q)X)^\epsilon\sum_{\substack{\Pi\in\mathcal{F}(\q\p_0^{c(\sigma)}\p_1^\nu)\\C(\Pi)\le (N(\q)X)^{1-\eta}}} \left|\frac{L(1/2,\Pi\otimes\pi_1)L(1/2,\tilde{\Pi}\otimes\tilde{\pi}_2)}{L(1,\Pi,\Ad)}\right|.
\end{equation*}
Now recalling non-negativity of $L(1,\Pi,\Ad)$ and applying Cauchy--Schwarz we majorize the above by
$$\le (N(\q)X)^\epsilon\prod_{i=1}^2\left(\sum_{\substack{\Pi\in\mathcal{F}(\q\p_0^{c(\sigma)}\p_1^\nu)\\C(\Pi)\le (N(\q)X)^{1-\eta}}}\frac{|L(1/2,\Pi\otimes\pi_i)|^2}{L(1,\Pi,\Ad)}\right)^{1/2}.$$
Thus it suffices to show that the inner-most sums are $O\left((N(\q)X)^{n-\delta}\right)$.

We diadically decompose the sum above and estimate it by
$$\ll_\epsilon (N(\q)X)^\epsilon\sup_{\substack{\mathfrak{e}\mid\q\p_0^{c(\sigma)}\p_1^\nu;\, Y_v\in \Rr_{\ge 1}, v\mid\infty\\N(\mathfrak{e})\prod_{v\mid\infty} Y_v\le (N(\q)X)^{1-\eta}}}\sum_{\substack{\Pi\in\mathcal{F}(\mathfrak{e})\\C(\Pi_v)\le Y_v,\, v\mid\infty}}\frac{|L(1/2,\Pi\times\pi_i)|^2}{L(1,\Pi,\Ad)}$$
Using Theorem \ref{hybrid-moment}, this is
$$\ll_\epsilon (N(\q)X)^\epsilon\sup_{\substack{\mathfrak{e}\mid\q\p_0^{c(\sigma)}\p_1^\nu;\, Y_v\in \Rr_{\ge 1}, v\mid\infty\\N(\mathfrak{e})\prod_{v\mid\infty} Y_v\le (N(\q)X)^{1-\eta}}}\left(N(\mathfrak{e})\prod_{v\mid\infty}Y_v\right)^{n+\epsilon}\ll (N(\q)X)^{\epsilon+(1-\eta)(n+\epsilon)}.$$
We conclude the proof by taking $\epsilon$ sufficiently small.
\end{proof}

\begin{proof}[Proof of Theorem \ref{mixed-moment-2}]
    It immediately follows from \eqref{main-formula-II} and Lemma \ref{spectral-side-truncation}, and noting that $\vol(\overline{K_0}(\q))^{-1}=\frac{\zeta_\q(1)}{\zeta_\q(n+1)}N(\q)^n$.
\end{proof}

\begin{proof}[Proof of Corollary \ref{non-vanishing}]
    From Theorem \ref{mixed-moment-2} We see that
    $$\sum_{\substack{\Pi\text{ cuspidal}\\C(\Pi)>(N(\q)X)^{1-\eta}}}\frac{L(1/2,\Pi\otimes\pi_1)L(1/2,\tilde{\Pi}\otimes\tilde{\pi}_2)}{L(1,\Ad,\Pi)}\H(\Pi)\neq 0$$
    for sufficiently large $N(\q)X$. Thus for all sufficiently large $Y>0$ there exists a cuspidal $\Pi$ with $C(\Pi)>Y$ such that $L(1/2,\Pi\otimes\pi_1)L(1/2,\tilde{\Pi}\otimes\tilde{\pi}_2)\neq 0$. Thus we conclude.
\end{proof}

\bibliographystyle{abbrv}
\bibliography{references}

\end{document}